\documentclass[a4paper]{article}

\usepackage[english]{babel}
\usepackage[utf8x]{inputenc}
\usepackage[T1]{fontenc}

\usepackage[a4paper,top=3cm,bottom=2cm,left=2cm,right=5cm,marginparwidth=4cm]{geometry}

\usepackage{amsmath}
\usepackage{amsthm}
\usepackage{amssymb}
\usepackage{graphicx}
\usepackage[colorinlistoftodos]{todonotes}
\usepackage[colorlinks=true, allcolors=blue]{hyperref}
\usepackage{xstring}
\usepackage{forloop}
\usepackage{bm}
\usepackage{natbib}

\newcounter{idx}

\newcommand{\ud}[1]{%
  \StrLen{#1}[\slen]
  \forloop[-1]{idx}{\slen}{\value{idx}>0}{%
    \StrChar{#1}{\value{idx}}[\crtLetter]%
    \IfSubStr{tlQWERTZUIOPLKJHGFDSAYXCVBNM}{\crtLetter}
      {\raisebox{\depth}{\rotatebox{180}{\crtLetter}}}
      {\raisebox{1ex}{\rotatebox{180}{\crtLetter}}}}%
}

\newtheorem{theorem}{Theorem}
\newtheorem{definition}{Definition}
\newtheorem{lemma}{Lemma}
\newtheorem{corollary}{Corollary}
\newtheorem{proposition}{Proposition}
\newtheorem{observation}{Observation}

\DeclareRobustCommand{\VAN}[3]{#2} 

\title{Heading in the right direction?\\ {\Large Using head moves to traverse phylogenetic network space}}

\author{Remie Janssen\footnote{Delft Institute of Applied Mathematics, Delft University of Technology, Van Mourik Broekmanweg 6,
2628 XE Delft,
The Netherlands. R.Janssen-2@tudelft.nl. Research funded by the Netherlands Organization for Scientific Research (NWO), Vidi grant 639.072.602 of dr. Leo van Iersel.} \\ \href{mailto:remiejanssen@gmail.com}{remiejanssen@gmail.com}  }
\begin{document}
\maketitle

\begin{abstract}
Head moves are a type of rearrangement moves for phylogenetic networks. They have mostly been studied as part of more encompassing types of moves, such as rSPR moves. Here, we study head moves as a type of moves on themselves. We show that the tiers ($k>0$) of phylogenetic network space are connected by \emph{local} head moves. Then we show tail moves and head moves are closely related: sequences of tail moves can be converted to sequences of head moves and vice versa, changing the length by at most a constant factor. Because the tiers of network space are connected by rSPR moves, this gives a second proof of the connectivity of these tiers. Furthermore, we show that these tiers have small diameter by reproving the connectivity a third time. As the head move neighbourhood is in general small, this makes head moves a good candidate for local search heuristics. Finally we prove that finding the shortest sequence of head moves between two networks is NP-hard. 
\end{abstract}

\section{Introduction}
For biologists it is vital to know the evolutionary history of the species they study. Evolutionary histories are among other things needed to find the reservoir/initial infection for some disease \citep[e.g.,][]{gao1999origin,lessler2016assessing}, or to learn about the evolution of genes, giving us insight in how they work \cite[e.g.,][]{shindyalov1994can,yu2015rasp,joy2016ancestral,guyeux2017ability,atas2018phylogenetic}.\\

\begin{figure}[h!]
\begin{center}
\includegraphics[scale=0.7]{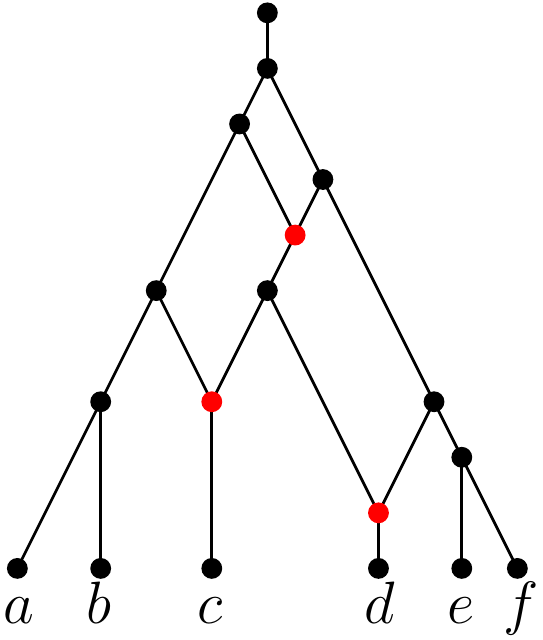}
\end{center}
\caption{A (tier-3) phylogenetic network with six leaves (i.e. taxa) at the bottom, and the root (ancestral taxon) at the top. Arcs are directed downwards, showing the passing of time. The red nodes are the three reticulations (i.e. reticulate evolutionary events).}
\label{fig:NetworkExample}
\end{figure}

These histories are traditionally represented as phylogenetic trees. This focus on trees has recently started shifting to phylogenetic networks, in which more biological processes can be represented. These biological processes, such as hybridization and horizontal gene transfer, are collectively known as reticulate evolutionary events, because they cause a reticulate (latin:~reticulatus~=~net-like) structure in the representation of the history.\\
\\

The extra structure of phylogenetic networks makes them harder to reconstruct than their tree counterparts. For certain models, it is still possible to quickly construct networks from data \cite[e.g.,][]{van2018polynomial}. 
Reconstruction is harder in most other cases, depending on the kind of data and the model of reconstruction. For some models the reconstruction of phylogenetic trees might even be hard already.
For such models, some kind of local search is often employed \citep{felsenstein2004inferring,lakner2008efficiency,nguyen2014iq,lessler2016assessing}. This is a process where the goal is to find a (close to) optimal tree by exploring the space of trees making only small changes. These small changes are called rearrangement moves.

Several rearrangement moves have long been studied for phylogenetic trees, the most prominent ones are Nearest Neighbour Interchange (NNI), Subtree Prune and Regraft (SPR), and Tree Bisection and Reconnection (TBR) \citep{felsenstein2004inferring,semple2003phylogenetics}.
Based on these moves for trees, the last decade has seen a surge in research on rearrangement moves for phylogenetic networks. There are several ways of generalizing the moves to networks, which means there is a relatively large number of moves for networks, including rSPR moves \citep{gambette2017rearrangement}, rNNI moves \citep{gambette2017rearrangement}, SNPR moves \citep{bordewich2017lost}, tail moves \citep{janssen2017exploring}, and head moves.

\begin{figure}[h!]
\begin{center}
\includegraphics[scale=0.7]{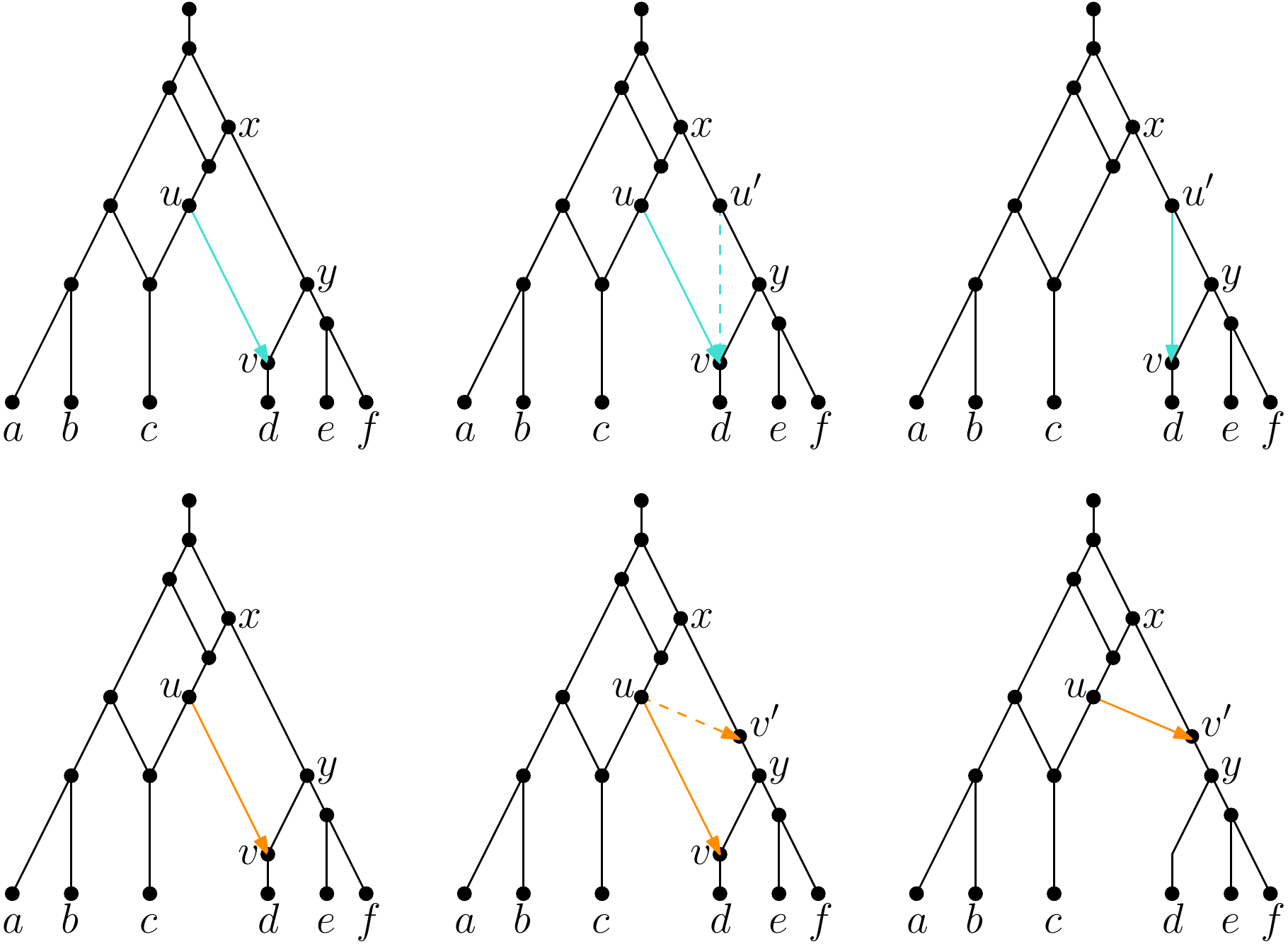}
\end{center}
\caption{Top: the tail move $(u,v)$ to $(x,y)$; Bottom: the head move $(u,v)$ to $(x,y)$. On the left, the starting networks in which the moving edges are coloured. The right networks are the resulting networks after the moves, with the moved edge coloured differently. The middle graph is a combination of the left and the right network, with the moving edge coloured differently. The solid coloured edge is the moving edge of the network before the move, the dashed coloured edge is the moving edge of the network after the move. We distinguish the moves with edge colours: blue is a tail move, orange is a head move.}
\label{fig:MoveExample}
\end{figure}

All these moves are quite similar in that they only change the location of one edge in the network. This means a lot of properties of the search spaces defined by different moves can be related. In this paper we study relations of such properties for the spaces corresponding to tail and head moves. The results we obtain can also be used in the study of rSPR moves, because rSPR moves consist of tail moves and head moves; and it can be used for the study of SNPR moves for the same reason.\\
\\
We start by proving that each tier of phylogenetic network space is connected by head moves (Section~\ref{sec:connectivity}). We push this result a bit further in Section~\ref{sec:LocalConnectivity}, where we prove that the result still holds for distance-2 head moves, but not for distance-1 head moves. After this, we take a slightly different angle and we prove results about distances: in Section~\ref{sec:HeadToTailAndBack} we prove that each head move can be replaced by at most $16$ tail moves, and each tail move can be replaced by at most $15$ head moves. This not only reproves connectivity of tiers of head move space, but also gives relations for distances between two networks measured by different rearrangement moves (rSPR, head moves, and tail moves). In Section~\ref{sec:HeadDistance}, we prove the upper bound $6n+6k-1$ for the diameter of tier-$k$ of network space with $n$ taxa. Lastly, in Section~\ref{sec:NP-Hard}, we prove that computing the head moves distance between two networks is NP-hard. It remains open whether this computation is hard in a fixed tier.

\section{Preliminaries}\label{sec:Prelim}
\subsection{Phylogenetic networks}

\begin{definition}\label{def:networkTree}
A \emph{binary phylogenetic network} with leaves $X$ is a directed acyclic graph (DAG) $N$ with:
\begin{itemize}
\item one \emph{root} (indegree-0, outdegree-1 node) 
\item $|X|$ \emph{leaves} (indegree-1, outdegree-0 nodes) bijectively labelled with the elements of $X$
\item all other nodes are either \emph{split nodes} (or splits; indegree-1, outegree-2 nodes), or \emph{reticulations} (indegree-2, outdegree-1 nodes).
\end{itemize}
Incoming edges of reticulation nodes are called \emph{reticulation edges}, and incoming edges of split nodes are called \emph{tree edges}. The \emph{reticulation number} of $N$ is the number of reticulation nodes. The set of all networks with reticulation number $k$ is called the $k$-th tier of phylogenetic network space.
\end{definition}

For simplicity, we will often refer to binary phylogenetic networks as phylogenetic networks or as networks. Phylogenetic networks are a generalization of phylogenetic trees. These trees can be defined as phylogenetic networks without any reticulation nodes. Many phylogenetic problems start with a set of phylogenetic trees and ask for a `good' network for this set of trees. This often means that the trees must be contained in this network in the following sense.

\begin{definition}
A tree $T$ can be \emph{embedded} in a network $N$ if there exists an $X$ labelled subgraph of $N$ that is a subdivision of $T$. We say $T$ is an embedded tree of $N$.
\end{definition}

Because a network is a DAG, there are unambiguous ancestry relations between the nodes. We imagine a network as having the root at the top, and the leaves at the bottom. This induces the following terminology.

\begin{definition}
Let $u,v$ be nodes in a network $N$. Then we say:
\begin{itemize}
\item $u$ is \emph{above} $v$, and $v$ is \emph{below} $u$, if there is a directed path from $u$ to $v$ in $N$.
\item $u$ is \emph{directly above} $v$, and $v$ \emph{directly below} $u$, if there is an edge $(u,v)$ in $N$.
\end{itemize}
Similarly, we say an edge $(u,v)$ is above a node $w$ or above an edge $(w,z)$ if $v$ is above $w$. An edge $(u,v)$ is below a node $z$ or an edge $(w,z)$ if $u$ is below $z$.

We also use terminology inspired by biology: if $u$ is above $v$, we say that $u$ is an \emph{ancestor} of $v$, and if $u$ is directly above $v$, we say that $u$ is a \emph{parent} of $v$ and $v$ a \emph{child} of $u$. 

A \emph{Lowest Common Ancestor (LCA)} of two nodes $u$ and $v$ is a node $z$ which is above both $u$ and $v$, such that there are no other such nodes below $z$.
\end{definition}

An important substructure in a network is the triangle. This structure has a big impact on the rearrangements we can do in a network.

\begin{definition}
Let $N$ be a network and $t,s,r$ nodes of $N$. If there are edges $(t,s)$, $(t,r)$ and $(s,r)$ in $N$, then we say $t$, $s$ and $r$ form a \emph{triangle} in $N$. We call $t$ the \emph{top} of the triangle, $s$ the \emph{side} of the triangle, and the reticulation $r$ the \emph{bottom} of the triangle. 
\end{definition}

The following lemma describes one property of triangles that we will use very often.

\begin{lemma}
Let $N$ be a network with a triangle $t,s,r$, where $s$ is a split node. Then reversing the direction of the edge $(s,r)$ gives a new network $N'$. We say the direction of the triangle is reversed.
\end{lemma}
\begin{proof}
Note that there is no edge $(r,s)$ in $N$, as this would induce a cycle. Hence removing $(s,r)$ and adding $(r,s)$ produces a new directed graph (with no multi-edges). The only nodes of which the degree changes are $s$ and $r$: $s$ changes from a split node to a reticulation, and $r$ changes from a reticulation to a split node. Hence, we still have the right types and numbers of nodes (one root, $|X|$ leaves bijectively labelled with $X$). Therefore, $N'$ is a phylogenetic network, with $t,r,s$ the reversed triangle. 
\end{proof}

\subsection{Rearrangement moves}
The main topics of this paper are head and tail moves, which are types of rearrangement moves on phylogenetic networks. Several types of moves have been defined for rooted phylogenetic networks. The most notable ones are tail moves \citep{janssen2017exploring}, rooted Subtree Prune and Regraft (rSPR) and rooted Nearest Neighbour Interchange (rNNI) moves \citep{gambette2017rearrangement} and SubNet Prune and Regraft (SNPR) moves \citep{bordewich2017lost}.
These moves typically change the location of one or both ends of an edge, or they remove or add an edge.

Although head and tail moves are defined in detail in \cite{janssen2017exploring}, we repeat the definitions here for the convenience of the reader.

\begin{definition}[Head move]\label{def:HeadMove}
Let $e=(u,v)$ and $f$ be edges of a network. A head move of $e$ to $f$ consists of the following steps (Figure~\ref{fig:HeadMove}):
\begin{enumerate}
\item delete  $e$;
\item \emph{subdivide} $f$ with a new node $v'$;
\item \emph{suppress} the indegree-1 outdegree-1 node $v$;
\item add the edge $(u,v')$.
\end{enumerate}
\emph{Subdividing} an edge $(u,v)$ consists of deleting it and adding a node $x$ and edges $(u,x)$ and $(x,v)$. \emph{Suppressing} an indegree-1, outdegree-1 node $x$ with parent $u$ and child $v$ consists of removing  edges $(u,x)$ and $(x,v)$ and node $x$, and then adding an edge $(u,v)$.

Head moves are only allowed if the resulting digraph is still a network (Definition~\ref{def:networkTree}). 
We say that a head move is a \emph{distance-$d$ head move} if, after step 2, a shortest path from $v$ to $v'$ in the underlying undirected graph has length at most $d+1$ (number of edges in the path).
\end{definition}

\begin{figure}[h!]
\begin{center}
\includegraphics[scale=0.6]{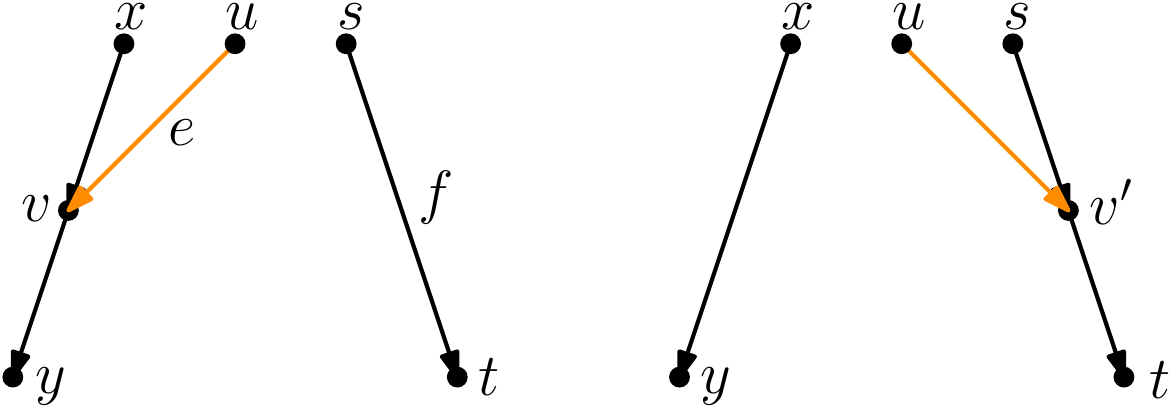}
\end{center}
\caption{The head move $e=(u,v)$ from $(x,y)$ to $f=(s,t)$ described in Definition~\ref{def:HeadMove}.}
\label{fig:HeadMove}
\end{figure}

\begin{definition}[Tail move]\label{def:TailMove}
Let $e=(u,v)$ and $f$ be edges of a network. A tail move of $e$ to $f$ consists of the following steps (Figure~\ref{fig:TailMove}):
\begin{enumerate}
\item delete $e$;
\item subdivide $f$ with a new node $u'$;
\item suppress the indegree-1 outdegree-1 node $u$;
\item add the edge $(u',v)$.
\end{enumerate}
Tail moves are only allowed if the resulting digraph is still a network (Definition~\ref{def:networkTree}).
We say that a tail move is a \emph{distance-$d$ tail move} if, after step 2, a shortest path from $u$ to $u'$ in the underlying undirected graph has length at most $d+1$.
\end{definition}

\begin{figure}[h!]
\begin{center}
\includegraphics[scale=0.6]{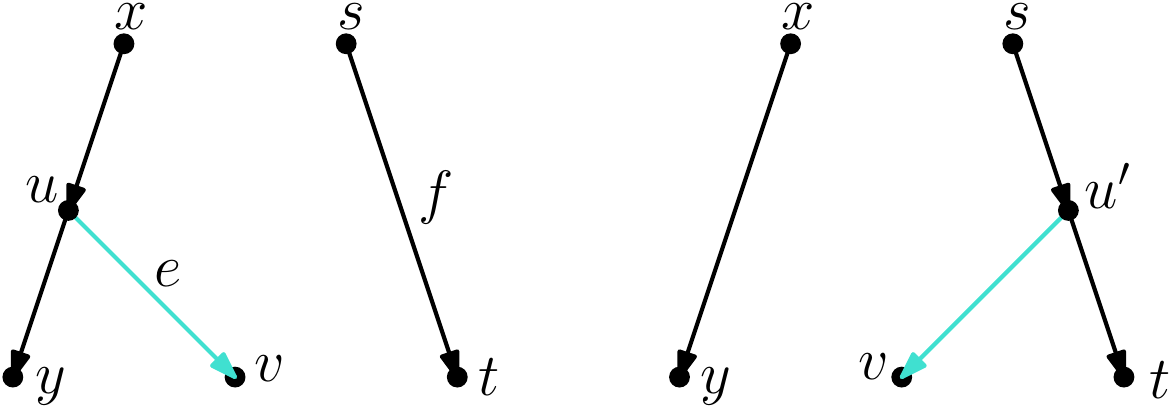}
\end{center}
\caption{The tail move $e=(u,v)$ from $(x,y)$ to $f=(s,t)$ described in Definition~\ref{def:TailMove}.}
\label{fig:TailMove}
\end{figure}


The previously mentioned rearrangement moves can almost all be described by combinations of head and tail moves. An rSPR move is either a head or a tail move; an rNNI move is a distance-1 head or tail move. Only SNPR moves cannot be fully described this way. 

An SNPR move is either a tail move, or a move that changes the reticulation number by adding (SNPR+) or deleting (SNPR-) a reticulation edge. These moves are called \emph{vertical} moves because adding or deleting a reticulation edge makes you go up or down a tier in phylogenetic network space. Moves such as tail and head moves that stay within a tier (not changing the reticulation number) are called \emph{horizontal} moves.

\subsubsection{Validity of moves}
As we want to use rearrangement moves to traverse network space, we are only interested in moves that result in a phylogenetic network. The definitions in the previous subsection ensure this always happens for tail and head moves. In this paper we often propose a sequence of moves by stating: move the tail of edge $e$ to edge $f$, then move the head of edge $e'$ to $f'$ and so forth. We then check whether these moves are \emph{valid} (also called \emph{allowed}), that is, whether applying the steps in the definitions of the previous subsection  produces a phylogenetic network.

We first introduce the concept of movable edges. A necessary condition for a rearrangement move to be valid, is that the moving edge is movable. This concept is useful because it is a property of only the moving edge, and can hence be easily checked. For different types of moves, we have different definitions of movability.

\begin{definition}
Let $(u,v)$ be an edge in a network $N$, then $(u,v)$ is tail movable if $u$ is a split node with parent $p$ and other child $c$, and there is no edge $(p,c)$ in $N$.
\end{definition}

This is equivalent to saying that an edge with tail $u$ is tail movable if $u$ is a split node and $u$ is not the side of a triangle. There is a similar definition for head moves.

\begin{definition}
Let $(u,v)$ be an edge in a network $N$, then $(u,v)$ is head movable if $v$ is a reticulation node with other parent $p$ and child $c$, and there is no edge $(p,c)$ in $N$.
\end{definition}

When the type of move is clear from context, we will use the term \emph{movable}. Using the concept of movability, we can now quickly give sufficient conditions for a move to be valid. Movability essentially ensures that `detaching' the edge does not create parallel edges. We need some extra conditions to make sure that reattaching the edge does not create parallel edges, and that the resulting network has no cycles. These are exactly the second and third conditions from the following lemma.

\begin{lemma}
A tail move $(u,v)$ to $(s,t)$ is valid if all of the following hold:
\begin{itemize}
\item $(u,v)$ is tail movable;
\item $v\neq t$.
\item $v$ is not above $s$;
\end{itemize}
\end{lemma}
\begin{proof}
Because $(u,v)$ is tail movable, the removal of $(u,v)$ and subsequent suppression of $u$ does not create parallel edges. Because $v\neq t$, subdividing $(s,t)$ with a node $u'$ and adding the edge $(u',v)$ does not create parallel edges either. This means that the result $N'$ of the head move is a directed graph. 

Now suppose this graph has a cycle. As each path that does not use $(u',v)$ corresponds to a path in $N$, the cycle must use $(u',v)$. This means that there is a path from $v$ to $u'$ in $N'$. Because $u'$ is a split node with parent $s$, there must also be a path from $v$ to $s$ in $N'$. This implies there was also a path from $v$ to $s$ in $N$, but this contradicts the third condition: $v$ is not above $s$. We conclude that $N'$ is a DAG.

Because all labelled nodes are not changed by the tail move, $N'$ is a phylogenetic network, and the tail move is valid.
\end{proof}

The proof of the corresponding lemma for head moves is completely analogous.

\begin{lemma}\label{lem:validHead}
A head move $(u,v)$ to $(s,t)$ is valid if all of the following hold:
\begin{itemize}
\item $(u,v)$ is head movable;
\item $u\neq s$.
\item $t$ is not above $u$;
\end{itemize}
\end{lemma}

We will very frequently use the following corollary of this lemma, which makes it very easy to check whether some moves are valid.

\begin{corollary}\label{cor:MoveUpDown}
Let $(u,v)$ be a tail movable edge, then moving the tail of $(u,v)$ to an edge above $u$ is allowed. We also say that moving the tail of $(u,v)$ up is allowed. Similarly, moving the head of a head movable edge down is allowed.
\end{corollary}

\begin{lemma}\label{lem:HeadMoveEmbeddedTrees}
Let $N$ be a network and $\mathcal{T}$ its set of embedded trees, and let $N'$ with embedded trees $\mathcal{T}'$ be the result of one head move in $N$. Then there is a tree $T\in\mathcal{T}$ which is embedded in $N'$; furthermore, for each $T'\in\mathcal{T}'$ there is a tree $T\in\mathcal{T}$ at most one tail move away from $T'$.
\end{lemma}
\begin{proof}
Let $(u,v)$ be the edge that is moved in the head move from $N$ to $N'$. Then $v$ is a reticulation and it has another incoming edge $(w,v)$. There is an embedded tree $T$ of $N$ that uses this edge, and therefore does not use $(u,v)$. This means that changing the location of $(u,v)$ does not change the fact that $T$ is embedded in the network.

For the second part: first suppose the embedding of $T'$ in $N'$ does not use the new edge $(u,v')$. Then clearly $T'$ can be embedded in $N$ without the edge $(u,v)$. This means it can also be embedded in $N$. 

Now suppose the embedding of $T'$ in $N'$ uses the new edge $(u,v')$. Now consider the tree obtained by taking the embedding of $T'$, removing $(u,v')$ and adding a path of reticulation edges leading from a node $w$ in the embedding of $T'$ to $v'$ via the other incoming edge of $v'$. This tree only uses edges that are also in $N$, hence this tree is embedded in $N$. It is easy to see that this tree is at most one tail move away from $T'$: the one that moves the subtree below $v'$ to $w$.
\end{proof}

\subsection{Phylogenetic network spaces}
Considering all phylogenetic networks with the same leaf set and the same number of reticulations as a set can be interesting. For example, we might want to know the size of such a set for a particular number of leaves and reticulations \citep[e.g.][]{francis2017bounds}.

Things become even more interesting when we also take into account which networks are similar to each other, that is, if we look at phylogenetic network spaces. One way to define such spaces is as graphs, where each network is represented by a node, and there is an edge between two networks $N,N'$ if there exists a rearrangement move changing $N$ into $N'$. These spaces are important as search spaces for local search heuristics when looking at complicated phylogenetic networks problems. Hence, several properties of these phylogenetic network spaces have been studied.\\
\\
The most basic of these properties is connectivity. The introduction of each network rearrangement move was followed by a proof of connectivity of the corresponding spaces \citep[rSPR and rNNI moves:][]{gambette2017rearrangement}
\citep[NNI, SPR and TBR moves:][]{huber2016transforming}\citep[SNPR moves:][]{bordewich2017lost}\citep[tail moves:][]{janssen2017exploring}. 

Note that the spaces that take the shape of a graph come with a metric: the distance between two nodes (networks in our case). So when it is known that any network can be reached from any other network, a natural follow up is to ask about the distance between a pair of networks. 

\begin{definition}
Let $N$ and $N'$ be phylogenetic networks with the same leaf set in the same tier. We denote by $d_{M}(N,N')$ the \emph{distance} between phylogenetic networks $N$ and $N'$ using rearrangement moves of type $M$. That is, $d_{M}(N,N')$ is the minimum number of $M$-moves needed to change $N$ into $N'$.
\end{definition}

For phylogenetic trees, the distance between two trees is nicely characterized by a concept known as agreement forests. For phylogenetic networks, however, not much is known about such distances for a given pair of networks. Recently, \citet{klawitter2018distance} introduced an agreement forest analogue for networks, which bounds such distances but does not give the exact distance. 

The only other known bounds relate to the diameters, that is the maximal distance between any pair of points (i.e., networks) in a space (tier of phylogenetic network space). 

\begin{definition}
Let $k\in\mathbb{Z}_{\geq 0}$ be the number of reticulations, $n\in\mathbb{Z}_{\geq 2}$ be the number of leaves and $M$ a type of rearrangement move. We denote with $\Delta_{k}^{M}(n)$ the \emph{diameter} of tier-$k$ of phylogenetic network space with $n$ leaves using moves of type $M$:
\[\Delta_{k}^{M}(n)=\max_{N,N'}d_{M}(N,N'),\]
where $N,N'$ are tier-$k$ networks with $n$ leaves.
\end{definition}

Compared to pairwise distances, much more is known about the diameters of phylogenetic networks spaces. For all the mentioned types of moves some bounds on the diameters are known \cite{huber2016transforming,janssen2017exploring}. SNPR moves are an exception here, as these include vertical moves: each vertical move changes the tier by at most one. Hence, the diameter of SNPR move space is unbounded \citep{bordewich2017lost}. Of course for SNPR moves, the question can be rephrased to make it more informative, for example: allowing for movement through all tiers, what is the maximal distance between two networks in the same tier?

The last property we discuss is the neighbourhood size of a phylogenetic network, i.e. the number of networks that can be reached using one rearrangement move. The size of the neighbourhood is important for local search heuristics, as it gives the number of networks that need to be considered at each step. For networks, the only rearrangement move neighbourhood that has been studied is that of the SNPR move \citep{klawitter2017snpr}.

\section{connectivity}\label{sec:connectivity}
\subsection{Tier-1}
In this subsection we show that any two tier-1 networks with the same leaf set are connected by a sequence of head moves. This will form the basis for a more general proof of connectivity of any tier using local (distance-2) head moves in Section~\ref{sec:LocalConnectivity}. Our approach uses the fact that a tier-1 network is highly tree-like. Informally, we will look for a sequence of moves that changes the underlying tree-structure of the one network into the tree structure of the other network. Then we will place the reticulation edge where it needs to be.  

The following lemmas make this precise: Lemma~\ref{lem:MoveTriangle} proves that we can place the reticulation anywhere without changing the tree structure, Lemma~\ref{lem:ChangeBaseTree} proves that we can change the tree structure.

\begin{lemma}\label{lem:TierOneTriangleEmbeddings}
Let $M$ and $M'$ be tier-1 networks with a triangle and suppose there is a head move in $M$ resulting in $M'$, then both $M$ and $M'$ have the same embedded tree.
\end{lemma}
\begin{proof}
Each tier-1 network has at most two embedded trees, and if a tier-1 network has a triangle, there is exactly one tree that can be embedded in it. Let $T$ be the embedded tree of $M$. For each head move one of the embedded trees stays the same (Lemma~\ref{lem:HeadMoveEmbeddedTrees}). This means the head move from $M$ to $M'$ preserves the embedded tree $T$, and $T$ is an embedded tree of $M'$. Because $M'$ has a triangle, $M'$ has exactly one embedded tree. Therefore $T$ is the only embedded tree of both $M$ and $M'$.
\end{proof}

\begin{lemma}\label{lem:CreateTriangle}
Let $N$ be tier-1 networks with an embedded tree $T$. Then, there exists a network $M$ such that $M$ has a triangle, $M$ has embedded tree $T$, and $M$ can be obtained from $N$ with one head move.
\end{lemma}
\begin{proof}
Let $(u,v)$ be the reticulation edge not used by the embedding of $T$ in $N$. If the other child $c(u)$ of $u$ (i.e. not $v$) is a split node, do a head move of $(u,v)$ to any of the outgoing edges of $c(u)$, creating a triangle. Such a move is valid because any reticulation edge in a tier-1 network is movable; a node is never equal to its child node, so in particular $u\neq c(u)$; $c(u)$ is not above $u$, so neither are the children of $c(u)$.

If $c(u)$ is not a split node, it must be a leaf because the children $v$ and $c(u)$ of $u$ are distinct, there is only one reticulation in $N$, which is $v$. Furthermore, the parent $p(u)$ of $u$ must be a split node: $p(u)$ cannot be a reticulation, as $v$ is the only reticulation node and $v$ cannot be above $u$; $p(u)$ cannot be the root because then there would be no directed path from the root to the other parent of $v$ (i.e., not $u$). Therefore, in this case we can move the head of $(u,v)$ to the other child edge of $p(u)$ to create a network $M$ with a triangle. 

In both of these cases, a triangle is created using one head move in $N$. Furthermore, the head move uses the edge that is not used by the embedding of $T$ in $N$. Hence, the embedded tree of the resulting network is $T$. 
\end{proof}

\begin{lemma}\label{lem:MoveTriangle}
Let $N$ and $N'$ be tier-1 networks with a common embedded tree $T$, then there is a sequence of head moves from $N$ to $N'$.
\end{lemma}
\begin{proof}
By Lemma~\ref{lem:TierOneTriangleEmbeddings}, moving a triangle or reversing the direction of a triangle using one head move does not change the embedded trees. As shown in Figure~\ref{fig:TriangleMove} we can move a triangle up and down a split node using one head move. We can also change the direction of a triangle using one head move (Figure~\ref{fig:TriangleDirection}). This means that a triangle can be moved around a tier-1 network freely using head moves. 

Note that there exist networks $M$ and $M'$, which both have a triangle and embedded tree $T$, and which are one head move away from $N$ and $N'$ respectively (Lemma~\ref{lem:CreateTriangle}). Hence, there is a sequence of head moves from $M$ to $M'$, which simply moves the triangle into the right position. Therefore, there also exists a sequence of head moves from $N$ to $N'$ (via $M$ and $M'$).
\end{proof}

\begin{figure}[h!]
\begin{center}
\includegraphics[scale=0.5]{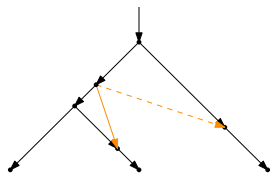}
\end{center}
\caption{The distance-3 head move used to move a triangle up or down one split node.}
\label{fig:TriangleMove}
\end{figure}

\begin{figure}[h!]
\begin{center}
\includegraphics[scale=0.5]{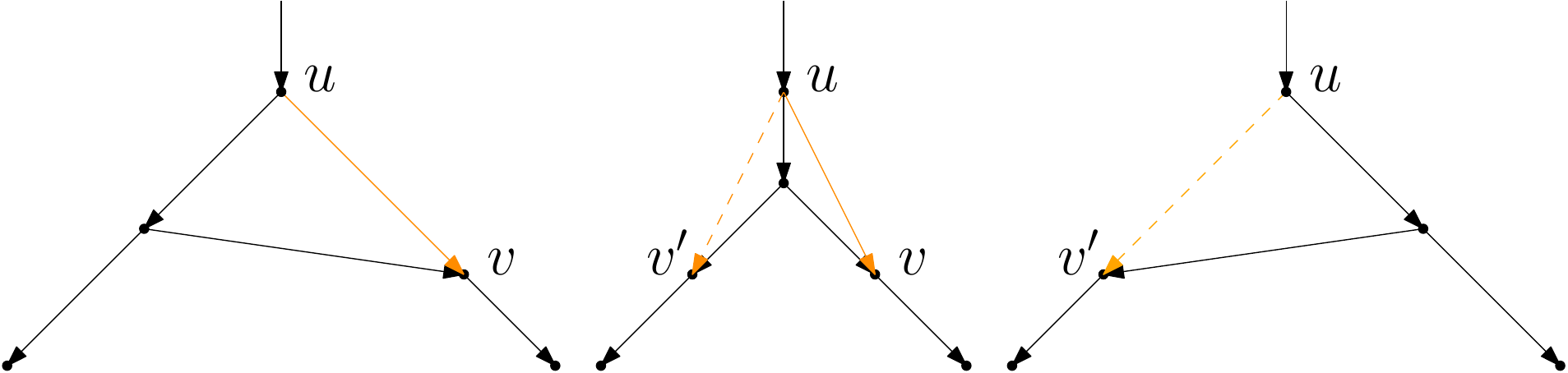}
\end{center}
\caption{The distance-1 head move used to change the direction of a triangle.}
\label{fig:TriangleDirection}
\end{figure}

\begin{lemma}\label{lem:ChangeBaseTree}
Let $N$ be a tier-1 network with embedded tree $T$, and let $T'$ be any tree on the same leaf set. Then there exists a sequence of head moves from $N$ to a network $N'$ with $T'$ as a embedded tree. 
\end{lemma}
\begin{proof}
It suffices to prove this for any $T'$ that is one SPR move removed from $T$, because the space of phylogenetic trees with the same leaf set is connected by SPR moves. Hence, let $(u,v)$ to $(x,y)$ be the SPR move that transforms $T$ into $T'$. 

By Lemma~\ref{lem:MoveTriangle}, there is a sequence of head moves transforming $N$ into a network $M$ with the following properties: 
the tree $T$ can be embedded in $M$; $M$ has a reticulation edge $(a,b)$ where $a$ lies on the image of $(x,y)$ in $M$, and the head $b$ lies on the image of the other outgoing edge $(x,z)$ of $x$ if $x$ is not the root and on the image of one of the child edges $(y,z')$ of $y$ otherwise. 

This creates a situation where there are edges $(x,a)$, $(a,b)$, $(p,b)$ $(a,y)$ and $(b,\zeta)$ with $p=x$ and $\zeta=z$ or $p=y$ and $\zeta=z'$ . The case $p=x$ is depicted in Figure~\ref{fig:TierOneHeadTail}.

\begin{figure}[h!]
\begin{center}
\includegraphics[scale=0.5]{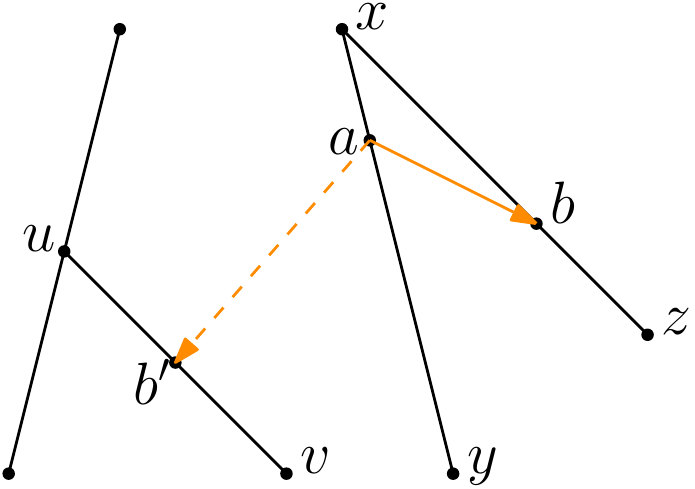}
\end{center}
\caption{The head move used to simulate the SPR move $(u,v)$ to $(x,y)$ on the embedded tree $T$ to $T'$. The triangle is already in the position described in the proof of Lemma~\ref{lem:ChangeBaseTree}. The starting network consists of the solid arrows, by doing the head move $(a,b)$ to $(u,v)$ (orange) we get the network consisting of the solid black and the dashed orange arcs. The embedded tree using the new reticulation arc corresponds to $T'$.}
\label{fig:TierOneHeadTail}
\end{figure}

Now do a head move of $(a,b)$ to the image of the $T$-edge $(u,v)$, this is allowed because any reticulation edge in a tier-1 network is movable; $b$ is not equal to the image of $u$ as $b$ is a reticulation node and the image of $u$ a split node; and the image of $v$ is not above $a$, as otherwise the tail move $(u,v)$ to $(x,y)$ could not be valid. Let $N'$ be the resulting network, and note that the embedded tree using the new reticulation edge is $T'$.
\end{proof}

\subsection{Hiding other reticulations}
Using the results about tier-one networks in the previous section, we will now prove connectivity of any tier. To do this, we will `hide' the other reticulations at the top of the network. This makes the network very treelike, except near the root. This means we concentrate all the complications in one place in the network, and we can handle all of it simultaneously.

\begin{definition}\label{def:ReticsAtTop}
A network has $k$ reticulations at the top if it has the following structure:
\begin{enumerate}
\item[1)] the node $c$: the child of the root;
\item[2)] nodes $a_i$ and $b_i$ and an edge $(a_i,b_i)$ for each $i\in\{1,\ldots,k\}$;
\item[3)] the edges $(c,a_1)$ and $(c,b_1)$;
\item[4)] for each $i\in\{1,\ldots,k-1\}$ there are edges $(a_i,a_{i+1})$ and $(b_i,b_{i+1})$ or edges $(a_i,b_{i+1})$ and $(b_i,a_{i+1})$.
\end{enumerate}
We say the there are $k$ reticulations neatly at the top if they are all directed to the same edge, i.e. we replace point $4)$ with
\begin{enumerate}
\item[4')] for each $i\in\{1,\ldots,k-1\}$ there are edges $(a_i,a_{i+1})$ and $(b_i,b_{i+1})$.
\end{enumerate}
Examples are shown in Figure~\ref{fig:ReticsTop}.
\end{definition}

\begin{figure}[h!]
\begin{center}
\includegraphics[scale=0.9]{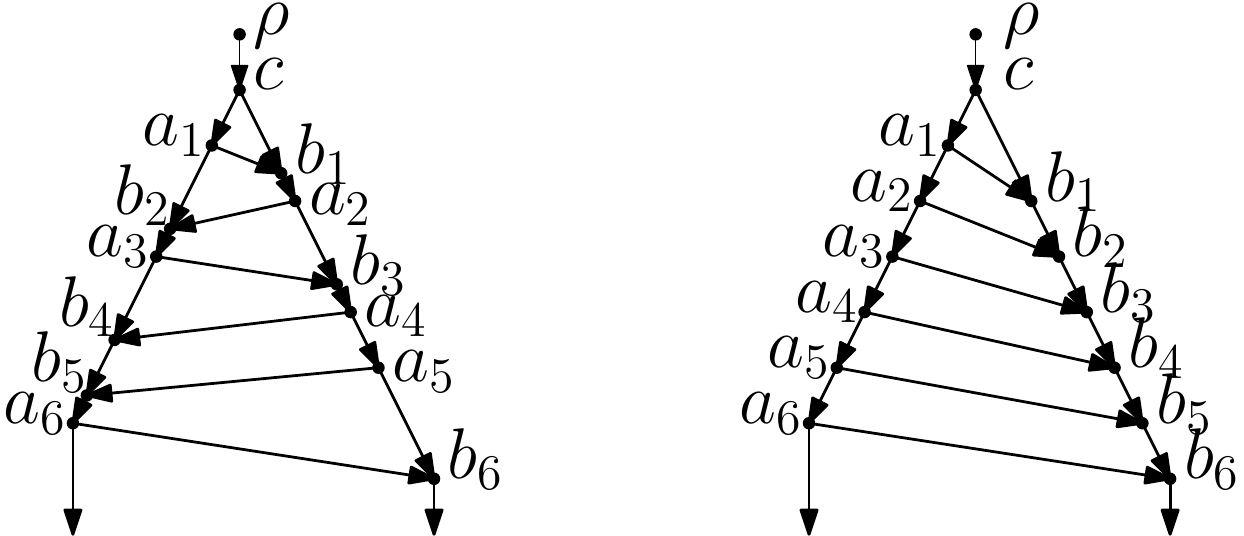}
\end{center}
\caption{Two networks with 6 reticulations at the top. In the right network, the reticulations are neatly at the top.}
\label{fig:ReticsTop}
\end{figure}

\begin{lemma}\label{lem:ReticsToTop}
Let $N$ be a tier-$k$ network. Then there is a sequence of head moves turning $N$ into a network with $k$ reticulations at the top.
\end{lemma}
\begin{proof}
Note that the network induces a partial order on the reticulation nodes. Suppose $N$ has $l<k$ reticulations at the top. Let $r$ be a highest reticulation node that is not yet at the top. One of the two corresponding reticulation edges is head movable. Let this be the edge $(s,r)$. 

If $s$ is a child of $a_l$ or $b_l$ (as in Definition~\ref{def:ReticsAtTop}, i.e. $s$ is directly below the top reticulations), then one head move suffices to get this reticulation to the top. Otherwise there is at least one node between $s$ and the top, let $t$ be the lowest such node, that means that $t$ is the parent of $s$. Because $r$ is a highest reticulation that is not at the top, $t$ is a split node and there are edges $(t,s)$ and $(t,q)$. Moving the head of $(s,r)$ to $(t,q)$ is a valid move that creates a triangle. 

Now we move this triangle to the top with head moves as in Lemma~\ref{lem:MoveTriangle}. This way we get the reticulation to the top. Doing this for all reticulations produces a network with all reticulations at the top.
\end{proof}

The following lemma ensures that the top reticulations can be directed neatly using head moves. The moves used in sequences to achieve this are much like the ones used in Lemma~\ref{lem:MoveTriangle} to change the direction of a triangle. Like for a triangle, we should define `changing the direction'.

\begin{definition}
Let $N$ be a network with $k$ reticulations at the top. 
\emph{Changing the direction} of an edge $(a_i,b_i)$ (as in Definition~\ref{def:ReticsAtTop}) consists of changing $N$ into a network $N'$ that is isomorphic to $N$ when $(a_i,b_i)$ is replaced by $(b_i,a_i)$. Note that labels $a_j$ and $b_j$ do not coincide between $N$ and $N'$. Changing the direction of a set of such edges at the same time is defined analogously.
\end{definition}

\begin{lemma}\label{lem:NeatlyOnTop}
Let $N$ be a network with $k$ reticulations at the top. Then the reticulations can be redirected so that they are neatly on top (directed to either edge) with at most $k$ head moves. The network below $a_k$ and $b_k$ (notation as in Definition~\ref{def:ReticsAtTop}) is not altered in this process.
\end{lemma}
\begin{proof}
We redirect the top reticulations starting with the lowest one. The move $(u_{i-1},b_i)$ to $(a_i,v_{i+1})$ with $u_{i-1}$ the parent of $b_i$ that is not $a_i$ and $v_{i+1}$ the child of $a_i$ that is not $b_i$ (Figure~\ref{fig:topDirectionChange}) changes the direction of the chosen edge $(a_i,b_i)$ and all the reticulation edges above; it leaves all other edges fixed as they were.
\end{proof}

\begin{figure}[h!]
\begin{center}
\includegraphics[scale=1.0]{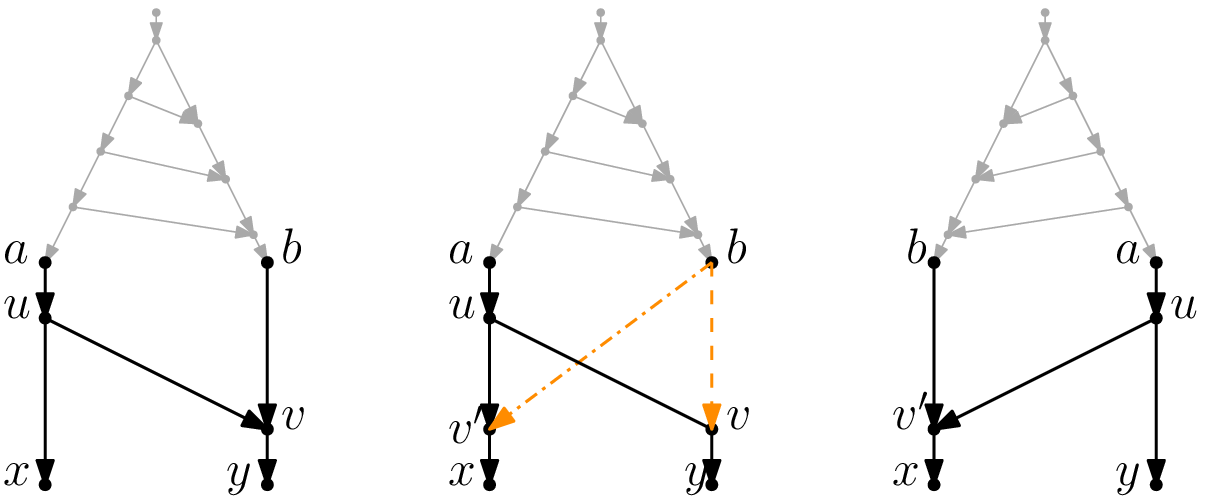}
\end{center}
\caption{The move used in Lemma~\ref{lem:NeatlyOnTop} to redirect the $i$ highest reticulations at the top: head move $(u_{i-1},b_i)$ to edge $(a_i,v_{i+1})$. This move changes the direction of all reticulations at the top that are higher than the moved edge. The part of the network below $a_i$ and $b_i$ does not change.}
\label{fig:topDirectionChange}
\end{figure}

Note that networks with all reticulations at the top are highly tree like. Like tier-1 networks with a triangle, they have only one embedded tree. Now, as in the case of tier-1 networks, we want to use one reticulation to change this embedded tree. To do this, we use the lowest reticulation arc $(a_k,b_k)$ to create a triangle that can move around the lower part of the network. 

\begin{definition}
Let $N$ be a network with $k$ reticulations at the top (notation as in Definition~\ref{def:ReticsAtTop}) and a split node $x$ directly below $a_k$. \emph{Moving a triangle from the top} consists of creating a triangle at $x$ by a head of $(a_k,b_k)$ to one of the outgoing edges $(x,c(x))$ of $x$. \emph{Moving a triangle to the top} is the reverse of this operation.
\end{definition}

\begin{observation}
A tier-$k$ network with $k-1$ reticulations at the top and the $k$th reticulation at the bottom of a triangle has exactly one embedded tree.
\end{observation}

\begin{lemma}\label{lem:GeneralMovingTriangle}
Let $N$ and $N'$ be tier-$k$ networks on the same leaf set with $k-1$ reticulations at the top and the $k$-th reticulation at the bottom of a triangle. Suppose $N$ and $N'$ have the same embedded trees, then there exists a sequence of head moves from $N$ to $N'$.
\end{lemma}
\begin{proof}
Note that the network consists of $k-1$ reticulations at the top, and two pendant subtrees (the same as the pendant subtrees as of the top split node in the embedded tree), one of which contains a triangle. Moving a triangle through one of these subtrees is the same as moving a triangle though a tier-1 network, so this is possible by Lemma~\ref{lem:MoveTriangle}. To be able to move the triangle anywhere, we need to be able to move it from the one pendant subtree to the other one. This can be done by moving the triangle to the top, and then moving it down on the other side after redirecting all the top reticulations. Note that none of these moves change the embedded tree: each of the intermediate networks has exactly one embedded tree, and doing a head move keeps at least one embedded tree. Hence, moving the triangle to the right place and then redirecting the top reticulations as needed gives a sequence from $N$ to $N'$.
\end{proof}

\begin{lemma}\label{lem:GeneralChangeBaseTree}
Let $N$ and $N'$ be tier $k>0$ networks on the same leaf set with all reticulations neatly at the top. Then there exists a sequence of head moves turning $N$ into $N'$. 
\end{lemma}
\begin{proof}
This works exactly as for tier-1 networks (Lemma~\ref{lem:ChangeBaseTree}): $N$ and $N'$ again both have exactly one embedded tree, and we aim to change this embedded tree. Like in Lemma~\ref{lem:ChangeBaseTree}, we assume that these embedded trees are one SPR move apart. In this case, we can move triangles anywhere below the $k-1$ reticulations at the top by Lemma~\ref{lem:GeneralMovingTriangle}. We use the lowest one of the top reticulation edges $(a_k,b_k)$ to create the moving triangle. The only case that takes special attention is where we do an SPR move to the root edge of the tree. The case where the SPR move places an edge at any other location is proved entirely analogous to Lemma~\ref{lem:ChangeBaseTree}.

Let $(u,v)$ to $(\rho,c)$ be such an SPR move to the root edge of the tree, and let $x$ and $y$ be the nodes directly below the top: $x$ on the side of the reticulations $b_i$, and $y$ on the side of the split nodes $a_i$. Do the SPR move of $(u,v)$ to $(a_k,x)$, that is, to an edge directly below the top. This produces the network with $k-1$ reticulations at the top, and (in Newick notation) embedded tree $((T\downarrow x,T\downarrow v),T\downarrow y)$, where $T\downarrow z$ denotes the part of tree $T$ below $z$. Then do the SPR move $(u',x)$ to $(b_k,y)$, the other side of the top, producing a network with $k-1$ reticulations at the top and embedded tree $((T\downarrow x,T\downarrow y),T\downarrow v)$. 

This creates the desired network with $v$ below one side of the top, and $x$ and $y$ on the other side. Both these SPR moves are performed using the technique of Lemma~\ref{lem:ChangeBaseTree}. After the SPR moves, we move the triangle back to the top without changing the embedded tree, and redirect the top reticulations as needed to produce $N'$.
\end{proof}

\begin{theorem}\label{the:HeadConnectTiers}
tier-$k$ of phylogenetic networks is connected by head moves, for all $k>0$.
\end{theorem}
\begin{proof}
Let $N$ and $N'$ be two networks in the same tier with the same leaf set. By Lemmas~\ref{lem:ReticsToTop} and~\ref{lem:NeatlyOnTop} we can turn both networks into networks $N_n$ and $N_n'$ with all reticulations neatly at the top using only head moves.
Now Lemma~\ref{lem:GeneralChangeBaseTree} tells us that we can turn $N_n$ into $N_n'$ using only head moves.
\end{proof}


\section{Local head moves}\label{sec:LocalConnectivity}
\subsection{Distance-1 is not enough}
One might hope that distance-1 head moves are enough to go from some network to any other network in the same tier. For tail moves such a result has been proved \citep{janssen2017exploring}, so it seems reasonable to expect a similar result for head moves. However, a simple example (Figure~\ref{fig:NoDist1HeadMoveExample}) shows that distance-1 head moves are not enough to connect the tiers of phylogenetic network space. This example is of a tier-1 network with 3 leaves, but it can easily be generalized to higher tiers and more leaves.

\begin{figure}[h!]
\begin{center}
\includegraphics[scale=0.8]{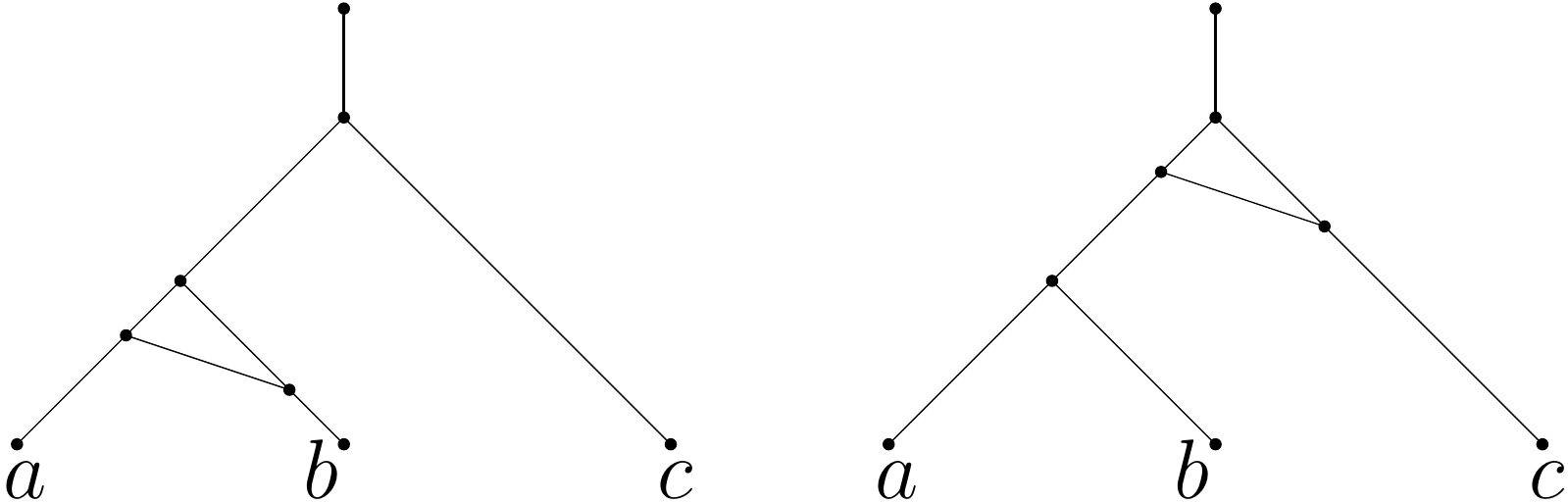}
\end{center}
\caption{There is no sequence of distance-1 head moves between these two networks.}
\label{fig:NoDist1HeadMoveExample}
\end{figure}

\subsection{Distance-2 suffices}
To prove the connectivity of tiers of network space using distance-2 head moves, we repeat the steps for general networks given in Section~\ref{sec:connectivity}. In particular, we prove that any step in the general proof can be taken using distance-2 head moves.

\begin{lemma}\label{lem:TriangleD2}
Triangles can be moved up and down split nodes and to/from the top using distance-2 head moves.
\end{lemma}
\begin{proof}

The hardest case we need to treat is the one where we move a triangle to/from the top. The other, easier, case is moving triangles up/down split nodes. This can be done using a simplified version of the sequence of head moves for the harder case.

A sequence of six distance-2 head moves for the harder case is shown in Figure~\ref{fig:TriangleMove2Head}. All intermediate DAGs are networks because there are no reticulations in the network that can cause parallel edges (except if $a=b$, in that case reverse the roles of $b$ and $c$), and acyclicity can easily be checked in the figure.

Note that moving triangles which are not at the top in a network where all other reticulations are at the top is easier: the given sequence of moves works where we skip steps 4 and 5 (these are only necessary because of the reticulations at the top).
\end{proof}

\begin{figure}[h!]
\begin{center}
\includegraphics[scale=0.8]{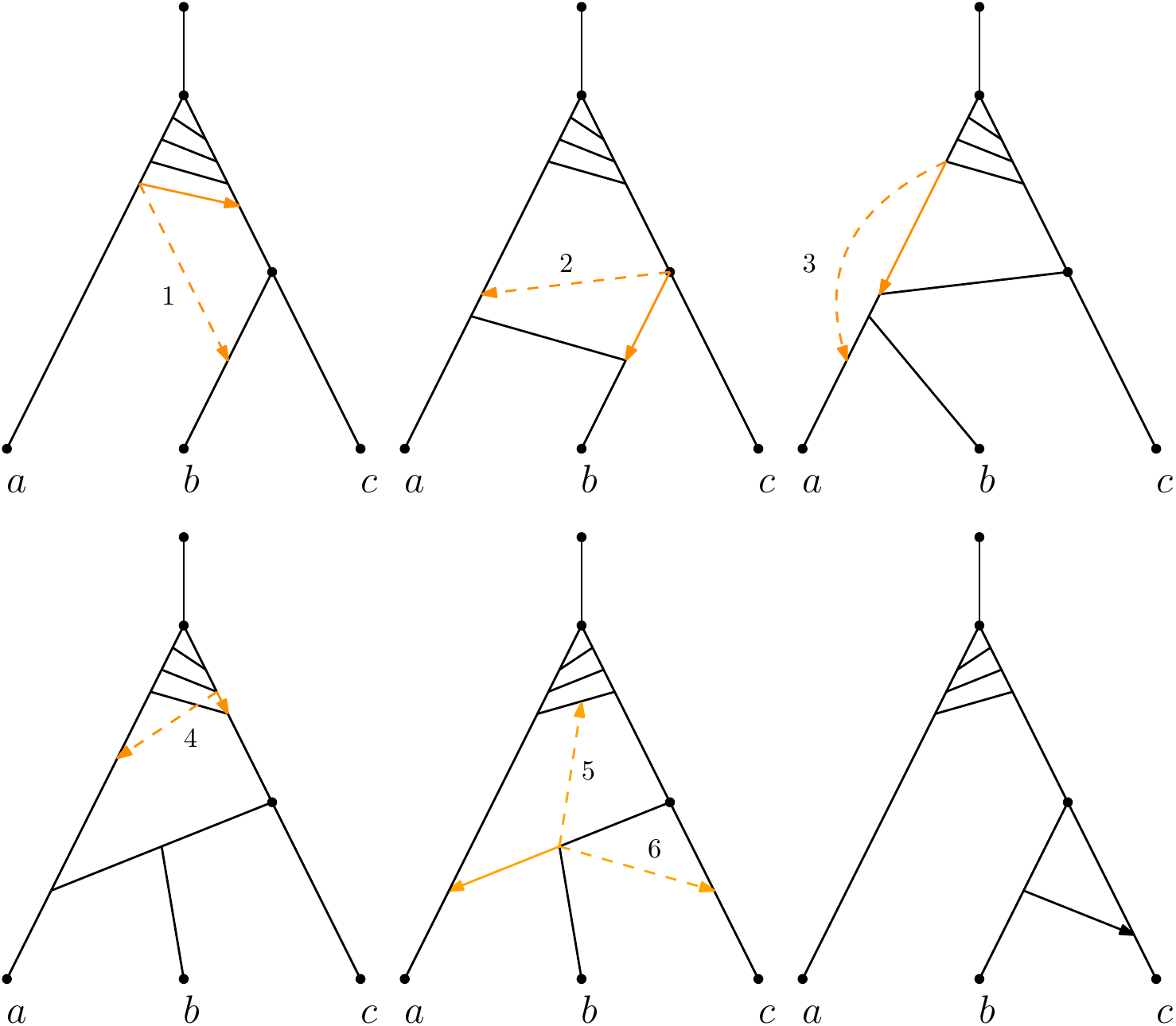}
\end{center}
\caption{Sequence of moves used to move triangles with distance 2 head moves. The numbers indicate the order of the distance 2 head moves. If there are no reticulations at the top, we may skip step 4 and do steps 5 and 6 with one move.}
\label{fig:TriangleMove2Head}
\end{figure}

\begin{lemma}\label{lem:RedirectTopReticD2}
Let $N$ be a network with $k$ reticulations at the top. Then the reticulations can be redirected so that they are neatly on top (directed to either edge) with at most $k-1$ distance-1 head moves.
\end{lemma}
\begin{proof}
This follows directly from the proof of Lemma~\ref{lem:NeatlyOnTop}, because the head move used in that proof is a distance-1 move.
\end{proof}

\begin{lemma}\label{lem:HeadByDistance2}
Let $N$ be a network and $v$ a highest reticulation below the top reticulations. Suppose $(u,v)$ to $(x,y)$ is a valid head move resulting in a network $N'$. Then there is a sequence of distance-2 head moves from $N$ to $N'$.
\end{lemma}
\begin{proof}
Pick an up-down path from $v$ to $(x,y)$ not via $(u,v)$. Note that if there is a part of this path above $u$, it is also above $v$ and therefore only contains split nodes. Sequentially move the head of $(u,v)$ to the pendant branches of this path as in Figure~\ref{fig:BranchHopping}. It is clear this works except at the point where $u$ is on the up-down tree path (the obvious move is a distance-3 move), and at the top.

Note that at the top, we need to move the head to the lowest reticulation edge at the top. This is of course only possible if this reticulation edge is directed away from $u$. If it is not, we redirect it using one distance-2 head move (Lemma~\ref{lem:RedirectTopReticD2}), and redirect it back after we move the moving head down to the other branch of the up-down tree path.

If $u$ is on the up-down path, we use Lemma~\ref{lem:TriangleD2} to pass this point: Let $c(u)$ be the other child of $u$ (not $v$) and $p(u)$ the parent node of $u$; moving the head from a child edge of $c(u)$ to the other child edge $(p(u),w)$ of $p(u)$ is equivalent to moving the triangle at $c(u)$ to a triangle at $p(u)$, which we can do by Lemma~\ref{lem:TriangleD2}.

We have to be careful, because if the child of $u$ is not a split node, this sequence of moves does not work. However, if $c(u)$ is a reticulation node, there exists a different up-down path from $v$ to $(x,y)$ not through $u$: such a path may use the other incoming edge of $c(u)$.

At all other parts of the up-down path, the head may be simply moved to the edges on the path. Using these steps, we can move the head of $(u,v)$ to $(x,y)$ with only distance-2 head moves.
\end{proof}

\begin{figure}[h!]
\begin{center}
\includegraphics[scale=0.8]{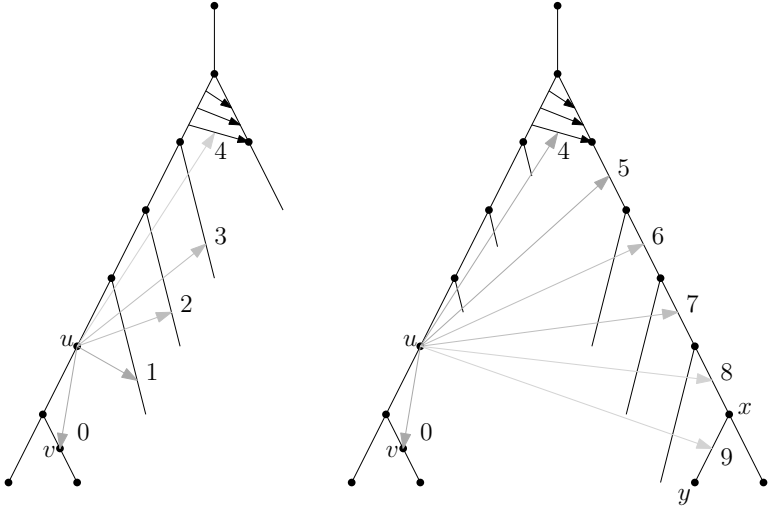}
\end{center}
\caption{An example of a sequence as used in Lemma~\ref{lem:HeadByDistance2}. Note that on the side of the tree containing the tail of the moving edge, we use the side branches to avoid cycles. The numbers represent the order of the distance-2 head moves. Note that the move from position 0 to position 1 is not a distance-2 head move, in this case we use the sequence of moves described in Lemma~\ref{lem:TriangleD2}. Also note that position 4 is only allowed when the lowest reticulation at the top is directed away from the tail of the moving edge.}
\label{fig:BranchHopping}
\end{figure}

Using these lemmas, it is easy to prove the distance-2 equivalent of Lemma~\ref{lem:ReticsToTop} for moving reticulations to the top.

\begin{lemma}\label{lem:ReticToTopD2}
Let $N$ be a tier-$k$ network, then there is a sequence of distance-2 head moves turning $N$ into a network with all reticulations at the top. 
\end{lemma}
\begin{proof}
Note that there exists a highest reticulation node $r$ not yet at the top of which one of the incoming edges $(s,r)$ is head movable. Like in Lemma~\ref{lem:ReticsToTop}, there is one (arbitrary distance) head move creating an extra reticulation at the top (if $s$ is directly below the top); or there is a head move creating a triangle below the parent of $s$. By Lemma~\ref{lem:HeadByDistance2} this head move can be simulated by distance-2 head moves.

In the case that we had to create a triangle, we can move the triangle up to the top using a sequence of distance-2 head moves (Lemma~\ref{lem:TriangleD2}). Repeating this for all reticulations, we arrive at a network with $k$ reticulations at the top.
\end{proof}

\begin{lemma}\label{lem:SPRMovesD2}
Let $N$ and $N'$ be tier $k>0$ networks with all reticulations neatly at the top. Then there exists a sequence of distance 2 head moves turning $N$ into $N'$. 
\end{lemma}
\begin{proof}
Exactly as the proof of Lemma~\ref{lem:GeneralChangeBaseTree}, all triangle moves are substituted as in Lemma~\ref{lem:TriangleD2}, and all other long distance head moves are simulated with distance-2 head moves as in Lemma~\ref{lem:HeadByDistance2}.
\end{proof}

\begin{theorem}
tier-$k$ of phylogenetic networks is connected by distance 2 head moves, for all $k>0$.
\end{theorem}
\begin{proof}
Exactly as the proof of Theorem~\ref{the:HeadConnectTiers}. 
Let $N$ and $N'$ be two arbitrary networks in the same tier with the same leaf set. Use Lemma~\ref{lem:ReticToTopD2} and Lemma~\ref{lem:RedirectTopReticD2} to change $N$ and $N'$ into networks $N_n$ and $N'_n$ with all reticulations neatly at the top using only distance-2 head moves. Now Lemma~\ref{lem:SPRMovesD2} tells us that there is a sequence of distance-2 head moves from $N_n$ to $N'_n$. Hence, tier-$k$ of phylogenetic network space is connected by distance-2 head moves.
\end{proof}

\section{Relation to tail moves}\label{sec:HeadToTailAndBack}
In this section, we show that each tail move can be replaced by a sequence of at most 15 head moves, and each head move can be replaced by a sequence of at most tail moves. 

\subsection{Tail move replaced by head moves}
In this section each move is a head move unless stated otherwise.

\begin{lemma}\label{lem:HorizontalTailSplitToHead}
Let $(u,v)$ from $(x,a_L)$ to $(x,a_R)$ be a valid tail move in a tier $k>0$ network $N$ resulting in a network $N'$. Then there exists a sequence of head moves from $N$ to $N'$ of length at most 6.
\end{lemma}
\begin{proof}
To prove this, we have to find a reticulation somewhere in the network that we can use, as the described part of the network might not contain any reticulations. 

Note that there exists a head movable reticulation edge $(t,r)$ in $N$ with $t$ not below both $a_L$ and $a_R$: Find a highest reticulation  node below $a_L$ and $a_R$; if it exists, one edge is movable, this edge cannot be below both; if there is no such reticulation, then  there is a reticulation $r$ that is not below both $a_L$ and $a_R$ and so the same holds for its movable edge.

First assume we find a head movable edge $(t,r)$ with $t=x$. Note that $r$ cannot be the same node as $u$, as $u$ is a split node and $r$ is a reticulation. This means that $r=a_R$, and $(x,a_R)$ is movable. Move $(x,a_R)$ to $(u,a_L)$, which is allowed because $x\neq u$, $a_L$ not above $x$ and $(t,r)=(x,a_R)$ is movable. Now moving $(u,r')$ to $(s,z)$, the arc created by suppressing $a_R$ after the previous move, we get network $N'$.

Now assume we find a head movable edge $(t,r)$ with $x\neq t$. Suppose w.l.o.g. $(t,r)$ is not below $a_L$, 
then we can use the following sequence of $4$ moves except in the cases we mention in bold below the steps. For this lemma, we call this sequence the `normal' sequence (Figure~\ref{fig:Lemma11NormalSequence}).

\begin{itemize}
\item Move $(t,r)$ to $(u,a_L)$, keeping $(x,a_R)$ except if $\bm{x=t}$. This can be done if (Lemma~\ref{lem:validHead}):\\
$(t,r)$ is movable (it is by choice of $(t,r)$),\\ no parallel edges by reattaching $(t,r')$, i.e. $t\neq u$ (note that $\bm{t=u}$ may occur),\\ no cycles at reattachment because $t$ not below $a_L$ (by choice of $(t,r)$).
\item Move $(u,r')$ to $(x,a_R)$, creating edge $(t,a_L)$ as $(x,a_R)\neq (t,r')$ and $(x,a_R)\neq(r',a_L)$;\\ 
Only not movable if $\bm{(t,a_L)\in N}$,\\ reattachment causes no cycles because $a_R$ is not above $x$,\\ reattachment causes no parallel edges because $u\neq x$.
\item Move $(x,r'')$ to $(t,a_L)$;\\ Movable because $v\neq a_R$ (otherwise the tail move is not valid),\\ no cycles because $a_L$ not above $x$,\\ no parallel edges because $t\neq x$.
\item Move $(t,r''')$ back to its original position.\\ Allowed because this produces $N'$.
\end{itemize}

\begin{figure}[h!]
\begin{center}
\includegraphics[scale=0.8]{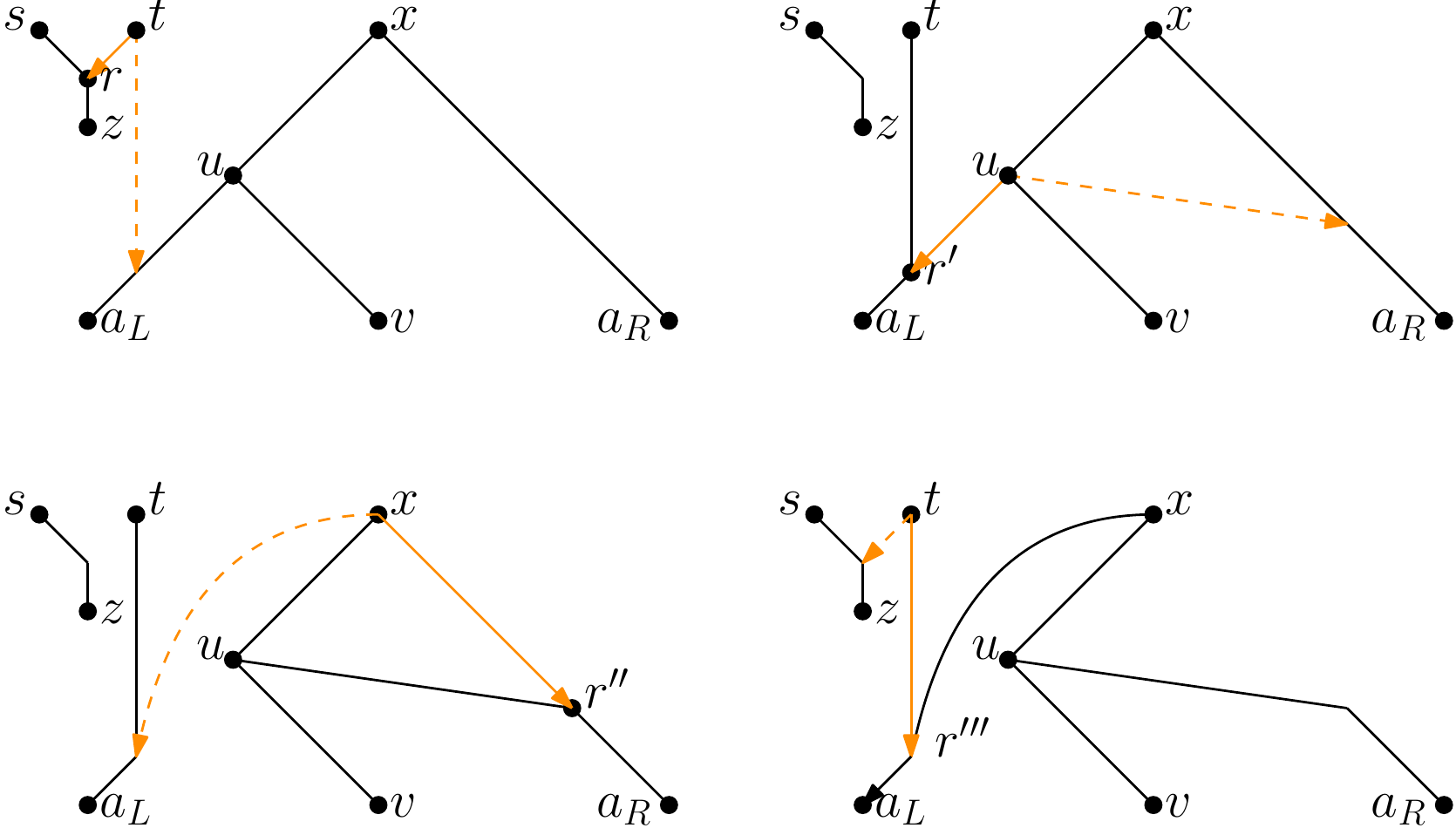}
\end{center}
\caption{The `normal' sequence of head moves simulating a tail move in Lemma~\ref{lem:HorizontalTailSplitToHead}.}
\label{fig:Lemma11NormalSequence}
\end{figure}

We now look at the situations where $t=x$, $t=u$ or $(t,a_L)\in N$ separately. We will split up in cases to keep the proof clear. Note that $(t,r)\in N$ is still a movable edge.

\begin{enumerate}
\item{$\bm{t=u}$.}
 \begin{enumerate}
  \item{$\bm{r=a_L}$.}\label{case:UaLequalsTR}
   Move $(t,r)$ to $(x,a_R)$, then move $(x,r')$ back to the original position of $r$, creating $N'$. This is a sequence of $2$ head moves.
   \item{$\bm{r=v}$.}\label{case:UVequalsTR}
   Move $(t,r)$ to $(x,a_R)$, creating a triangle at $x$. Now reverse the triangle by moving $(x,r')$ to $(t,a_L)$. Now create $N'$ by moving $(t,r'')$ back to the original position of $r$. This is a sequence of $3$ head moves.
 \end{enumerate}
\item{$\bm{t\neq u}$.}
 Note that we can assume that \textbf{$\bm{(u,a_L)}$ is not movable}, as otherwise we are in the previous case. Because $t\neq u$, we may also assume $\bm{(t,a_L)\in N}$, otherwise this is no special case and we can use the sequence of moves from the start of this proof. Hence, $a_L$ is a reticulation node on the side of a triangle formed by $t$, $a_L$, and the child $c(a_L)$ of $a_L$.
 \begin{enumerate}
  \item{\textbf{$\bm{t}$ is below $\bm{a_R}$.}}
    \begin{enumerate}
     \item{\textbf{$\bm{t}$ is below $\bm{v}$.}}
      Since $t$ is below both $a_R$ and $v$, there is a highest reticulation $s$ strictly above $t$ and below both $a_R$ and $v$. Since $s$ is strictly above $t$, it is strictly above $a_L$. Therefore we are either in the `normal' case, or in Case~\ref{case:UVequalsTR} of this analysis with movable edge $(p(s),s)$. This means this situation can be solved using at most $4$ head moves.
     \item{\textbf{$\bm{t}$ is not below $\bm{v}$.}}
      As $(t,a_L)$ is a reticulation edge in the triangle, it is movable, and because $t$ is not below $v$, the head move $(t,a_L=r)$ to $(u,v)$ is allowed. Now the \emph{tail} move $(u,r')$ to $(x,a_R)$ is still allowed, because $v$ is not above $x$. As $(u,c(a_L))$ is movable in this new network, we can simulate this tail move like in Case~\ref{case:UaLequalsTR}. Afterwards, we can put the triangle back in its place with one head move, which is allowed because it produces $N'$. All this takes $6$ head moves.
   \end{enumerate}
  \item{\textbf{$\bm{t}$ is not below $\bm{a_R}$.}}
   Because $t\neq x$, we know that $(t,a_R)\not\in N$. Therefore we can do the `normal' sequence of moves from the start of this proof in reverse order, effectively switching the roles of $a_L$ and $a_R$. Because we use the `normal' sequence of moves, this case takes at most $4$ head moves.
 \end{enumerate}
\end{enumerate}
\end{proof}

Now, to prove the case of a more general tail move, we need to treat another simple case first. 

\begin{lemma}\label{lem:HorizontalShortTailMoveSimulated}
Let $(u,v)$ from $(x_L,r)$ to $(x_R,r)$ be a valid tail move in a network $N$ turning it into $N'$, then there is a sequence of head moves from $N$ to $N'$ of length at most 4.
\end{lemma}
\begin{proof}
Let $z$ be the child of $r$, and note that not all nodes described must necessarily be unique. All possible identifications are $x_L=x_R$ and $v=z$, other identifications create cycles. First note that in the situation $x_L=x_R$, the networks $N$ and $N'$ before and after the tail move are isomorphic. Hence we can restrict our attention to the case that $x_L\neq x_R$. To prove the result, we distinguish two cases.
\begin{enumerate}
\item{$\bm{z\neq v}$.}\label{case:SimpleLocalTailSimulation} This case can be solved with two head moves: $(x_R,r)$ to $(x_L,u)$ creating new reticulation node $r'$ above $u$ followed by $(r',u)$ to $(u,z)$. The first head move is allowed because $v\neq z$, so $(x_R,r)$ is head movable; $x_R\neq x_L$; and $u$ is not above $x_R$ because both its children aren't: $z$ is below $a_R$, and if $v$ is above $x_R$, the tail move $N\to N'$ is not allowed. The second head move is allowed because it produces the valid network $N'$. Hence the tail move can be simulated by at most $2$ head moves (Figure~\ref{fig:LocalTailToHead1}).




\begin{figure}[h!]
\begin{center}
\includegraphics[scale=0.7]{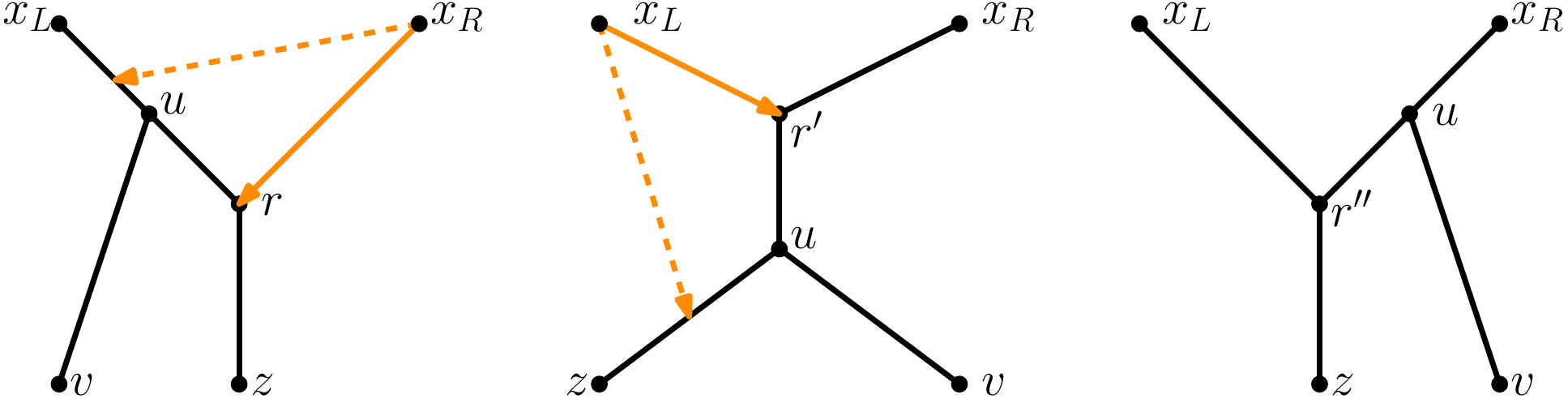}
\end{center}
\caption{The two moves used to simulate a tail move in Case~\ref{case:SimpleLocalTailSimulation} of Lemma~\ref{lem:HorizontalShortTailMoveSimulated}.}
\label{fig:LocalTailToHead1}
\end{figure}

\item{$\bm{z  =  v}$.}\label{case:ExtendedLocalTailToHead} The proposed moves of the previous case are not valid here, because they lead to parallel edges in the intermediate network. To prevent these, we reduce to the previous case by moving $(u,z)$ to any edge $e$ not above $z$ and $e\neq (z,c(z))$ (hence neither above $x_L$ nor above $x_R$) and moving it back afterwards. Note that if there is such an edge $e$, then the head move $(u,z)$ to $e$ is allowed. 
Afterwards, the tail move $(u,z')$ to $(x_R,r)$ is `still' allowed 
and can therefore be simulated by 2 head moves as in the previous case, the last head move, moving $(u,z')$ back is allowed because it creates the DAG $N'$ which is a network. Such a sequence of moves uses $4$ head moves (Figure~\ref{fig:LocalTailToHead2a}).

\begin{figure}[h!]
\begin{center}
\includegraphics[scale=0.6]{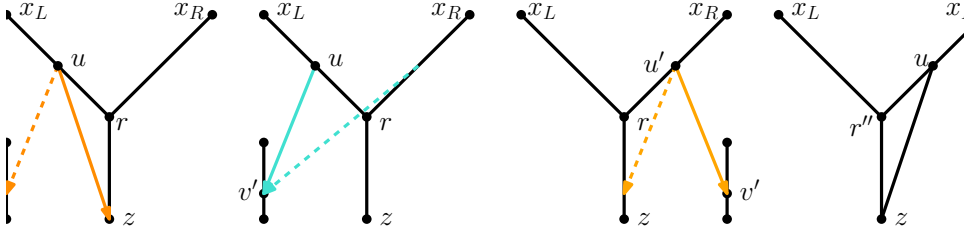}
\end{center}
\caption{The four moves used in Case~\ref{case:ExtendedLocalTailToHead} of Lemma~\ref{lem:HorizontalShortTailMoveSimulated}. The middle depicted move is the tail move of Case~\ref{case:SimpleLocalTailSimulation}, which can be replaced by two head moves.}
\label{fig:LocalTailToHead2a}
\end{figure}

It rests to prove that there is such a location (not above $z$ and excluding $(z,c(z))$) to move $(u,z)$ to. Recall that we assume any network has at least two leaves. Let $l$ be a leaf not equal to $c(z)$, then its incoming edge $(p(l),l)$ is not above $c(z)$ and not equal to $(z,c(z))$. Hence this edge $e=(l,p(l))$ suffices as a location for the first head move.

\end{enumerate}

We conclude that any tail move of the form $(u,v)$ from $(x_L,r)$ to $(x_R,r)$ can be simulated by 4 head moves.
\end{proof}

\begin{lemma}\label{lem:TailToHeadReticNotAboveaR}
Let $(u,v)$ from $(x_L,a_L)$ to $(x_R,a_R)$ be a valid tail move in a network $N$ resulting in a network $N'$. Suppose $a_L\neq a_R$, $a_L$ is not above $a_R$, and there exists a movable reticulation edge $(t,r)$ not below $a_R$. Then there exists a sequence of head moves from $N$ to $N'$ of length at most 7.
\end{lemma}
\begin{proof}
Note that $v$ cannot be above either of $x_L$ and $x_R$. The only possible identifications within the nodes $a_L, a_R, x_L, x_R, u, v$ are $a_L=a_R$, $x_L=x_R$ and $a_R=x_L$ (but not simultaneously), all other identifications lead to parallel edges, cycles in either $N$ or $N'$, a contradiction with the condition ``$a_L$ is not above $a_R$'', or a trivial situation where the tail move leads to an isomorphic network. The first of these two identifications have been treated in the previous two lemmas, so we may assume $\bm{a_L\neq a_R}$ and $\bm{x_L\neq x_R}$.
We now distinguish several cases to prove the tail move can be simulated by a constant number of head moves in all cases.

\begin{enumerate}
 \item{\bf $\bm{(t,a_R)\not\in N}$.} 
 \begin{enumerate}
  \item {$\bm{r=x_R}$.} 
   As $(t,r)$ is movable and not below $a_L$ or $v$, we can move the head of this edge to $(x_L,u)$. The head move $(x_L,r')$ down to $(u,a_L)$ is then allowed. Let $s$ be the parent of $r$ in $N$ that is not $t$. Since $u\neq s$ (otherwise the original tail move was not allowed), the head move $(u,r'')$ to $(s,a_R)$ is allowed, where $s$ is the other parent of $r$ in $N$ (i.e., not $t$). Lastly $(s,r''')$ to $(t,u)$ gives the desired network $N'$. 
  \item {$\bm{r\neq x_R}$.}\label{case:Lemma13CaseFigure}
   In this case, we can move $(t,r)$ to $(x_R,a_R)$ in $N$ (if $t=x_R$ then $(t,a_R)\in N$, contradicting the assumptions of this case). Because neither $a_L$ nor $v$ can be above $x_R$ and $x_L\neq x_R$, we can now move $(x_R,r')$ to $(x_L,u)$. Then we move down the head $(x_L,r'')$ to $(u,a_L)$, followed by $(u,r''')$ to $(t,a_R)$. If $u=t$ and $r=v$, the last move is not allowed, and if $u=t$ and $r=a_L$ these last two moves are not allowed. In these cases, we simply skip these move. Lastly, we move $(t,r'''')$ to $(s,z)$ to arrive at $N'$, where $s$ and $z$ are the other parent and the child of $r$ in $N$. Hence the tail move of this situation can be simulated by $5$ head moves.
\end{enumerate}

\begin{figure}[h!]
\begin{center}
\includegraphics[scale=0.7]{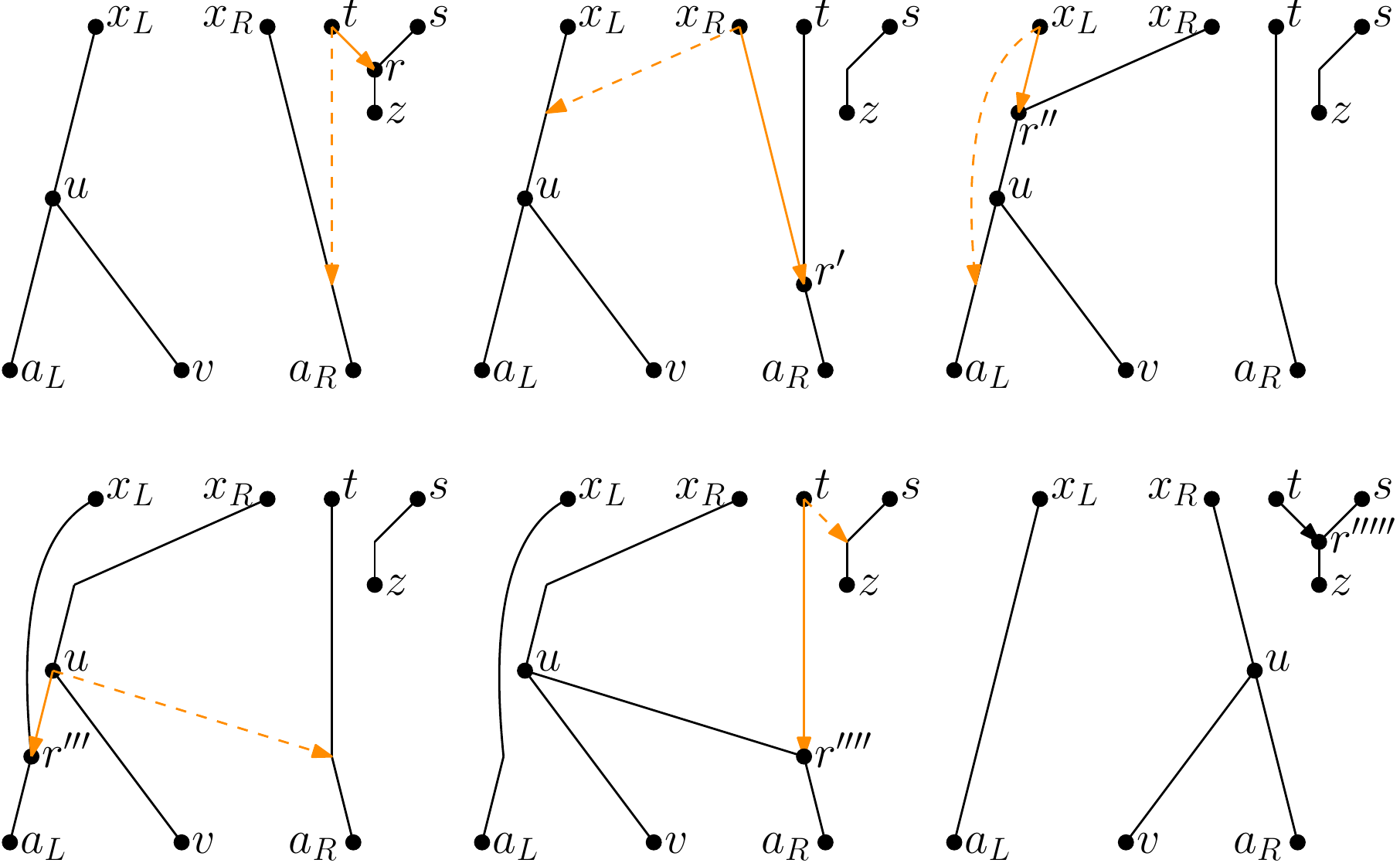}
\end{center}
\caption{The five moves used to simulate a tail move in Case~\ref{case:Lemma13CaseFigure} of Lemma~\ref{lem:TailToHeadReticNotAboveaR}.}
\label{fig:Lemma13Sequence}
\end{figure}

\item{$\bm{(t,a_R)\in N}$.}\label{case:TARInN}
Again $(t,r)$ is the head movable edge. Let $z$ be the child of $r$ and $s$ the other parent of $r$.
\begin{enumerate}
\item{$\bm{z=a_R}$.}\label{Case:ZEqualsAR}
 Note first that in this case, we must have either $x_R=t$ or $x_R=r$, otherwise one of the edges $(t,a_R)$ and $(r,z)$ is not in $N$. 
 \begin{enumerate}
  \item{\bf$\bm{x_R=t}$} \label{Case:ZEqualsARandTequalsXR}
   This case is quite easy, and can be solved with 3 head moves. Because $r$ and $t=x_R$ are distinct, $a_R=z$ is a reticulation node with movable edge $(t=x_R,z=a_R)$. The sequence of moves is: $(t,z)$ to $(x_L,u)$, then $(x_L,z')$ to $(u,a_L)$, then $(u,z'')$ to $(r,c(z))$.
  \item{\bf$\bm{x_R=r}$}
   Note that the \emph{tail} move $(u,v)$ to $(t,a_R)$ is also allowed in this case because $(u,v)$ is tail movable, $v\neq a_R$ and $v$ not above $t$ (otherwise the tail move to $(r,z)$ is not allowed either). This tail move is of the type of the previous case, and takes at most 3 head moves. Now the move $(u',v)$ to $(r,z)$ is of the type of Lemma~\ref{lem:HorizontalShortTailMoveSimulated}, which takes at most 4 head moves to simulate. We conclude any tail move of this case can be simulated with $7$ head moves.
 \end{enumerate}
\item{$\bm{z\neq a_R}$.} 
\begin{enumerate}
\item{$\bm{a_R\neq r}$.}\label{case:TailToHeadFull2bi}
\begin{enumerate}
\item{$\bm{x_R  =  t}$ and $\bm{v\neq r}$.}
We can move the \emph{tail} to $(x_R,r)$ with three head moves like in Case~\ref{Case:ZEqualsARandTequalsXR}. The resulting DAG is a network because $v$ is not above $x_R$ and $v\neq r$. Call the resulting new location of the tail $u'$. We can get to $N'$ with two head moves (Lemma~\ref{lem:HorizontalTailSplitToHead}~Case~\ref{case:UaLequalsTR}): $(u',r)$ to $(t,a_R)$, then $(t,r')$ to $(s,z)$. This case therefore takes at most $5$ head moves.
\item{$\bm{x_R  =  t}$ and $\bm{v  =  r}$.}
The following sequence of four head moves suffices: $(t=x_R,r=v)$ to $(x_L,u)$, then $(x_L,r')$ to $(u,a_L)$, then $(u,r'')$ to $(t,a_R)$ and finally $(t,r''')$ to $(u,z)$. Hence this case takes at most $4$ head moves.
\item{$\bm{x_R\neq t}$ and $\bm{a_R\neq s}$.}
Because $a_R\neq s$, the edge $(x_R,a_R)$ is movable. Also, because $x_R\neq x_L$ and $u$ not above $x_R$, $(x_R,a_R)$ can be moved to $(x_L,u)$. Now $(x_L,a_R')$ is movable, and it can be moved down to $(u,a_L)$. Finally, the head move $(u,a_R'')$ to $(t,c(a_R))$ results in $N'$. Hence in this case we need at most $3$ head moves. 
\item{$\bm{x_R\neq t}$ and $\bm{a_R  =  s}$.}
The following sequence of five head moves suffices: $(t,r)$ to $(u,a_L)$, then $(x_R,s)$ to $(x_L,u)$, then $(x_L,s')$ to $(u,r')$, then $(u,s'')$ to $(t,c(r))$, and finally $(t,r')$ to $(s''',c(r))$. Hence this case takes at most $5$ head moves.
\end{enumerate}


\item{$\bm{a_R = r}$.} 
In this case either $x_R=t$ or $x_R=s$. 
\begin{enumerate}
\item{$\bm{x_R=t}$.} This case is easily solved with 3 head moves: $(x_R,a_R)$ to $(x_L,u)$, then $(x_L,a_R')$ to $(u,a_L)$, then $(u,a_R'')$ to $(s,z)$. 
\item{$\bm{x_R=s}$.} If $(s,r)$ is movable (i.e. there is no edge $t,z$), then we can relabel $t \leftrightarrow s$ and treat like the previous case. Otherwise, there is an edge $(t,z)$ and we use the following sequence of moves: $(t,a_R)$ to $(u,a_L)$, then $(x_R=s,z)$ to $(x_L,u)$, then $(x_L,z')$ to $(u,a_R')$, then $(u,z'')$ to $(t,c(z))$, then $(t,a_R')$ to $(z''',c(z))$. The tail move of this situation can therefore be replaced by 5 head moves.
\end{enumerate}
\end{enumerate}
\end{enumerate}

\end{enumerate}
\end{proof}

\begin{lemma}
Let $(u,v)$ from $(x_L,a_L)$ to $(x_R,a_R)$ be a valid tail move in a network $N$ resulting in a network $N'$. Suppose $a_L\neq a_R$, $a_L$ is not above $a_R$, and all movable reticulation edges are below $a_R$. Then there exists a sequence of head moves from $N$ to $N'$ of length at most 15.
\end{lemma}
\begin{proof}
Like in the proof of last lemma, we assume that $a_L$ is not above $a_R$. Because the network has at least one reticulation, we can pick a highest reticulation $r$ in the network, let $(t,r)$ be its movable edge. As each movable reticulation edge is below $a_R$, so is $(t,r)$. Let us denote the root of $N$ with $\rho$, and distinguish two subcases:
\begin{enumerate}
\item{$\bm{x_R\neq \rho}$.}
Because $x_R$ is above $a_R$, it must be a split node, it has another child edge $(x_R,b)$ with $b\neq a_R$ not above $t$: if $b$ were above $t$, there has to be a reticulation above $r$, contradicting our choice of $r$.
\begin{enumerate}
\item{$\bm{r\neq b}$.}
In this case, we can move $(t,r)$ to $(x_R,b)$ in both $N$ and $N'$, producing networks $M$ and $M'$.
Now $(x_R,r')$ is movable in $M$, and by relabelling $t'=x_R$ we can see that there is one tail move between $M$ and $M'$ of the same type as Case~\ref{case:TailToHeadFull2bi} of Lemma~\ref{lem:TailToHeadReticNotAboveaR}. To see this, take $r'$ as the relevant reticulation with movable edge $(t',r')$ and consider the tail move $(u,v)$ to $(x_R,a_R)$ producing $M'$. This case can therefore be solved with at most $5+2=7$ head moves.
\item{{\bf $\bm{r  =  b}$ and $\bm{(t,c(r))\not\in N}$}.}
In this case, $(x_R,r)$ is movable, and not below $a_R$, contradicting our assumptions. 
\item{{\bf $\bm{r  =  b}$ and $\bm{(t,c(r)))    \in N}$}.}
Because $N$ has at least two leaves, there must either be at least 2 leaves below $r$, or there is a leaf not below $r$. Let $l$ be an arbitrary leaf below $r$ in the first case, or a leaf not below $r$ in the second case. Note that the head move $(t,c(r))$ to the incoming edge of $l$ is allowed, and makes $(x_R,r)$ movable. Now the tail move $(u,v)$ to $(x_R,a_R)$ is still allowed, because $v\neq a_R$, $v$ is not above $x_R$ and $(u,v)$ is tail movable. 
For this tail move we are in a case of Lemma~\ref{lem:TailToHeadReticNotAboveaR} because $(x_R,r)$ is not below $a_R$, hence this tail move takes at most 7 moves. After this move, we can do one head move to put $(t,c(r))$ back. Hence this case takes at most $9$ moves.
\end{enumerate}
\item{$\bm{x_R  =  \rho}$.}
Let $y,z$ be the children of $a_R$. Now first do the tail move of $(u,v)$ to one of the child edges $(a_R,z)$ of $a_R$. This is allowed because $a_R$ is the top split. The sequence of head moves used to do this tail move is as in the previous case. Note that $N'$ is now one tail move away: $(u',z)$ to $(a_R,y)$. This is a horizontal tail move along a split node as in Lemma~\ref{lem:HorizontalTailSplitToHead}, which takes at most $6$ head moves. As the previous case took at most $9$ head moves, this case takes at most $15$ head moves in total.

\end{enumerate}
\end{proof}

\begin{lemma}\label{lem:tailtoheadaLnotaboveaR}
Let $(u,v)$ from $(x_L,a_L)$ to $(x_R,a_R)$ be a valid tail move in a network $N$ resulting in a network $N'$. Suppose $a_L\neq a_R$ and $a_L$ is not above $a_R$, then there exists a sequence of head moves from $N$ to $N'$ of length at most 15.
\end{lemma}
\begin{proof}
This is a direct consequence of the previous two lemmas.
\end{proof}

\begin{lemma}
Let $(u,v)$ from $(x_L,a_L)$ to $(x_R,a_R)$ be a valid tail move in a network $N$ resulting in a network $N'$. Suppose $a_L\neq a_R$ and $a_L$ is above $a_R$, then there exists a sequence of head moves from $N$ to $N'$ of length at most 15.
\end{lemma}
\begin{proof}
Note that in this case $a_R$ is not above $a_L$ in $N'$.
Reversing the labels $x_L\leftrightarrow x_R$ and $a_L\leftrightarrow a_R$ we are in the situation of Lemma~\ref{lem:tailtoheadaLnotaboveaR} for the reverse tail move $N'$ to $N$. This implies the tail move can be replaced by a sequence of at most 15 head moves.
\end{proof}

\begin{theorem}
Any tail move can be replaced by a sequence of $15$ head moves.
\end{theorem}
\begin{proof}
Follows from the previous lemmas.
\end{proof}

\subsection{Head move replaced by tail moves}
In this section each move is a tail move unless stated otherwise.

We first recall a result from \citep{janssen2017exploring}: any distance one head move can be replaced by a constant number of tail moves, so the following result holds.

\begin{lemma}\label{lem:D1HeadMove}
Let $(u,v)$ from $(x_1,y_1)$ to $(x_2,y_2)$ be a valid distance one head move in a network $N$ resulting in a network $N'$. Then there is a sequence of at most $4$ tail moves between $N$ and $N'$, except if $N$ and $N'$ are different networks with two leaves and one reticulation.
\end{lemma}

And there is the following special case, for which we repeat the proof here.

\begin{lemma}\label{lem:D1HeadMoveDown}
Let $(u,v)$ from $(x_1,y_1)$ to $(x_2,y_2)$ be a valid head move in a network $N$ resulting in a network $N'$. Suppose that $y_1=x_2$ and $x_2$ is a \emph{split node}, then there is a sequence of at most $1$ tail moves between $N$ and $N'$. 
\end{lemma}
\begin{proof}
Let $c(x_2)$ be the other child of $x_2$ (not $y_2$), then the tail move $(x_2,c(x_2))$ to $(x_1,v)$ suffices.
\end{proof}

Now we prove that there is such a sequence of constant length if $y_1$ is above $x_2$ and $x_2$ is a split node. The proof uses the previous lemma by creating a similar situation in a constant number of tail moves.

\begin{lemma}\label{lem:HeadMoveDown}
Let $(u,v)$ from $(x_1,y_1)$ to $(x_2,y_2)$ be a valid head move in a network $N$ resulting in a network $N'$. Suppose that $y_1$ is above $x_2$, $y_1\neq x_2$, and $x_2$ is a split node, then there is a sequence of at most $4$ tail moves between $N$ and $N'$.
\end{lemma}
\begin{proof}
We split this proof in two cases: $(x_2,y_2)$ is movable, or it is not. We prove in both cases there exists a constant length sequence of tail moves between $N$ and $N'$.
\begin{enumerate}
\item{\textbf{$\bm{(x_2,y_2)}$ is tail-movable.}} Tail move $(x_2,y_2)$ up to $(v,y_1)$, this is allowed because any tail move up is allowed if the moving edge is tail-movable (Corollary~\ref{cor:MoveUpDown}). Now $(u,v)$ is still head-movable, hence we can move it down to $(x_2',y_2)$. As this is exactly the situation of Lemma~\ref{lem:D1HeadMoveDown}, we can replace this head move by one tail move. Now tail-moving $(x_2',v')$ back down results in $N'$, so this move is allowed, too. Hence there is a sequence of $3$ tail moves between $N$ and $N'$.
\item{\textbf{$\bm{(x_2,y_2)}$ is not tail-movable.}}\label{case:headmovedownbelowasplit} Because $x_2$ is a split node and $(x_2,y_2)$ is not movable, there has to be a triangle with $x_2$ at the side, formed by the parent $p$ of $x_2$ and the other child $c$ of $x_2$. Note that $(p,x_2)$ is tail-movable, and that it can be moved up to $(v,y_1)$. After this move, Lemma~\ref{lem:D1HeadMoveDown} tells us we can head-move $(u,v)$ to $(p',x_2)$ using one tail move. The next step is to tail move $(p',v')$ back down to the original position of $p$. The resulting network is allowed because it is one valid distance one head move away from $N'$ (as $c$ is not above $u$). Lastly, we do this distance one head move, which again can be simulated by one tail move by Lemma~\ref{lem:D1HeadMoveDown}. Note that this sequence is also valid if $p=y_1$. Hence there is a sequence of at most $4$ tail moves between $N$ and $N'$.
\end{enumerate}
\end{proof}

\begin{figure}[h!]
\begin{center}
\includegraphics[scale=0.7]{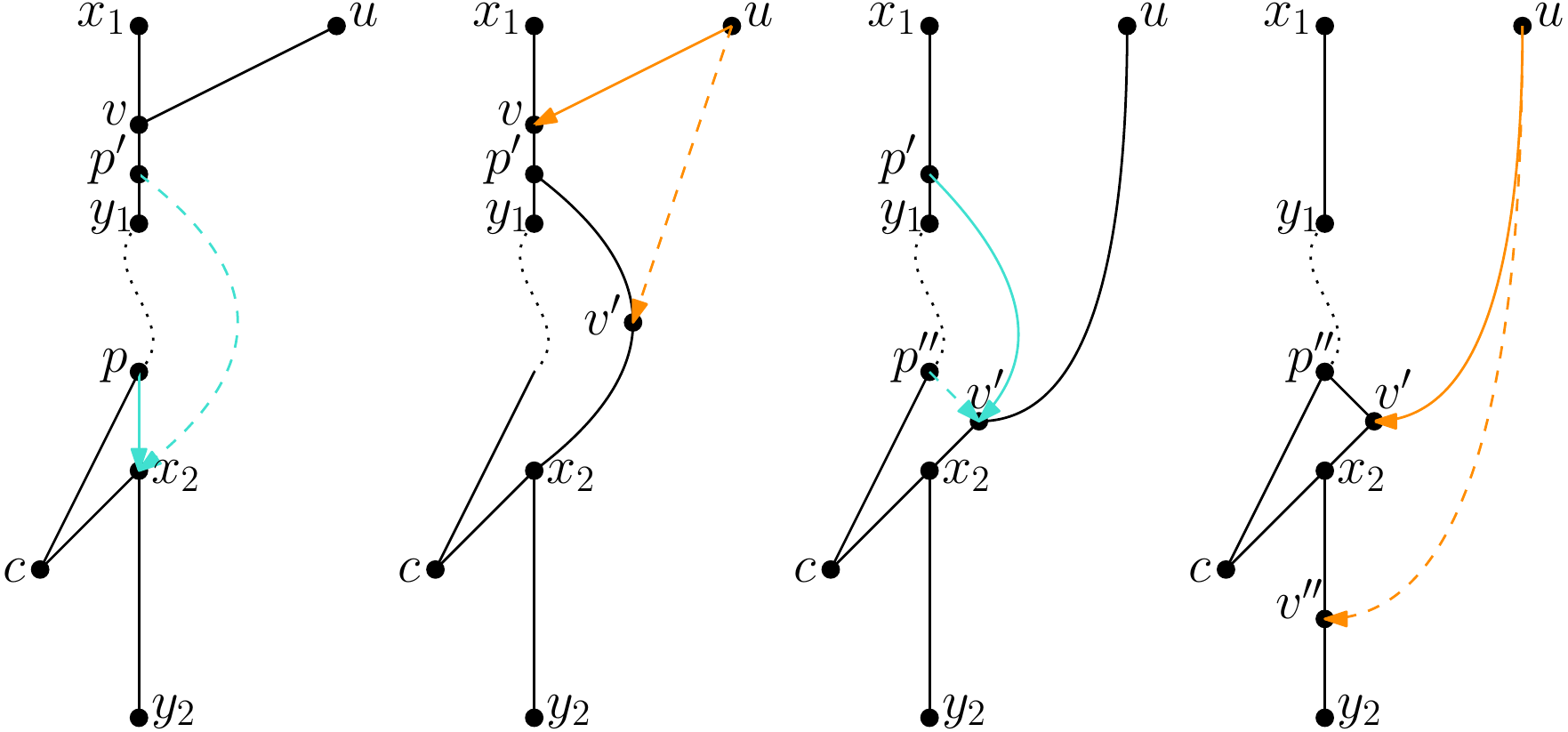}
\end{center}
\caption{The four moves used in Case~\ref{case:headmovedownbelowasplit} of Lemma~\ref{lem:HeadMoveDown}. The two same coloured edges of which one is dashed in each part are the location of an edge before and after a move. The edges with the other colour correspond to moving edge of the previous move.}
\label{fig:HeadMoveDownBelowASplit}
\end{figure}

\begin{lemma}\label{lem:makeMovable}
Let $(u,v)$ from $(x_1,y_1)$ to $(x_2,y_2)$ be a valid head move in a network $N$ resulting in a network $N'$. Suppose that $y_1$ is (strictly) above $x_2$ and $x_2$ is a reticulation, then there are networks $M$ and $M'$ such that the following hold:
\begin{enumerate}
\item turning $N$ into $M$ takes at most one tail move;
\item turning $N'$ into $M'$ takes at most one tail move;
\item there is a head move between $M$ and $M'$, moving the head \emph{down} to an edge whose top node is a reticulation; 
\item there is a tail movable edge $(s,t)$ in $M$ with $t$ not above $x_2$.
\end{enumerate}
\end{lemma}

\begin{proof}
Note that we have to find a sequence of a tail move followed by a head move and finally a tail move again, between $N$ and $N'$ such that the head move is of the desired type and the network after the first tail move has a movable edge not above the top node $x_2$ of the receiving edge of the head move.

Note that if there is a tail movable edge $(s,t)$ in $N$ with $t$ not above $x_2$, we are done by the previous lemmas: take $M:=N$ and $M':=N'$. Hence we may assume that there is no such edge in $N$. Suppose all leaves (of which there are at least $2$) are below $y_2$, then there must also be a split node below $y_2$. And as one of its child edges is movable, there is a tail movable edge below $y_2$ (and hence not above $x_2$). So if all leaves are below $y_2$, we can again choose $M:=N$ and $M':=N$. 

Because our networks have at least $2$ leaves, the remaining part is to show the lemma assuming that there is a leaf $l_1$ not below $y_2$. Note that there also exists a leaf $l_2$ below $y_2$. Now consider an LCA $j$ of $l_1$ and $l_2$. We note that $j$ is a split node of which at least one outgoing edge \textbf{$\bm{(j,m)}$ is not above $\bm{x_2}$}. If $(j,m)$ is tail movable, then $M:=N$ and $M':=N'$ suffices, so assume \textbf{$\bm{(j,m)}$ is not tail movable}. Let $i$ be the parent of $j$, and $k$ be the other child of $j$; because $j$ is a split node and $(j,m)$ is not movable, \textbf{$\bm{i}$, $\bm{j}$ and $\bm{k}$ form a triangle} (Figure~\ref{fig:HeadToTailBetterSituation}). 

\begin{figure}[h!]
\begin{center}
\includegraphics[scale=0.7]{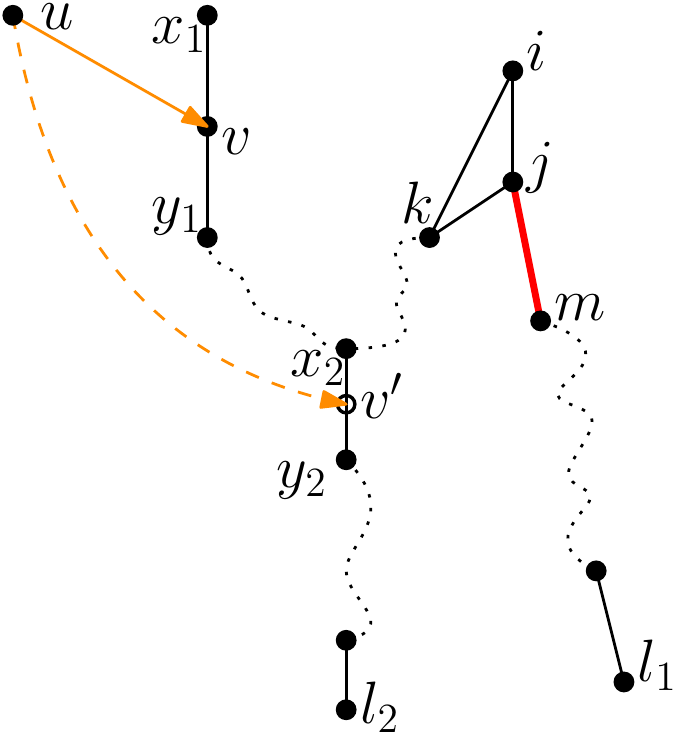}
\end{center}
\caption{The situation of Lemma~\ref{lem:makeMovable} in which we want to make the red edge $(j,m)$ movable, in the network before (with orange solid line) and after (with orange dashed line) the head move. Dotted black lines indicate ancestral relations, but are not necessarily edges in the network.}
\label{fig:HeadToTailBetterSituation}
\end{figure}

The idea is to `break' this triangle with one tail move in $N$ and $N'$ simultaneously, meaning we either move one of the edge of the triangle, or we move a tail to an edge of the triangle. If we can break the triangle in both networks keeping $(u,v)$ movable, creating new networks $M$ and $M'$, then choosing $(s,t):=(j,m)$ in $M$ will work. The last part of this proof shows how we do this. We have to split in two cases:
\begin{itemize}
\item{\textbf{$\bm{i}$ is the child of the root.}} In this case we break the triangle by moving a tail to the triangle. As $v$ is a reticulation and there is no path from any node below $m$ to $v$ (if so, there is a path from $m$ to $x_2$), there must be a split node $p$ below $k$ and (not necessarily strictly) above both parents of $v$. At least one of the outgoing edges $(p,q)$ is movable in $N$. If $v$ is a child of $p$ and $(p,v)$ is movable, then we choose $q=v$, otherwise any choice of $(p,q)$ will suffice.

Because $(p,q)$ is movable (by choice of $(p,q)$) and $k$ is above $p$, the tail move $(p,q)$ to $(j,k)$ is valid. Now the head move $(u,v)$ (or $(u',v)$ if $p=u$ in $N$) to $(x_2,y_2)$ is valid, because $x_2$ is below $v$, and $(u,v)$ is movable because $(u,v)$ was movable in $N$, and the only ways to create a triangle with $v$ on the side with one tail move are:
\begin{itemize}
\item suppressing one node of a four-cycle that includes $v$ to create a triangle by moving the outgoing edge of that node that is not included in the four-cycle. As this node is $p$, and $p$ is above both parents of $v$, the suppressed node must be on the incoming edge of $v$ in the four-cycle (Figure~\ref{fig:createTriangle}~top). However, in that case $v$ is a child of $p$ and $(v,p)$ is tail movable, so we choose to move $(v,p)$ up for the first move, which keeps $(u,v)$ head movable.
\item moving the other incoming edge of $v$ (not $(u,v)$) to the other incoming edge of the child $c(v)$ of $v$ (so not $(v,c(v))$). But as the tail move moves $(p,q)$ to $(j,k)$, we see that $k=c(v)$ which contradicts the fact that $v$ is strictly below $k$ in $N$. Hence this cannot result in a triangle with $v$ on the side (Figure~\ref{fig:createTriangle}~bottom~left).
\item moving the other incoming edge of the child $c(v)$ of $v$ (so not $(v,c(v))$) to the incoming edge of $v$ that is not $(u,v)$. As we move $(p,q)$ to $(j,k)$, we see that $v=k$ and $u=i$. But then $c(v)=q$ must be below the other child $m$ of $j$, and as $x_2$ is below $q$, this contradicts the fact that $(j,m)$ is not above $x_2$. Hence this cannot result in a triangle with $v$ on the side (Figure~\ref{fig:createTriangle}~bottom~right).
\end{itemize}

\begin{figure}[h!]
\begin{center}
\includegraphics[scale=0.7]{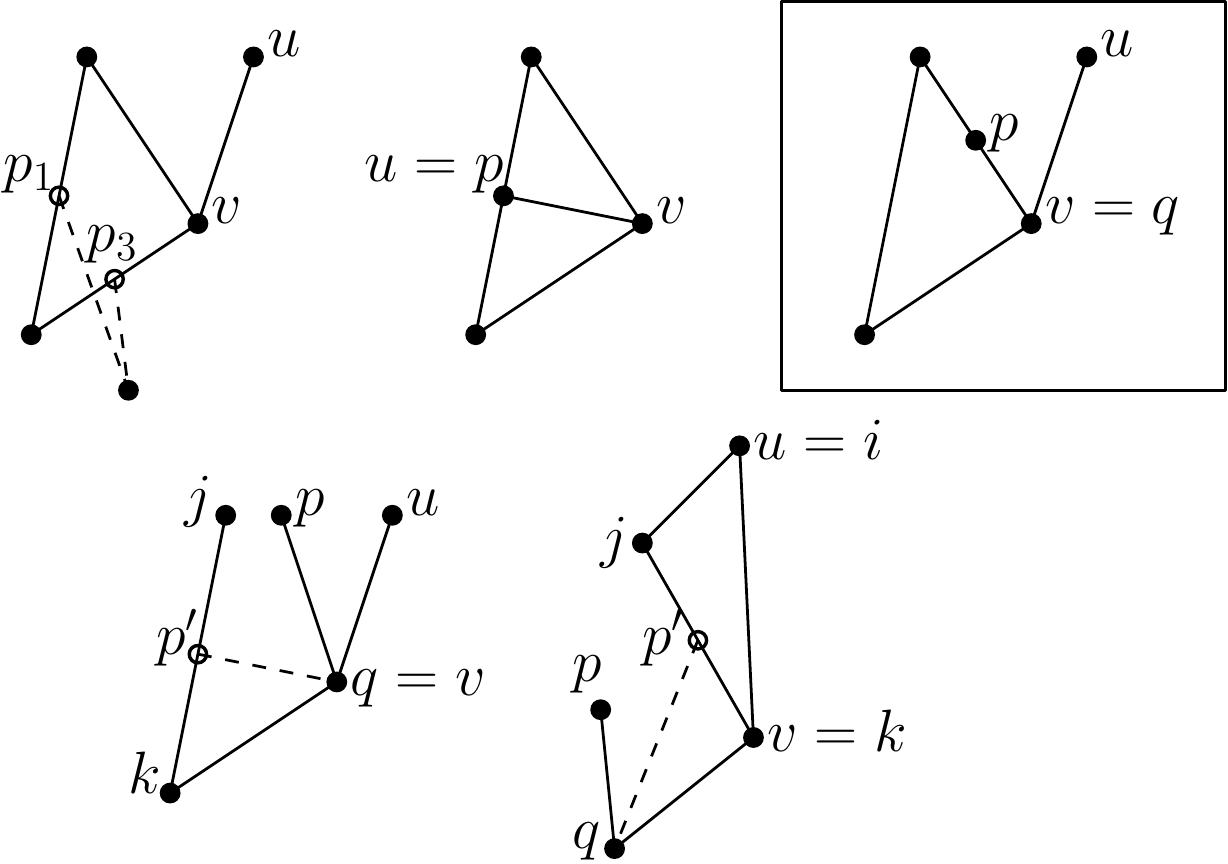}
\end{center}
\caption{The ways of making $(u,v)$ not head movable in Lemma~\ref{lem:makeMovable}. Top: creating a triangle by suppressing a node in a four cycle. The first two of these are invalid because $p$ is not above both parents of $v$. The right one does not give any contradictions, but forces us to choose to move $(p,v)$, so that no triangle is produced. Bottom: creating a triangle by moving an edge to become part of the triangle. Both these options contradict our assumptions.}
\label{fig:createTriangle}
\end{figure}

The preceding shows that $(u,v)$ is still head movable after the first tail move. Because $p$ is above $x_2$ through two paths, $y_1$ is still above $x_2$ after the tail move $(p,q)$ to $(i,j)$. Also we did not change $x_2$, so it still is a reticulation. This means that the head move $(u,v)$ to $(x_2,y_2)$ is still valid and of the right type. Furthermore $(j,m)$ is a tail movable edge with $m$ not above $x_2$. Now note that after the head head move $(u,v)$ to $(x_2,y_2)$, we can move $(p',q)$ back to its original position to obtain $N'$.
 
Hence we produce $M$ by tail-moving $(p,q)$ to $(i,j)$ and $M'$ by moving the corresponding edge to $(i,j)$ in $N'$. We can do this because $(i,j)$ is still an edge in $N'$: indeed it is not subdivided by the head move, and $i$ and $j$ are both split nodes, so they do not disappear either. So this case is proven.

\item{\textbf{$\bm{i}$ is not the child of the root.}} In this case we can move the tail of $(i,k)$ (possibly equal to $(u,v)$) up to the root in $N$. Now note that s $j$ is a split node, the tail move cannot create any triangles with a reticulation on the side. This means that $(u,v)$ is still movable after the tail move. Furthermore, after the tail move $x_2$ is still a reticulation node below $y_1$, and $(j,m)$ is movable and not above $x_2$. Hence the head move $(u,v)$ to $(x_2,y_2)$ is allowed and of the appropriate type. Now moving the tail of $(i',k)$ back to the incoming edge of $j$, we get $N'$.

Hence this case works with $M$ being the network obtained by moving $(i,k)$ up to the root edge in $N$, and $M'$ the network obtained by moving $(i,j)$ up to the root edge in $N'$.

\end{itemize}

\end{proof}

\begin{lemma}
Let $(u,v)$ from $(x_1,y_1)$ to $(x_2,y_2)$ be a valid head move in a network $N$ resulting in a network $N'$. Suppose that $y_1$ is above $x_2$ and $x_2$ is a reticulation, then there is a sequence of at most $8$ tail moves between $N$ and $N'$. 
\end{lemma}
\begin{proof}
By Lemma~\ref{lem:makeMovable}, with cost of 2 tail moves, we can assume there is a tail movable edge $(s,t)$ that can be moved to $(x_2,y_2)$. Make this the first move of the sequence. Because the head move $(u,v)$ to $(s',y_2)$ goes down, and $(u,v)$ is head-movable, this head move is allowed. By Lemma~\ref{lem:HeadMoveDown}, there is a sequence of at most 4 tail moves simulating this head move. Now we need one more tail move to arrive at $N'$: the move putting $(s,t)$ back to its original position. This all takes at most $8$ moves.
\end{proof}

All previous lemmas together give us the following result. 

\begin{proposition}\label{prop:MoveHeadDown}
Let $(u,v)$ from $(x_1,y_1)$ to $(x_2,y_2)$ be a valid head move in a network $N$ resulting in a network $N'$. Suppose that $y_1$ is above $x_2$ or $y_2$ is above $x_1$, then there is a sequence of at most $8$ tail moves between $N$ and $N'$. 
\end{proposition}

Now we continue with head moves where the original position of the head and the location it moves to are incomparable.

\begin{proposition}\label{prop:MoveHeadNOTDown}
Let $(u,v)$ from $(x_1,y_1)$ to $(x_2,y_2)$ be a valid head move in a network $N$ resulting in a network $N'$, where $N$ and $N'$ are not networks with two leaves and one reticulation. Suppose that $y_1$ is not above $x_2$ and $y_2$ is not above $x_1$, then there is a sequence of at most $16$ tail moves between $N$ and $N'$. 
\end{proposition}
\begin{proof}
Find an LCA $s$ of $x_1$ and $x_2$. We split into different cases for the rest of the proof:
\begin{enumerate}
\item{\textbf{$\bm{s\neq x_1,x_2}$.}}\label{Case:sNeqX1X2} One of the outgoing edges $(s,t)$ of $s$ is tail-movable and it is not above one of $x_1$ and $x_2$. Suppose $t$ is not above $x_1$, then we can do the following (Figure~\ref{fig:headToTailNonComparable}):
\begin{itemize}
\item Tail move $(s,t)$ to $(x_1,v)$;\\ allowed because $t\neq v$, $(s,t)$ movable, and $t$ not above $x_1$.
\item Tail move $(s',v)$ to $(x_2,y_2)$;\\allowed because $(x_1,t)\not\in N$: otherwise $x_1$ was the only LCA of $x_1$ and $x_2$; $y_1$ and hence $v$ is not above $x_2$; $(x_2,y_2)\neq(u,v)$.
\item Distance one head move $(u,v)$ to $(s'',y_2)$;\\No parallel edges by removal: if so, they are between $s''$ and $y_1=y_2$, but then the move actually resolves this; no parallel edges by placing: $u\neq s''$; no cycles: $y_2$ not above $u$, otherwise cycle in $N'$
\item Move $(s'',y_1)$ back up to $(x_1,t)$;\\Moving a tail up is allowed if the tail is movable.
\item Move $(s''',t)$ back up to its original position.\\Moving a tail up is allowed if the tail is movable.
\end{itemize}
As the head move used in this sequence is a distance-1 move, it can be simulated with at most 4 tail moves. Hence the sequence for this case takes at most 8 tail moves.

\begin{figure}[h!]
\begin{center}
\includegraphics[scale=0.7]{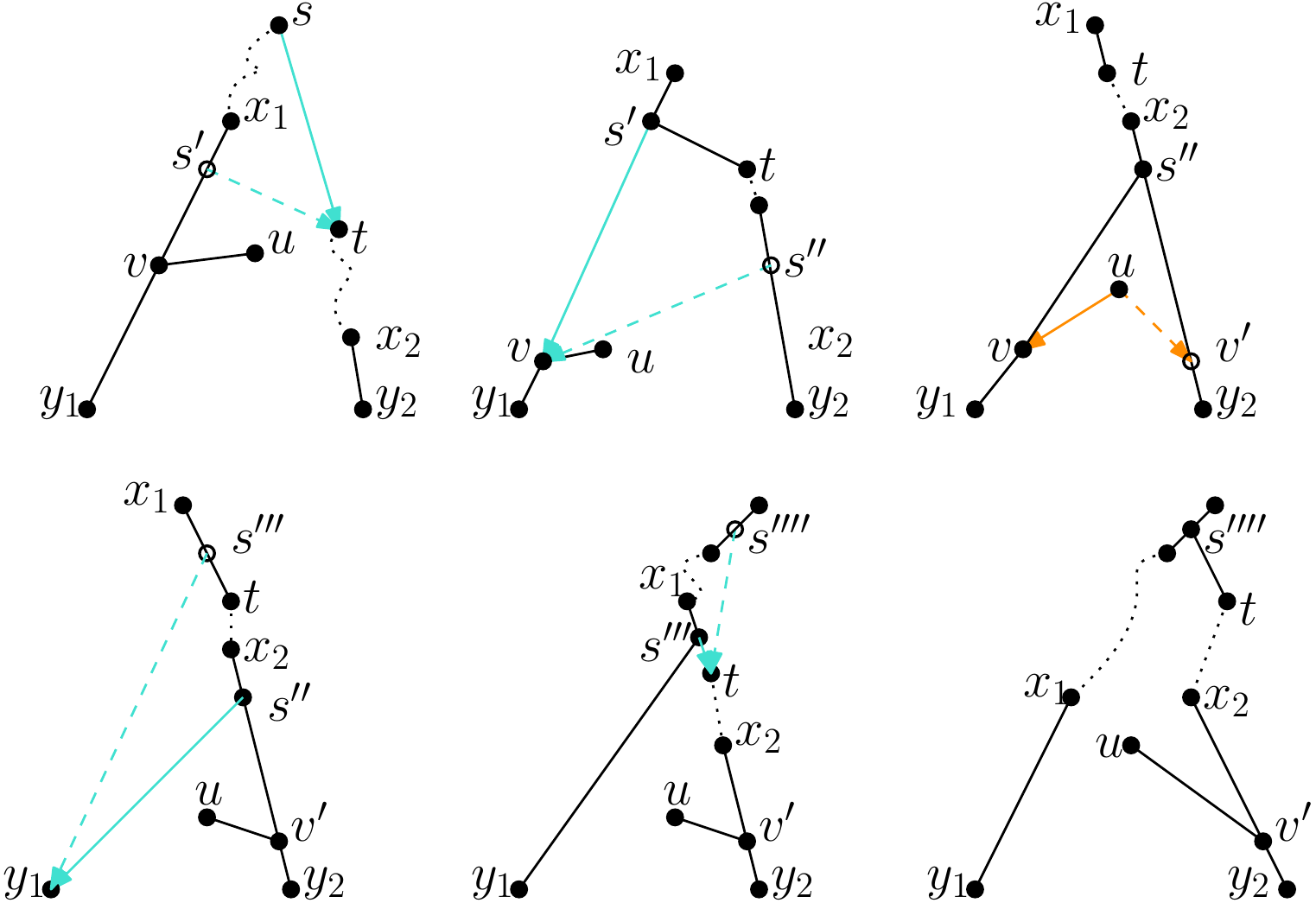}
\end{center}
\caption{The sequence of moves used in Case~\ref{Case:sNeqX1X2} of Proposition~\ref{prop:MoveHeadNOTDown}.}
\label{fig:headToTailNonComparable}
\end{figure}

 \item{\textbf{$\bm{s=x_1}$.}} 
  \begin{enumerate}
   \item{\textbf{$\bm{u}$ is not below $\bm{x_1}$.}}\label{case:HeadAlmostDown}
    Head move $(u,v)$ to the other child edge of $x_1$, this takes at most $4$ tail moves by Lemma~\ref{lem:D1HeadMove}. Now we have to move the head of $(u,v')$ down to create $N'$, this takes at most $8$ tail moves by Proposition~\ref{prop:MoveHeadDown}. Hence for this case we need at most $12$ tail moves.
   \item{\textbf{$\bm{u}$ is below $\bm{x_1}$.}} 
   In this case the previous approach is not directly applicable, as moving the head of $(u,v)$ to the other child edge of $x_1$ creates a cycle. Hence we need to take a different approach, where we distinguish the following cases:
   \begin{enumerate}
    \item{\textbf{$\bm{(x_1,v)}$ is tail movable.}}\label{case:HeadAlmostDownX1VMovable} Tail move $(x_1,v)$ down to $(x_2,y_2)$, this is allowed because $y_1$ is not above $x_2$. Then do the sideways distance one head move $(u,v)$ to $(x_1',y_2)$, this takes at most 4 tail moves. Then move $(x_1',y_1)$ back up to create $N'$. This takes at most $6$ moves
    \item{\textbf{$\bm{(u,v)}$ is tail movable.}} Move $(u,v)$ up to the incoming edge $(t,s)$ of $s$. The head move $(u',v)$ to $(x_2,y_2)$ is still allowed, except if $s,u,v$ form a triangle with the child of $u$ as well as of $v$ being $y_1$ in $N$, but in that case $x_1$ was not the LCA of $x_1$ and $x_2$. Hence we can simulate the head move with at most $12$ tail moves by Case~\ref{case:HeadAlmostDown} of this analysis. As afterwards we can move the tail of $(u',v')$ back to its original position, this case takes at most $16$ moves.
    \item{\textbf{Neither $\bm{(x_1,v)}$ nor $\bm{(u,v)}$ is tail movable.}}\label{case:HardHeadToTailLast}
    We create the situation of Case~\ref{Case:sNeqX1X2} by reversing the direction of the triangle at $x_1$, this takes at most $4$ tail moves because it is a distance one head move. Only if the bottom node of the triangle is $x_2$, we do not get this situation, but then the head move is composed of two distance one head moves, so it can be simulated with $8$ tail moves. If we are actually in the situation of Case~\ref{Case:sNeqX1X2}, simulate the head move with at most $8$ moves as done in that case. This is allowed because it produces $N'$ with the direction of a triangle reversed, which is a valid network. Then reverse the direction of the triangle again using at most $4$ tail moves. This way we obtain $N'$ with at most $16$ tail moves (Figure~\ref{fig:headToTailNonComparableLast}).

   \end{enumerate}
  \end{enumerate}

\item{\textbf{$\bm{s=x_2}$.}} This can be achieved with the reverse sequence for the previous case.

\end{enumerate}
\end{proof}

\begin{figure}[h!]
\begin{center}
\includegraphics[scale=0.7]{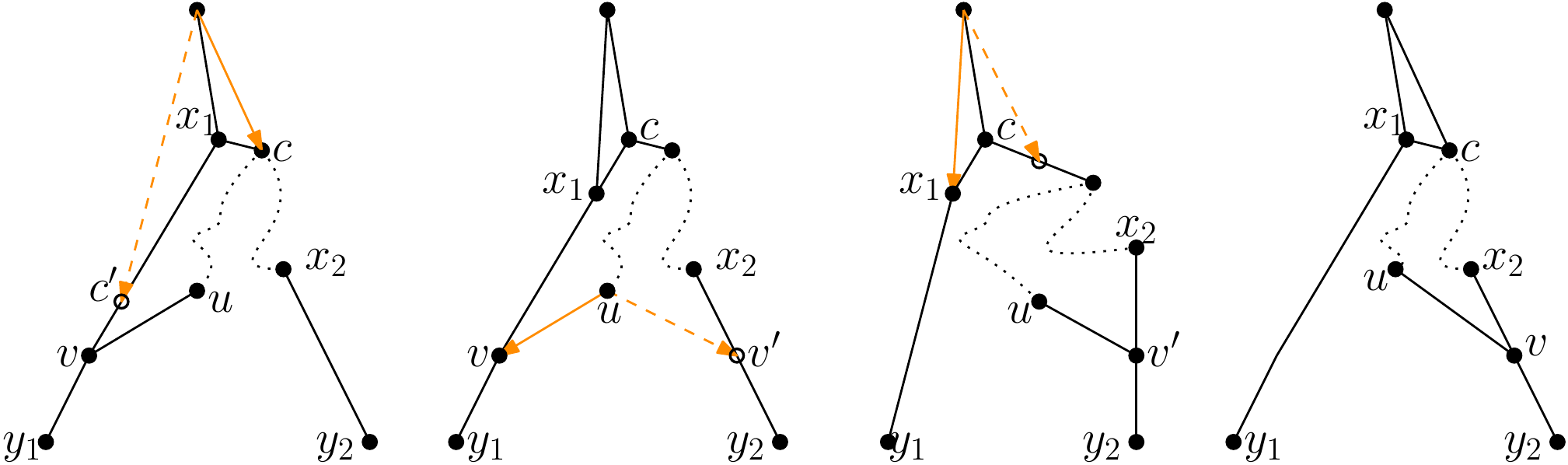}
\end{center}
\caption{The sequence of moves used in Case~\ref{case:HardHeadToTailLast} of Proposition~\ref{prop:MoveHeadNOTDown}.}
\label{fig:headToTailNonComparableLast}
\end{figure}

Proposition~\ref{prop:MoveHeadNOTDown} and Proposition~\ref{prop:MoveHeadDown} directly imply the following theorem
\begin{theorem}
Suppose there is a head move turning $N$ into $N'$, and $N$ and $N'$ are not non-isomorphic tier 1 networks with 2 leaves. Then there is a sequence of at most $16$ tail moves between $N$ and $N'$.
\end{theorem}


\section{Head move diameter and neighbourhoods}\label{sec:HeadDistance}
\subsection{Diameter bounds}
There are some obvious results concerning the diameter of head move space found using results from Section~\ref{sec:HeadToTailAndBack} and existing bounds on the rSPR diameter. Each rSPR sequence of length $l$ can be replaced by a sequence of head moves of length at most $15l$, hence we get upper bounds $\Delta^k_{\mathrm{Head}}\leq 15\Delta^k_{\mathrm{rSPR}}$ on head move diameters. Furthermore, each sequence of head moves is also an rSPR sequence, hence the rSPR diameter gives lower bounds $\Delta^k_{\mathrm{rSPR}}\leq\Delta^k_{\mathrm{Head}}$. Similarly the rSPR bounds give bounds on the tail move diameters. 

These bounds for tail move diameters are inferior to the bounds in \cite{janssen2017exploring}. The tail move diameter bounds from that paper are obtained using a technique where an isomorphism is built incrementally. For head moves we can employ a similar technique. For any pair of networks, we build an isomorphism between growing subnetworks where in each step we only have to use a small number of moves to grow the isomorphism. For tail and rSPR moves it was convenient to build this isomorphism bottom-up; because head moves are essentially upside down tail moves, here we build an isomorphism top-down.

In this section, each move is a head move, unless stated otherwise. As we need to explicitly work with the vertices and arcs of different networks, we denote a network with nodes $V$ and arcs $A$ as $N=(V,A)$. We first define a few structures that we use extensively: upward closed sets, isomorphisms, and induced graphs.

\begin{definition}
Let $N=(V,A)$ be a network with $Y\subseteq V$ a subset of the vertices. We say that $Y$ is \emph{upward closed} if for each $u\in Y$ the parents of $u$ are also in $Y$. 
\end{definition}

\begin{definition}
Let $N=(V,A)$ and $N'=(V',A')$ be two directed acyclic graphs, then a map $\phi:V(N)\to V(N')$ is an \emph{(unlabelled) isomorphism} if $\phi$ is bijective and $(u,v)\in A$ if and only if $(\phi(u),\phi(v))\in A(N')$. If such an isomorphism exists, we say that $N$ and $N'$ are \emph{(unlabelled) isomorphic}. If additionally, there are labellings $l_N:X\to V(N)$ and $l_{N'}:X\to V(N')$ of the vertices and $\phi(l_N(x))=l_{N'}(x)$ for all $x\in X$ then $N$ and $N'$ are \emph{labelled isomorphic}.
\end{definition}

\begin{definition}
Let $N=(V,A)$ be a network and $Y\subseteq V$ a subset of the vertices, then $N[Y]$ denotes the directed subgraph of $N$ induced by $Y$: 
\[N[Y]:=(Y,A\cap (Y\times Y))\]
\end{definition}

\begin{lemma}\label{lem:DiameterProof}
Let $N_1$ and $N_2$ be tier $k>0$ networks with label set $X$ of size $n$, then there exists a pair of head move sequences $S_1$ on $N_1$ and $S_2$ on $N_2$ such that the resulting networks are unlabelled isomorphic and the total length is $|S_1|+|S_2|\leq 4n+6k-4$.
\end{lemma}
\begin{proof}
We incrementally build upward closed sets $Y_1\subseteq V(N_1)$ and $Y_2\subseteq V(N_2)$ such that $N_1[Y_1]$ and $N_2[Y_2]$ are unlabelled isomorphic with isomorphism $\phi$. Starting with $Y_1=\{\rho_1\}$ and $Y_2=\{\rho_2\}$ the roots only, we set the isomorphism $\rho_1\mapsto \rho_2$. Next we increase the size of $Y_1$ by changing the networks slightly with a constant number of head moves, and then adding a node to $Y_1$ and $Y_2$ and extending the isomorphism. We will add all the leaves to the isomorphism last.
\begin{enumerate}
 \item\label{case:treeNodeInN} \textbf{There is a highest node $\bm{x_1}$ of $\bm{N_1}$ not in $\bm{Y_1}$ such that $\bm{x_1}$ is a split node.} Because $x_1$ is a highest node not in $Y_1$, the parent $p_1$ of $x_1$ is in $Y_1$ and there is a corresponding node $p_2:=\phi(p_1)$ in $Y_2$. This node must have at least one child $x_2$ that is not in $Y_2$, as otherwise the degrees of $p_1$ and $p_2$ in $N_1[Y_1]$ and $N_2[Y_2]$ do not coincide.
  \begin{enumerate}
    \item \textbf{The node $\bm{x_2}$ is a split node.} In this case we can add $x_1$ and $x_2$ to $Y_1$ and $Y_2$ and set $\phi:x_1\mapsto x_2$ to get an extended isomorphism. We do not have to use any head moves to do this extension. 
    \item \textbf{The node $\bm{x_2}$ is a reticulation.} We make sure $p_2$ has a split node $y_2$ as a child not in $Y_2$, using at most 3 head moves. We can then add $x_1$ to $Y_1$ and $y_2$ to $Y_2$ and extend the isomorphism with $x_1\mapsto y_2$. To create this split node, we use a split node $c_2\in N_2\setminus Y_2$, which exists because there is a split node in $N_1\setminus Y_1$.
    \begin{enumerate}
      \item \textbf{The edge $\bm{(p_2,x_2)}$ is movable.}\label{case:DiameterTreeRetic} Move $(p_2,x_2)$ to the incoming edge $(t_2,c_2)$ of the split node $c_2$. This move is valid because $c_2$ cannot be above $p_2$ (otherwise $c_2\in Y_2$, a contradiction), and $t_2\neq p_2$ as otherwise $p_2$ would have a split node child not in $Y_2$. Now the edge $(t_2,x_2')$ is movable to any of the outgoing edges of $c_2$. Now $p_2$ has child node $c_2$, which is a split node, so we can extend the isomorphism with a split node $\phi:x_1\mapsto c_2$ using at most $2$ head moves.
      
      \begin{figure}[h!]
		\begin{center}
		\includegraphics[scale=0.5]{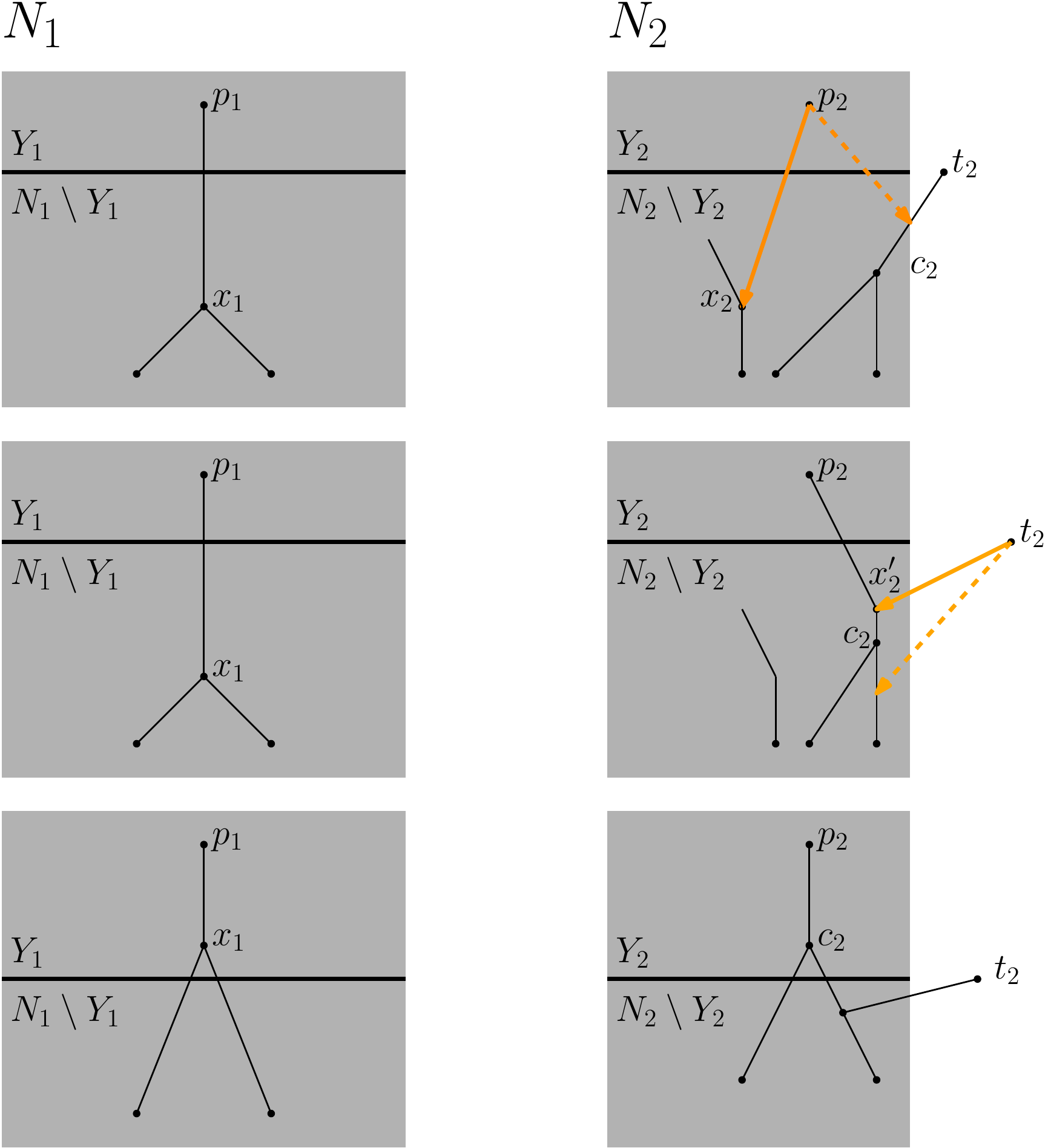}
		\end{center}
	  \caption{The moves and incremented isomorphism for Lemma~\ref{lem:DiameterProof}~Case~\ref{case:DiameterTreeRetic}. For nodes outside of the shaded region, it is not known whether they are in $Y_2$.}
	  \label{fig:DiameterTreeRetic}
	  \end{figure}
      
      \item \textbf{The edge $\bm{(p_2,x_2)}$ is not movable.} This means that $x_2$ is on the side of a triangle. Denote by $d_2$ the child of $x_2$ and the other parent of $x_2$ with $z_2$. Now note that $(z_2,d_2)$ is movable, and can be moved to an edge $(u_2,v_2)$ with $v_2$ not in $Y_2$ and $(u_2,v_2)$ distinct from both $(x_2,d_2)$ and from the outgoing edge of $d_2$. Such an edge exists: pick a leaf $l$ not equal to the child of $d_2$ (if that node is a leaf); as we add all leaves to the isomorphism last, the leaf is not in $Y_2$, furthermore, $l$ is not above $z_2$, and the incoming edge of $l$ is not equal to $(x_2,d_2)$ nor to the outgoing edge of $d_2$. 
      Doing the head move $(z_2,d_2)$ to the incoming edge of $l$ creates the situation of the previous case (Case~\ref{case:DiameterTreeRetic}), and we can use 2 more head moves to create a network with a split node $c_2$ below $p_2$ which maintains the isomorphism of the upper part $Y_2$. Hence we can extend the isomorphism with a split node $\phi: x_1\mapsto c_2$ using at most $3$ head moves.
    \end{enumerate}
    \item \textbf{The node $\bm{x_2}$ is a leaf.} Again, note there is a split node $c_2$ in $N_2\setminus Y_2$, and let its parent be $t_2$. Note also that $N_2$ has a reticulation node $r_2$ with incoming edge $(s_2,r_2)$ which is movable to $(p_2,x_2)$ (if $p_2=s_2$, then the other incoming edge $(s_2',r_2)$ is also movable, and can instead be moved to $(p_2,x_2)$). 
     \begin{enumerate}
      \item \textbf{The nodes $\bm{s_2}$ and $\bm{t_2}$ are different nodes.}\label{case:DiameterTreeLeaf}
        First move $(s_2,r_2)$ to $(p_2,x_2)$. Now the edge $(p_2,r_2')$ is movable, and can be moved to $(t_2,c_2)$, because $c_2$ is not above $p_2$ and $p_2\neq t_2$ (otherwise $p_2$ has a split node as child). This makes $(t_2,r_2'')$ movable, and we can move it to $(s_2,x_2)$ because $s_2\neq t_2$ and $x_2$ is a leaf, so it is not above $t_2$. Now lastly, we restore the reticulation by moving $(s_2,r_2''')$ back to its original position. Hence in this situation, $4$ head moves suffice to make $p_2$ the parent of a split node $c_2$, so that we can extend the isomorphism by $\phi:x_1\mapsto c_2$ with a split node.
        
       \begin{figure}[h!]
		\begin{center}
		\includegraphics[scale=0.5]{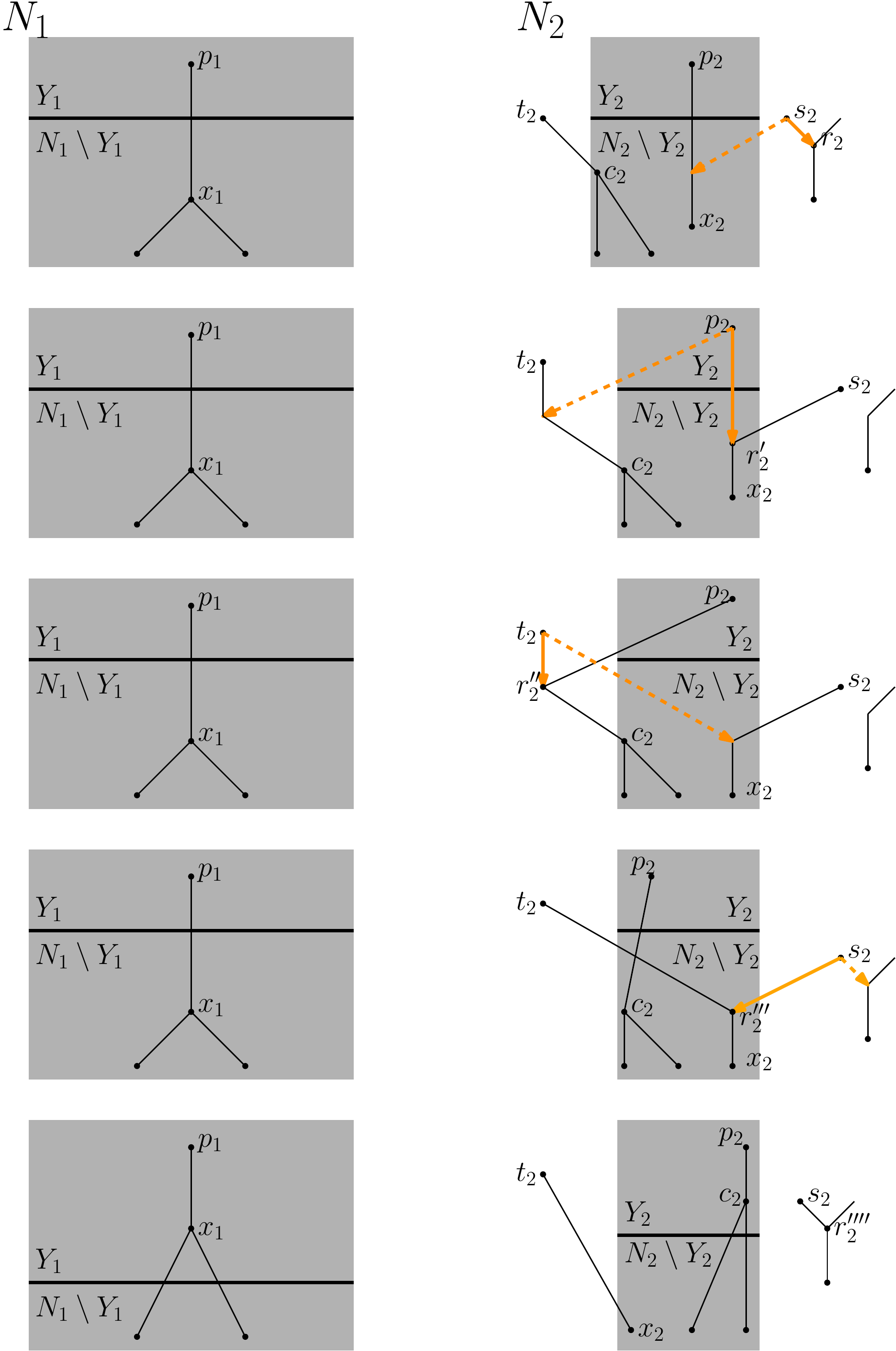}
		\end{center}
	  \caption{The moves and incremented isomorphism for Lemma~\ref{lem:DiameterProof}~Case~\ref{case:DiameterTreeLeaf}. For nodes outside of the shaded region, it is not known whether they are in $Y_2$.}
	  \label{fig:DiameterTreeLeaf}
	  \end{figure}
        
      \item \textbf{The nodes $\bm{s_2}$ and $\bm{t_2}$ are the same.} Note that a child of $t_2$ is a split node and a child of $s_2$ is a reticulation. This means that $s_2=t_2$ is a split node, as it has two distinct children. 
      The edge $(s_2,r_2)$ can be moved to the pendant edge $(p_2,x_2)$. Now the new edge $(p_2,r_2')$ can be moved to $(s_2,c_2)$, because $p_2\neq s_2$ and $c_2$ is not above $p_2$ (otherwise $c_2$ has to be in $Y_2$, contradicting our assumption). Now we can move $(s_2,r_2'')$ back to its original position. This all takes three head moves, and makes sure that a child $c_2$ of $p_2$ is a split node. This means we can extend the isomorphism by setting $\phi:x_1\mapsto c_2$ (and if $r_2$ was in $Y_2$, changing $\phi:\phi^{-1}(r_2)\mapsto r_2$ to $\phi:\phi^{-1}(r_2)\mapsto r_2'''$) using at most $3$ head moves to add a split node
     \end{enumerate}
  \end{enumerate}
 \item \textbf{There is a highest node $\bm{x_2}$ of $\bm{N_2}$ not in $\bm{Y_2}$ such that $\bm{x_2}$ is a split node.} Do the same as in the previous case (Case~\ref{case:treeNodeInN}) switching the roles of $N_1$ and $N_2$.
 \item \textbf{Each highest node $\bm{x_1}$ of $\bm{N_1}$ not in $\bm{Y_1}$ and $\bm{x_2}$ of $\bm{N_2}$ not in $\bm{Y_2}$ is a reticulation node or a leaf.} 
  \begin{enumerate}
   \item \textbf{There exists a highest node $\bm{x_1}$ of $\bm{N_1}$ not in $\bm{Y_1}$ which is a reticulation node.} This means the two parents $p_1$ and $q_1$ of $x_1$ are in $Y_1$, and consequently have corresponding nodes $p_2$ and $q_2$ in $Y_2$. Both these nodes also have at least one child not in $Y_2$, say $c^p_2$ and $c^q_2$.
   \begin{enumerate}
    \item \textbf{The children of $p_2$ and $q_2$ are equal (i.e., $\bm{c^p_2=c^q_2}$).} In this case, we can immediately extend the isomorphism with $\phi:x_1\mapsto c^p_2$.
    \item \textbf{Both nodes $\bm{c^p_2}$ and $\bm{c^q_2}$ are reticulations.}\label{case:ReticReticsBoth} Assume without loss of generality that $\bm{c^p_2}$ is not below $\bm{c^q_2}$.    
     \begin{enumerate}
      \item \textbf{The edge $\bm{(p_2,c^p_2)}$ is movable.} Move this edge to $(q_2,c^q_2)$, which is allowed because $c^q_2$ is not above $p_2$, and $p_2\neq q_2$. Now $p_2$ and $q_2$ have a common child $x_2:=c'^p_2$, so we can add one reticulation to $Y_1$ and $Y_2$ and extend the isomorphism by $\phi:x_1\mapsto x_2$ using $1$ head move.
      \item \textbf{The edge $\bm{(p_2,c^p_2)}$ is not movable.}\label{case:DiameterReticRetics} Because $(p_2,c^p_2)$ is not movable, $c^p_2$ must be the side node of a triangle, and therefore its outgoing edge $(c^p_2,z)$ is movable. By our assumption, $c^q_2$ is not above $c^p_2$, so we can move $(c^p_2,z)$ to $(q_2,c^q_2)$. Now the other incoming edge $(t,c^p_2)$ of $c^p_2$ becomes movable, and we can move it down to $(z',c^q_2)$. Now $p_2$ and $q_2$ have a common child $x_2:=z'$, and the isomorphism can be extended with one reticulation by setting $\phi:x_1\mapsto x_2$ using at most 2 head moves. 
      
      \begin{figure}[h!]
		\begin{center}
		\includegraphics[scale=0.5]{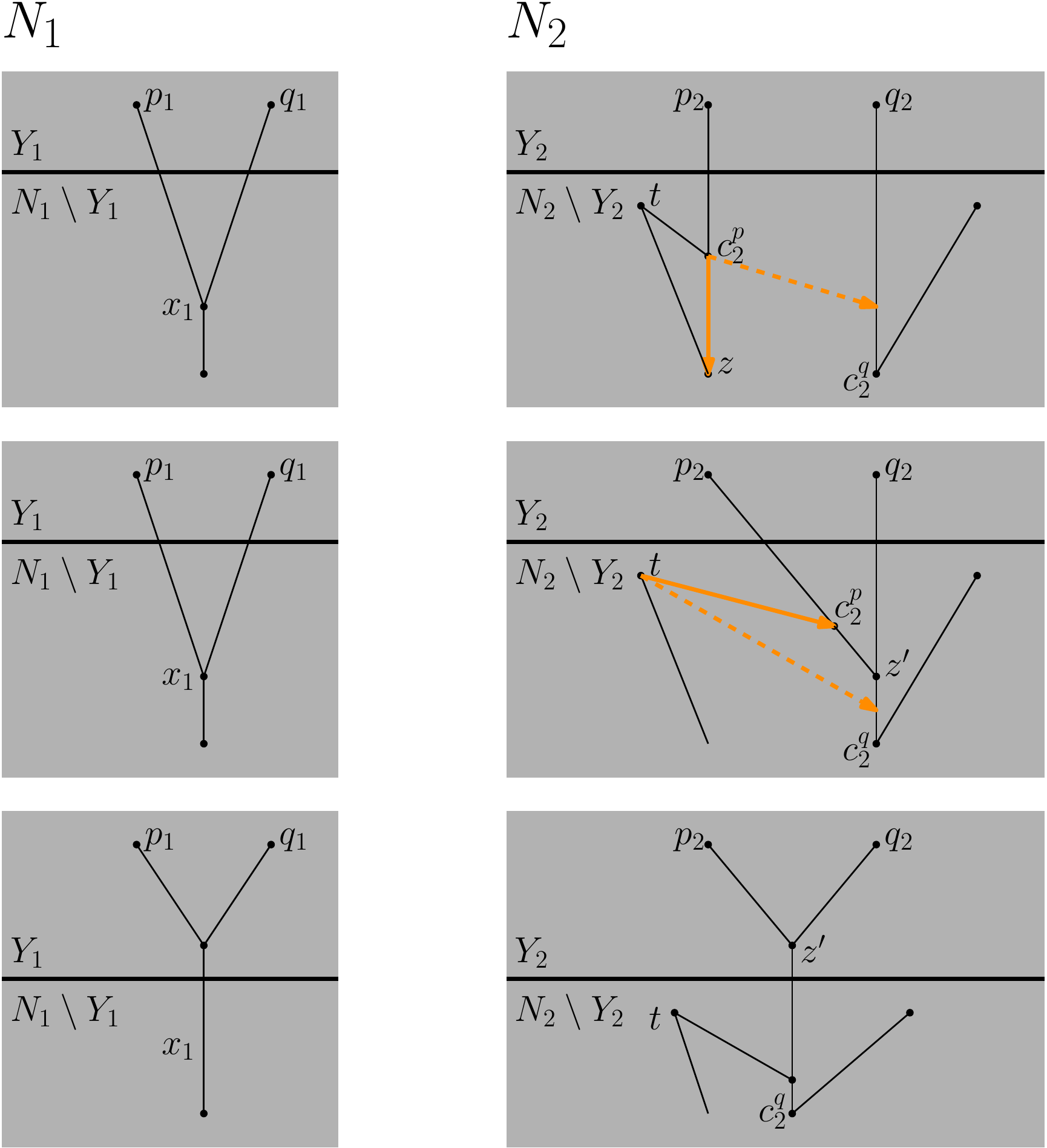}
		\end{center}
	  \caption{The moves and incremented isomorphism for Lemma~\ref{lem:DiameterProof}~Case~\ref{case:DiameterReticRetics}.}
	  \label{fig:DiameterReticRetics}
	  \end{figure}

     \end{enumerate}
    \item \textbf{The node $\bm{c^p_2}$ is a reticulation, and $\bm{c^q_2}$ is a leaf.} The subcases here work exactly like the previous subcases in Case~\ref{case:ReticReticsBoth}.
      \begin{enumerate}
        \item \textbf{The edge $\bm{(p_2,c^p_2)}$ is movable.} Move this edge to $(q_2,c^q_2)$, which is allowed because $c^q_2$ is not above $p_2$, and $p_2\neq q_2$. Now $p_2$ and $q_2$ have a common child $x_2:=c'^p_2$, so we can add one reticulation to $Y_1$ and $Y_2$ and extend the isomorphism by $\phi:x_1\mapsto x_2$ using one head move.
        \item \textbf{The edge $\bm{(p_2,c^p_2)}$ is not movable.} Because $(p_2,c^p_2)$ is not movable, $c^p_2$ must be the side node of a triangle, and therefore its outgoing edge $(c^p_2,z)$ is movable. Because $c^q_2$ is a leaf, it is not above $c^p_2$, so we can move $(c^p_2,z)$ to $(q_2,c^q_2)$. Now the other incoming edge $(t,c^p_2)$ of $c^p_2$ becomes movable, and we can move it down to $(z',c^q_2)$. Now $p_2$ and $q_2$ have a common child $x_2:=z'$, and the isomorphism can be extended with one reticulation by setting $\phi:x_1\mapsto x_2$ using at most 2 head moves. 
     \end{enumerate}
    \item \textbf{The node $\bm{c^q_2}$ is a reticulation, and $\bm{c^p_2}$ is a leaf.} Switch the roles of $p_2$ and $q_2$ and do as in the previous case.
    \item \textbf{Both nodes $\bm{c^p_2}$ and $\bm{c^q_2}$ are leaves.} Note that because $x_1$ is a reticulation node not in $Y_1$, there must also be a reticulation node $r_2\in N_2$ not in $Y_2$. Let its movable incoming edge be $(s_2,r_2)$. As $p_2\neq q_2$ we know that $s_2$ can be equal to at most one of $p_2$ and $q_2$, hence we can assume without loss of generality that $s_2\neq p_2$. Then the head move $(s_2,r_2)$ to $(p_2,c^p_2)$ is allowed, because the leaf $c^p_2$ cannot be above $s_2$. Now $(p_2,r_2')$ is movable because the child of $r_2'$ is a leaf, and it can be moved to $(q_2,c^q_2)$ because $p_2\neq q_2$ and $c^q_2$ is a leaf, and hence not above $p_2$. After this head move, $p_2$ and $q_2$ have a common child $x_2:=r_2''$, and the isomorphism can be extended with one reticulation by setting $\phi:x_1\mapsto x_2$ using at most 2 head moves. 
    \end{enumerate}
   \item \textbf{There exists a highest node $\bm{x_2}$ of $\bm{N_2}$ not in $\bm{Y_2}$ which is a reticulation node.} Do the same as in the previous case, switching the roles of $N_1$ and $N_2$.
   \item \textbf{All highest nodes of $N_1$ not in $Y_1$ and of $N_2$ not in $Y_2$ are leaves.}
    In this case, the networks are already unlabelled isomorphic: $N_1[Y_1]$ and $N_2[Y_2]$ are isomorphic, and the only nodes not part of the isomorphism are leaves, hence there is only one way (ignoring symmetries of cherries) to complete the isomorphism.
  \end{enumerate}
 \end{enumerate}

Note that this procedure first adds all split nodes and reticulations to the isomorphism, using four moves per split node and two moves per reticulation node at most. Then finally it adds all the leaves, without changing the networks any more. Noting that the number of split nodes is $n+k-1$, we see that we need to do at most $4(n+k-1)+ 2k= 4n+6k-4$ moves in $N_1$ and $N_2$ to get $N_1'$ and $N_2'$ which are unlabelled isomorphic. 
\end{proof}

\begin{lemma}
Let $N$ and $N'$ be tier $k>0$ networks with label set $X$ of size $n$, which are unlabelled isomorphic. Then there is a head move sequence from $N$ to $N'$ of length at most $2n$.
\end{lemma}
\begin{proof}
Note that the only difference between $N$ and $N'$ is a permutation of the leaves, say $\pi=(l_1^1,\ldots,l_{\Pi_1}^1)(l_{1}^2,\ldots,l_{\Pi_2}^2)\cdots(l_{1}^P,\ldots,l_{\Pi_q})$ to get from $N$ to $N'$ (where all $l_i^j$ are distinct). Note also that there is a reticulation in $N$ with a head movable edge $(t,r)$, which is movable to the incoming edge of any leaf. A sequence of moves from $N$ to $N'$ consists of the moves
\begin{itemize}
\item $(t,r)$ to $(p(l_{\Pi_j}^j),l_{\Pi_j}^j)$;
\item $(p(l_{\Pi_j}^j),r^{(1)'})$ to $(p(l_{\Pi_{j}-1}^j),l_{\Pi_{j}-1}^j)$;
\item $(p(l_{\Pi_{j}-1}^j),r^{(2)'})$ to $(p(l_{\Pi_{j}-2}^j),l_{\Pi_{j}-2}^j)$;
\item $\ldots$
\item $(p(l_{2}^j),r^{(\Pi_j-1)'})$ to $(p(l_{1}^j),l_{1}^j)$;
\item $(p(l_1^j),r^{(\Pi_j)'})$ to $(t,l_{\Pi_j}^j)$;
\item $(t,r^{(\Pi_j+1)'})$ to $(s,c)$,
\end{itemize}
for each cycle ($1\leq j \leq q$) of $\pi$, where $c$ is the child of $r$ in $N$ and $s$ is the other parent of $r$ in $N$. This permutes the leaves in $N$ by $\pi$ so that the resulting network is $N'$. The sequence is allowed provided no two subsequent leaves in a cycle have a common parent (e.g., $p(l_{i}^j)=p(l_{i-1}^j)$). There is always a permutation in which this does not happen, as if this happens, the two leaves are in a cherry. The worst case is attained when there are a maximal number of cycles in the permutation, which happens when $\pi$ consists of only 2-cycles. In such a case there will be $n/2$ cycles of length $2$. Each such a cycle takes four moves. An upper bound to the length of the sequence is therefore $4(n/2)=2n$. 
\end{proof}

A direct corollary of the previous two lemmas is the following theorem, giving an upper bound on the diameter of head move space. To see this, note that any head move is reversible, and hence we can concatenate sequences in different directions.

\begin{theorem}
Let $N$ and $N'$ be tier $k>0$ networks with label set $X$ of size $n$, then there is a head move sequence of length at most $6n+6k-4$ between $N$ and $N'$.
\end{theorem}

\subsection{Neighbourhood size} 
For local search strategies in phylogenetic network space, it is relevant to know the size of the neighbourhood of a network. This means we should consider the size of head move neighbourhoods. Head moves can only move reticulation arcs of which there are $2k$ in a tier-$k$ network. Furthermore, there are $2n+3k-1$ edges in a tier-$k$ network with $n$ leaves. Hence an upper bound on the head move neighbourhood size in a tier-$k$ network with $n$ leaves is $4kn+6k^2-2k$, i.e. of order $O(kn+k^2)$. We will now compare this with known bounds for neighbourhood sizes for other rearrangement moves.

For a lower bound for the tail move neighbourhood size we turn to Proposition~4.1 from a paper by Klawitter (\cite{klawitter2017snpr}) about SNPR neighbourhoods. Because SNPR moves are tail moves together with vertical moves, the sizes of SNPR neighbourhoods and tail move neighbourhoods can easily be compared. 

\begin{proposition}[\cite{klawitter2017snpr}~Proposition~4.1]
Let $n\geq 2$, then
\[
\begin{aligned}
n-1&\leq \min_{N\in\mathcal{TC}_n}\{|U^{\mathcal{TC}_n}_{SNPR}(N)|\}&\leq \frac{3}{2}n^2-\frac{7}{2}n+2,& \text{ and}\\
8n^2-O(n\log_2 n)&\leq \max_{N\in\mathcal{TC}_n}\{|U^{\mathcal{TC}_n}_{SNPR}(N)|\}&\leq 16n^2-38n+26.
\end{aligned}
\]
\end{proposition}

Note that the SNPR neighbourhood also includes SNPR+ and SNPR- neighbours: networks obtained by removing or adding a reticulation edge. 
The lower bound for the minimal neighbourhood size only considers SNPR+ and SNPR- moves, and is therefore not useful to us. The network they consider (the network with 2 leaves and one reticulation) has 1 tail move neighbour, itself. So technically, this could be seen as a lower bound, but it is irrelevant when we want to consider lower bounds for networks with $n$ leaves and $r$ reticulations: the minimal number of neighbours might grow very quickly.

As this is the only work on neighbourhood sizes to date, we only have a useful upper bound on the minimal neighbourhood size. 
The example network that proves this upper bound has exactly $n-1$ SNPR+/- neighbours, which means it has $(3/2)n^2-(9/2)n+3$ tail neighbours. The author states that this is a network with very little SNPR moves, so it seems that the number of tail move neighbours of a network is quadratic in the number of leaves.

Contrast the preceding with our upper bound on the number of head move neighbours. Note that the head move neighbourhood is considerably smaller if the number of reticulations is small. This is surprising considering that the diameter of the spaces defined by head and tail moves is of the same order of magnitude.


\section{Hardness of computing head move distance}\label{sec:NP-Hard}
In this section, we prove that the problem {\sc Head Distance} of computing the head move distance between two networks is NP-hard. The proof uses a reduction from {\sc rSPR Distance}, which is the problem of finding the rSPR distance between two rooted trees. The rough idea is to convert rSPR moves into head moves. 

Because rSPR moves change the location of the tail and not the head of an arc, we have to use a trick: we turn the tree upside down, which turns each tail into a head, and hence a tail move into a head move. Just reversing the direction of the arrows of the tree is not sufficient, as this gives a graph with multiple roots and one leaf. Hence we connect all these roots and add a second leaf to create a phylogenetic network. This construction is formalized in the following definitions. 

After these definitions, we will show that the minimal number of head moves between two upside down trees is equal to the number of rSPR moves between the two original trees. This proof uses the concepts of agreement forests.

\begin{definition}
Let $X=\{x_1,\ldots,x_n\}$ be an ordered set of labels. The caterpillar $C(X)$ is the tree defined by the Newick string
\[(\cdots(x_1,x_2),x_3)\cdots,x_n);\]
\end{definition}

\begin{definition}
Let $T$ be a phylogenetic tree with labels $X=\{x_1,\ldots,x_n\}$, the \emph{upside down version} of $T$ is a network $\ud{T}$ with $2n^2+2$ leaves ($e_{x,i}$ for $x\in X$ and $i\in[2n]$, $y$, and $\rho$) constructed by:
\begin{enumerate}
\item Creating the labelled digraph $S$, which is $T$ with all the edges reversed;
\item Creating the tree $D$ by taking $C(X\cup\{y\})$ and adding $2n$ pendant edges with leaves labelled $e_{x,1},\ldots,e_{x,2n}$ to each pendant edge $e=(\cdot,x)$ of $C(X)$;
\item Taking the disjoint union of $D$ and $S$;
\item identifying the node labelled $x_i$ in $D$ with the node labelled $x_i$ in $S$ and subsequently suppressing this node for all $i$.
\end{enumerate}
The \emph{bottom part} of $\ud{T}$ is the subgraph of $\ud{T}$ below (and including) the parents of the $e_{x,1}$.
\end{definition}

\begin{figure}[h!]
\begin{center}
\includegraphics[scale=0.6]{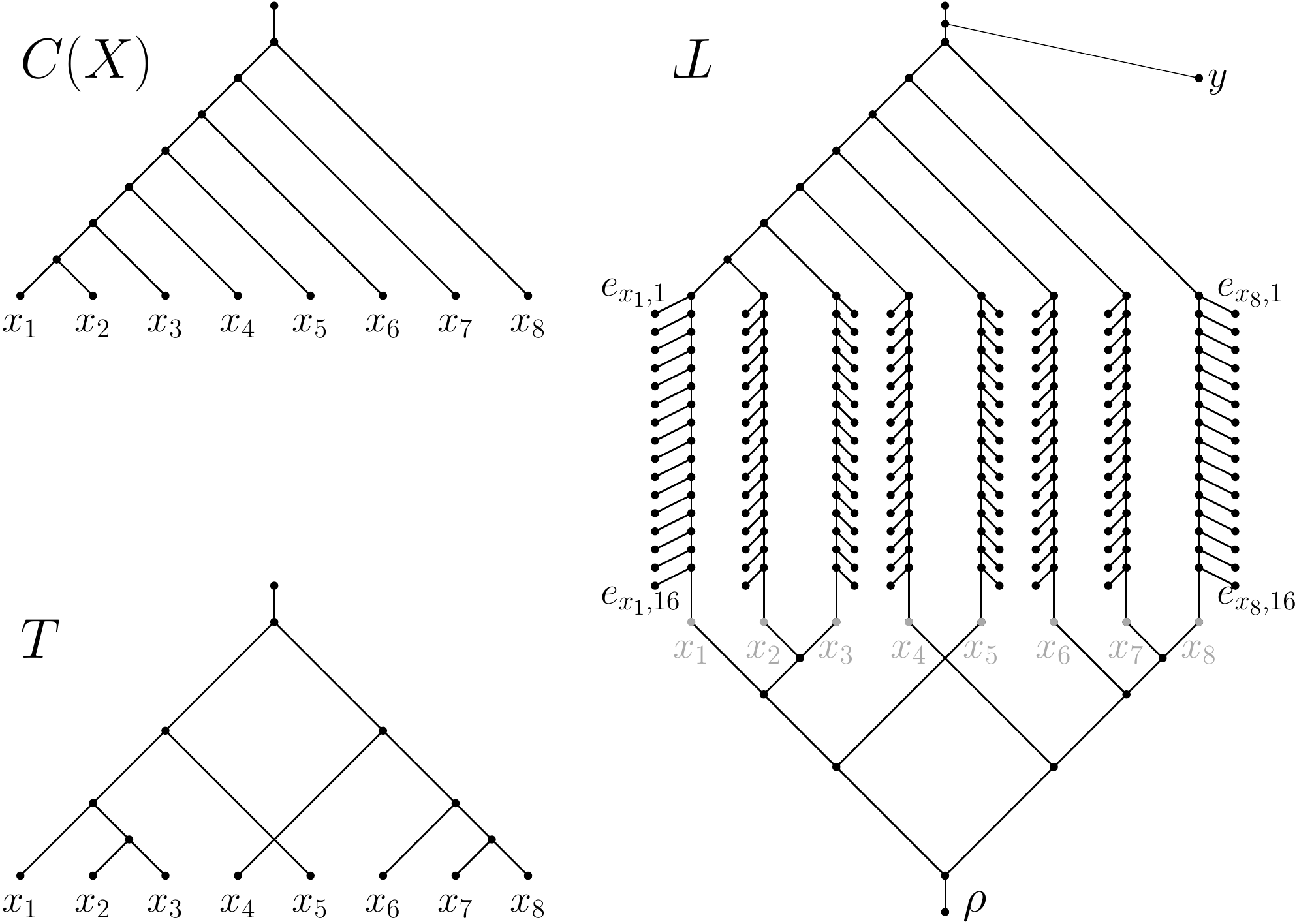}
\end{center}
\caption{Left the caterpillar $C(X)$ and the tree $T$, right, the upside down version of $T$. In the upside down version of $T$, the original leaves $x_i$ (grey) are suppressed.}
\label{fig:udTree}
\end{figure}

The rSPR distance between two trees can be characterized alternatively as the size of an agreement forest \citep{bordewich2005computational}. Here we use this alternative description as part of the reduction. To define agreement forests, we need the following definitions, which we have generalized slightly to work well for networks.

\begin{definition}
Let $T$ be a tree (for digraphs: the underlying undirected graph is a tree) with its degree-1 nodes labelled bijectively with $X$ and let $Y\subseteq X$ be a subset of the labels. Then $T|_Y$ is the \emph{subtree of $T$ induced by $Y$}; that is, it is the union of all shortest (undirected) paths between nodes of $Y$.
\end{definition}

\begin{definition}
Let $G$ and $G'$ be labelled digraphs. Suppose $G$ and $G'$ are labelled isomorphic after suppression of all their \emph{redundant nodes} (indegree-1 outdegree-1 nodes), then we write $G\equiv G'$, or say $G$ is \emph{s-isomorphic} to $G'$ (for suppressed isomorphic). 

An \emph{embedding} of a graph $H$ in $G$ is an \emph{s-isomorphism} $H\equiv S$ of $H$ with a subgraph $S$ of $G$. We say that $H$ \emph{can be embedded} in $G$ if an embedding of $H$ in $G$ exists. Note that any subgraph $H$ of $G$ can be embedded in $G$ as $H\equiv H$.
\end{definition}

Now we look at an important property of embeddings relating to subgraphs, which implies that being embeddable is transitive.

\begin{lemma}\label{lem:subgraphEmbedding}
Let $A$, $B$ and $H$ be digraphs with all degree-1 nodes labelled. Suppose $A\equiv B$ and $H$ is a subgraph of $A$, then $H$ can be embedded in $B$.
\end{lemma}
\begin{proof}
The s-isomorphism $A\equiv B$ is an isomorphism of graphs (topological minors) without redundant nodes. This isomorphism is a bijection between the non-redundant nodes of $A$ to the non-redundant nodes of $B$. The map of the edges is a map of paths of $A$ to paths of $B$, where the internal nodes of these paths may only be redundant nodes. Now consider the subgraph $H$ of $A$, and note that the non-redundant nodes of $H$ are non-redundant nodes of $A$ as well. Indeed, the only way to create new non-redundant leaves by taking a subgraph, is to create a leaf from a redundant node, but $L(H)\subseteq L(A)=L(B)$, so each degree-1 node of $H$ corresponds to a degree-1 node of $A$ and of $B$.
This means each non-redundant node of $H$ corresponds to a non-redudant node of $B$, and each edge of $H$ to a path between such nodes in $B$, and there is an s-isomorphism of $H$ with the subgraph of $B$ formed by these nodes and edges.
\end{proof}

Now we turn to the definition of an agreement forest, which, as mentioned earlier, characterizes the rSPR distance. Following the definition of the agreement forest, we define a tool similar to an agreement forest tailored to upside down versions of trees. This upside down agreement forest (udAF) can be turned into an agreement forest of the two original trees.

\begin{definition}
Let $T_1$ and $T_2$ be phylogenetic trees with labels $X$ and root $\rho$. Then a partition $\mathcal{P}=\{P_i\}$ of $X\cup\{\rho\}$ is an \emph{agreement forest (AF)} for $T_1$ and $T_2$ if the following hold:
\begin{itemize}
\item $T_1|_{P_i}\equiv T_2|_{P_i}$ for all $i$;
\item $T_t|_{P_i}$ and $T_t|_{P_j}$ are node-disjoint for all pairs $i,j$ with $i\neq j$ and fixed $t\in\{1,2\}$.
\end{itemize}
\end{definition}

\begin{definition}
Let $\ud{T}$ be the upside down version of the phylogenetic tree $T$ with label set $X$. Then an \emph{upside down agreement forest (udAF)} for $\ud{T}$ is a directed graph $F$ such that:
\begin{itemize}
\item $F$ is an undirected forest;
\item $F$ is a leaf-labelled graph with label set $\{e_{x,i}:x\in X, i\in[2n]\}\cup\{\rho\}$, where each label appears at most once;
\item $F\equiv S$ for some subgraph $S$ of the bottom part of $\ud{T}$. 
\end{itemize}
Note that the third requirement implies the first.
\end{definition}

\begin{lemma}\label{lem:udAFtoAF}
Let $T$ and $T'$ be phylogenetic trees with label set $X$. If $F$ is an udAF for $\ud{T}$ and for $\ud{T}'$, then there exists an AF of $T$ and $T'$ of size at most $|F|$, where $|F|$ denotes the number of components of $F$.
\end{lemma}
\begin{proof}
Let $\mathcal{K}$ be the set of components of $F$. For each $K\in \mathcal{K}$ we define the following part of the agreement forest:
\[K_{AF}:=\left\{x\in X | e_{x,i}\in K~\forall i\in[1,2n] \right\}\cup (K\cap\{\rho\}),\]
where $e_{x,i}$ indicates the $i$-th leaf of $\ud{T}$ corresponding to $x$. The agreement forest consists of these parts (ignoring the empty ones, resulting from components that have no complete sets of leaves), together with one part for each leaf that is in none of these parts, i.e.
\[AF:=\big\{Y\subseteq X\cup\{\rho\}| \exists K\in \mathcal{K}~\mathrm{ s.t. }~Y=K_{AF}\big\}\cup\big\{\{x\}\subset X\cup\{\rho\}| \forall K\in \mathcal{K}: x\not\in K_{AF}\big\}\setminus\big\{\emptyset\big\}.\]
Note that each component $Y$ of $AF$ corresponds either uniquely to a component $K$ of $F$ which has all $e_{x,i}$ for some leaf $x$, or it corresponds to a leaf $x$ for which not all $e_{x,i}$ are contained in one component of $F$. In the last case, there is a component of $F$ consisting of one leaf $e_{x,i}$ for some $i$. Note that this correspondence $AF\to \mathcal{K}$ must therefore be injective, and $AF$ has size at most $|F|$. What rests to prove is that $AF$ is indeed an agreement forest for $T$ and $T'$.\\
\\
Let $F'$ be the subgraph of $F$ where each component $K$ is restricted to the subgraph consisting of all paths between the leaves in $K_{AF}$. As (per definition of an udAF) $F$ can be embedded in the bottom part of $\ud{T}$, $F'$ can also be embedded in the bottom part of $\ud{T}$ as it is a subgraph of $F$ (Lemma~\ref{lem:subgraphEmbedding}). This embedding must be unique, because it is of a labelled forest into a labelled tree.

Let $E_x$ be the subgraph of $\ud{T}$ induced by the leaves $e_{x,i}$ and their parents for all $i$ and a fixed $x$. Now replace each subgraph $E_x$ with one leaf $x$ in both $F'$ and in the bottom part of $\ud{T}$. Let the resulting graphs be $F^s$ and $B^s$. Subsequently reverse the direction of each arc in both $B^s$ and in $F^s$ with resulting graphs $B^r$ and $F^r$. Note that the resulting graphs are $B^r=T$ and the union $F^r=\cup_{K\in\mathcal{K}}T|_{K_{AF}}$, and all the restricted trees $T|_{K_{AF}}$ are node disjoint.

We repeat this argument for $\ud{T}'$, and note that the modifications from $F$ to $F^r$ are independent of $\ud{T}$, so we have the equality $F^r=\cup_{K\in\mathcal{K}}T'|_{K_{AF}}$, where the parts $T'|_{K_{AF}}$ are again node disjoint. This means $T|_{K_{AF}}\equiv T'|_{K_{AF}}$ for each $K\in AF$ corresponding to a non-trivial component of $F$, and $T|_{P_i}$ and $T|_{P_j}$ are node disjoint for all nontrivial parts $P_i$ and $P_j$ of $AF$ (similarly for $T'$). Hence so far the elements of $AF$ corresponding to non-trivial components of $F$, meet all the requirements of an AF. 

The only other elements of $AF$ contain only one label, each of which is not in any of the non-trivial components of $AF$. Hence, for any such label $x$, the restriction $T|_{\{x\}}$ consists of only the node labelled $x$, which is not contained in any other component by definition (and similarly for $T'$). Furthermore, the s-isomorphism $T|_{\{x\}}\equiv T'|_{\{x\}}$ is trivial. Hence, $AF$ is indeed an agreement forest.
\end{proof}

The preceding lemma shows that an udAF for two upside down trees gives an AF for the original trees of the same size. Now we still need a connection between the number of head moves and an udAF. The following lemma shows that appropriate head move sequences correspond to udAFs of size related to the number of head moves. 

\begin{lemma}\label{lem:RemovingGivesAF}
Let $T$ and $T'$ be trees with label set $X$, and $|X|=n$. Suppose $S$ is a sequence of head moves $\ud{T}=N_0,\ldots,N_{|S|}=\ud{T}'$ of length $|S|< 2n$. Then there is an udAF $F$ of $\ud{T}$ and $\ud{T}'$ with at most $|S|+1$ components.
\end{lemma}
\begin{proof}
Let $B$ be the bottom part of $\ud{T}$. We prove this result using induction on the number of moves to prove that there exist subgraphs $F_i$ of $N_{i}$ which can be embedded in the bottom part of $\ud{T}$ and have $|F_i|\leq i$ components. Finally we prove the subgraph $F_{|S|}$ of $N_{|S|}=\ud{T}'$ must actually be a subgraph of the bottom part of $\ud{T}'$.

As a base of the induction, set $F_0=B$, which is connected and can clearly be embedded in itself and is a subgraph of $\ud{T}$.  

Now suppose we have subgraphs $F_i$ of $N_i$ with embeddings of $F_i$ in $B$ and $|F_i|\leq i$ for all $i<j\leq|S|$. We prove that there also exists a subgraph $F_j$ of $N_j$ with at most $j$ components that can be embedded in $B$. 

Note that $F_{j-1}$ is a subgraph of $N_{j-1}$ and therefore the moving edge $e_j=(u,v)$ can be either an edge of $F_{j-1}$, or it is in the complement $N_{j-1}\setminus F_{j-1}$. In the last case $e_j$ can have only its endpoints in $F_{j-1}$. Now construct $F_j$ as follows:
\begin{itemize}
\item remove edge $e_j=(u,v)$ from $F_{j-1}$ if it was contained in it;
\item clean up the resulting graph by removing all edges not contained in any undirected path between two leaves, and suppressing $v$ if it is a degree $2$ vertex after removal of $(u,v)$.
\item add the new endpoint if necessary. That is: let the target edge of the move be $t$, if $t$ is contained in the graph after cleaning up, subdivide $t$.
\end{itemize}
Note that $F_{j}$ can be embedded in $F_{j-1}$ because the only operations were: restriction to a subset of labels, subdivision, and suppression (Lemma~\ref{lem:subgraphEmbedding}). 

Because $F_{j-1}$ embeds in $B$, there is also an embedding of $F_j$ into $B$. Furthermore, $F_j$ is a subgraph of $N_j$ by construction: the three steps correspond exactly to the three steps of a head move in $N_{j-1}$. Lastly, $F_j$ has at most one more component than $F_{j-1}$, because the only operation that can increase the number of components is the removal of the edge in the first step, and because that is an edge removal in a graph, it creates at most one extra component.

We conclude that the desired subgraphs $F_{i}$ of $N_i$ exist for all $i\in [|S|]$.\\
\\
Note that we have not yet proven that $F:=F_{|S|}$ is an udAF for $\ud{T}'$, as $F$ might not embed in the bottom part of $\ud{T}'$. We now prove that $F$ is in fact a subgraph of the bottom part of $\ud{T}'$.

By construction, $F$ is a directed subgraph of $\ud{T}'$. Suppose (for a contradiction) that $F$ is not a subgraph of the bottom part of $\ud{T}'$, i.e. some part of $F$ lies partly in the top part of $\ud{T}'$. This means that there is a node $t$ of $F$ that corresponds to a split node (which we also call $t$) in the upper part of $\ud{T}'$. A split node of $F$ necessarily has two children $c_1$ and $c_2$, as $F$ embeds in the bottom part of $\ud{T}$. One of these children (w.l.o.g. $c_1$) must have a unique leaf descendant $e_{x_,}$. The other child ($c_2$) either has a leaf descendant $e_{x,j}$ (same $x$), or the next non-redundant descendant is a reticulation node.

If $c_2$ has a leaf descendant $e_{x,j}$, we note the following: $t$ is mapped to a split node in the top part of $\ud{T}'$. Hence the leaves below the one child of $t$ and the leaves below the other child of $t$ can never correspond to the same $x\in X$: indeed if $e_{y,i}$ is below $c_1$, then $e_{y,j}$ is also below $c_1$, and similarly for $c_2$; furthermore, as the only reticulations of $\ud{T}'$ are in the lower part of the network after the $e_{\cdot,\cdot}$ split off, the leaves below $c_1$ and $c_2$ are disjoint (except for the leaf corresponding to the root of $T'$). Hence, as $c_1$ has a descendant $e_{x,i}$, and $c_2$ has a descendant $e_{x,j}$, we have a contradiction.

Now if $c_2$'s first non-redundant descendant $d$ is a reticulation node, then this node maps to a reticulation in the bottom part of $\ud{T}'$. This means the path then the edge $(t,d)$ maps to a path from the top part of $\ud{T}'$ to a reticulation in the bottom part of $\ud{T}'$. Such a path must necessarily contain the all parents of the leaves $e_{x,i}$ for some $x\in X$. As the embeddings of all components in $F$ are node disjoint, and each leaf $e_{\cdot,\cdot}$ is a node of $F$, each leaf $e_{x,i}$ (fixed~$x$, for all $i\in[2n]$) is its own component in $F$. Hence $F$ has at least $2n+1$ components, implying that $|S|\geq 2n$, which gives us a contradiction with the assumptions of the lemma.\\
\\
Hence there does exist an embedding of $F$ in the bottom part of $\ud{T}$, and $F$ is an udAF for $\ud{T}'$ with at most $|S|+1$ components, as $F$ has at most $|S|+1$ components.
\end{proof}

Finally we put everything together in the following lemma and theorem: each candidate head move sequence defines an udAF, which in turn gives an AF for the original trees, which bounds the rSPR distance between these trees.

\begin{lemma}
Let $T_1$ and $T_2$ be trees with a common 
label set, then 
\[d_{\text{rSPR}}(T_1,T_2)=d_{\text{Head}}(\ud{T}_1,\ud{T}_2).\]
\end{lemma}
\begin{proof}
The inequality $d_{\text{rSPR}}(T_1,T_2)\geq d_{\text{Head}}(\ud{T}_1,\ud{T}_2)$ is obvious, as the rSPR sequence for the trees directly translates into a head move sequence for the upside down trees.

We now prove the other inequality. As $2n>d_{\text{rSPR}}(T_1,T_2)\geq d_{\text{Head}}(\ud{T}_1,\ud{T}_2)$ we only have to consider head move sequences of length at most $2n$. Suppose we have a sequence of head moves $S$ between $\ud{T}_1$ and $\ud{T}_2$ of length at most $2n$, then there exists a udAF of size at most $|S|+1$ for $\ud{T}_1$ and $\ud{T}_2$ (Lemma~\ref{lem:RemovingGivesAF}). Now Lemma~\ref{lem:udAFtoAF} tells us that there is an AF for $T$ and $T'$ of size at most $|S|+1$. Using the fact that the size of the MAF of $T$ and $T'$ minus one is equal to the rSPR distance between $T$ and $T'$ \citep{bordewich2005computational}, we get the following inequalities:
\[|S| \geq        |AF|-1        \geq         d_{\text{rSPR}}(T_1,T_2).\]
We conclude that $d_{\text{rSPR}}(T_1,T_2)=d_{\text{Head}}(\ud{T}_1,\ud{T}_2)$.
\end{proof}

\begin{theorem}
Computing the head move distance between two networks is NP-hard.
\end{theorem}
\begin{proof}
Direct consequence of the previous lemma, as computing the rSPR distance between two trees is NP-hard \citep{bordewich2005computational}.
\end{proof}

Note that the theorem above does not tell us whether it is hard to find the distance between networks of a fixed tier: increasing the size of the input for our construction corresponds to increasing the reticulation number.

\section{Discussion}
When generalizing rSPR moves on rooted trees to rooted networks, it is natural to consider tail moves, because each rSPR move in a tree is a tail move. However, when taking the view that an rSPR move is a move that changes one of the endpoints of an arc, head moves also belong to the generalization of rSPR moves \citep{gambette2017rearrangement}. In this view, it is equally natural to only consider head moves, as to only consider tail moves.\\
\\
We have showed that head moves are sufficient to connect all tiers of phylogenetic network space except tier-0. This might be surprising because head moves are relatively limited compared to tail moves: the head move neighbourhood is relatively small when compared to the tail move neighbourhood. On the other hand, when one reverses all the arcs of a network, each tail move becomes a head move. This makes the difference between connectivity results for these types of moves just a mathematical difference in numbers of roots, reticulations, and leaves, instead of a fundamental difference in biological interpretation. 

To unify these connectivity results, one could consider head or tail moves in a broader class of networks, which may have multiple roots and at least one leaf (instead of at least two) (Figure~\ref{fig:MultiRooted}). For such multi-rooted networks, connectivity results for head moves and for tail moves could easily be related. This reason for studying multi-rooted networks is mathematically inspired. 

\begin{figure}[h!]
\begin{center}
\includegraphics[scale=0.5]{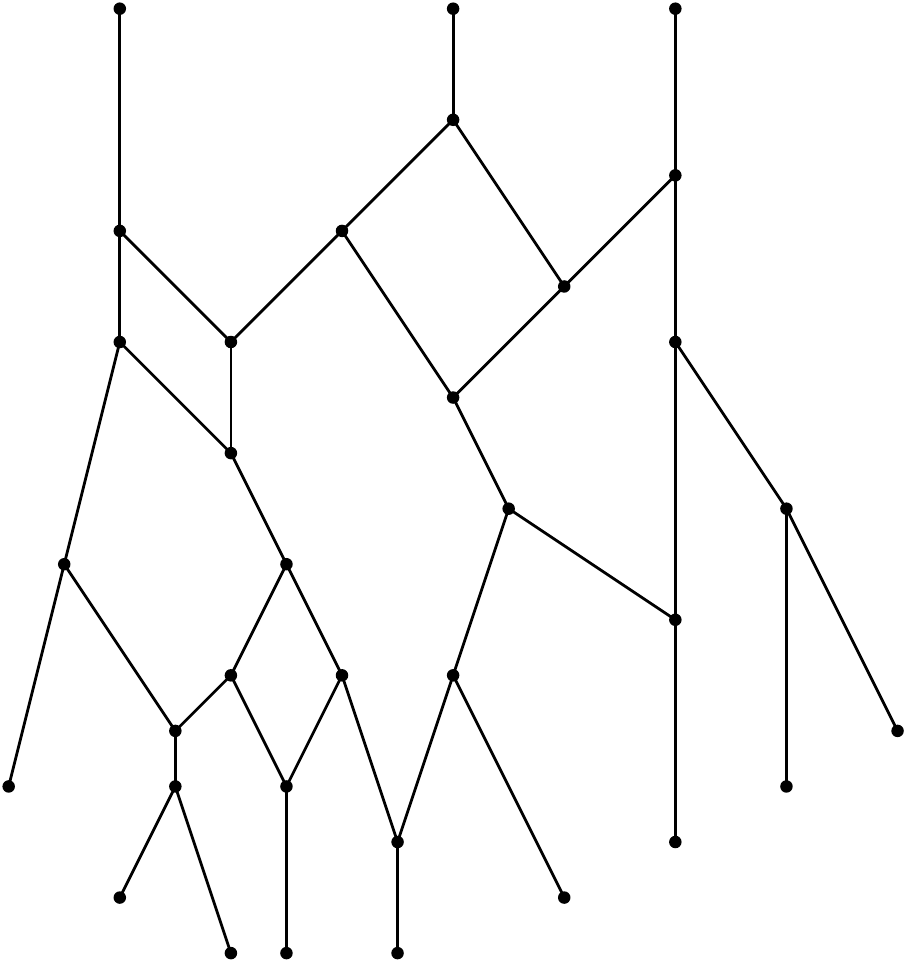}
\end{center}
\caption{A multi-rooted network, labels of leaves have been omitted.}
\label{fig:MultiRooted}
\end{figure}

Another reason to study multi-rooted networks is more inspired by biology: these networks could be interesting on their own as subnetworks of ordinary phylogenetic networks. The advantage is that one does not have to make assumptions about how these roots are connected higher up, that is, about the evolutionary history before the existence of these root genes/species \citep{haggerty2013pluralistic}. Additionally, a famous but slightly dated view of the evolutionary history is the net of life by \citet{doolittle1999phylogenetic}, which features multiple roots. A third reason becomes apparent when we take a broader view of phylogenetic networks that includes pedigrees: these often start with multiple individuals that may coalesce quite long ago.\\
\\
While we focussed mostly on head and tail moves of any distance, we have proven the connectivity of tiers of phylogenetic network space by distance-2 head moves. Distance-1 head moves are not sufficient in general because heads cannot move past their own tails. It would be interesting to see which networks are actually connected by distance-1 head moves.

\begin{figure}[h!]
\begin{center}
\includegraphics[scale=0.9]{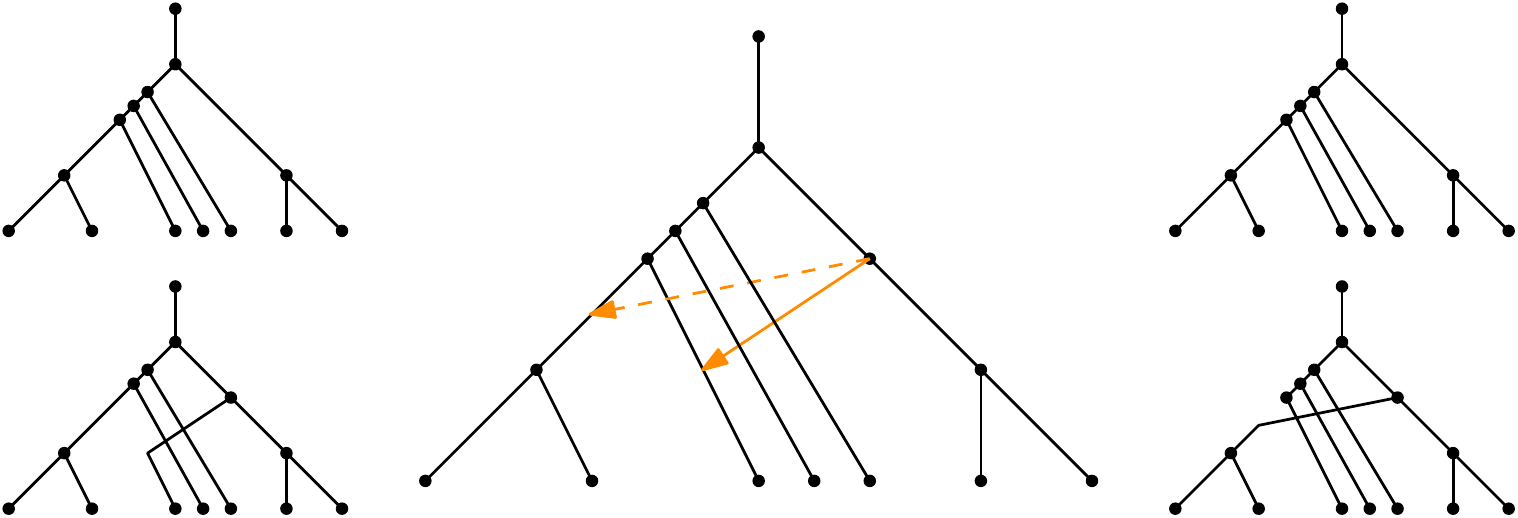}
\end{center}
\caption{A distance-2 head move in a network and the displayed trees before (left) and after (right) the head move. The top displayed tree is the same before and after the head move. The bottom tree before disappears, and is replaced by the tree bottom-right, which is a distance-3 tail move away from the top tree.}
\label{fig:ShortToLongDistance}
\end{figure}

It is unclear if this connectivity result for distance-2 moves is useful, especially in the context of embedded gene trees their relevance may be disputed. After all, short distance head moves do not generally correspond to short distance (tail) moves in the displayed trees (Figure~\ref{fig:ShortToLongDistance}). Problems studying the relation between displayed trees, which are interpreted as gene trees, and phylogenetic networks are often quite hard \citep{bordewich2007computing,van2012quadratic,van2016hybridization}. Hence strategies for solving these problems could benefit from local search heuristics.\\
\\
As mentioned, an important motivation for studying rearrangement moves is their possible use in local search strategies for phylogenetic networks. As such, it is important to understand the topological/geometric properties of the tiers of phylogenetic networks space. In this paper we started this study by giving bounds on the head move diameters, and finding additional connections between head move distance, tail move distance, and rSPR distance. 

Although the bounds for the head move diameter we found are already quite good -- both upper bound, and lower bound are linear in the number of leaves and the reticulation number, just like for tail moves and for rSPR moves \citep{janssen2017exploring} -- they could possibly be improved. 

As future research, one could try to discover the exact diameters. Another direction would be to try to apply our techniques for bounding the diameter to other types of moves, such as SNPR and PR moves \citep{bordewich2017lost,klawitter2018agreement}. Because these classes of moves also include vertical moves, this might be quite challenging.\\
\\
The other property of phylogenetic network space defined by head moves we touched on was the neighbourhood size. The head move neighbourhood is relatively small. Nevertheless, it is still possible to reach any network quite quickly, as the diameter still grows linearly. This means head moves might be very well suited for local search heuristics. 

Of course there are other factors to consider. Maybe head moves do not have the proper relation to the studied phylogenetic objectives. For example they might have unnecessarily irregular optimization landscapes, as an example in phylogenetic tree space: NNI moves give local optima (not globally optimal) for maximum parsimony, whereas SPR moves only give a global optimum for perfect sequence data \citep{urheim2016characterizing}. It would be interesting to analyse such relations for networks, too. For example by studying the occurrence of local optima for different kinds of parsimony \citep{fischer2015computing,van2017improved} using the existing types of rearrangement moves.

Another possible complicating factor in the relation between head moves and the optimization objective could be that head moves might be too restrictive for some types of networks. Indeed, we have not studied head moves for subclasses of networks. It might be useful to see if head moves also connect tiers of tree-child networks for example. Such questions have been answered for other moves \citep{bordewich2017lost}\\
\\
Lastly, in this paper we have studied the problem of computing the head move distance between two networks. For tail moves, rSPR moves \citep{janssen2017exploring} and for SNPR moves \citep{klawitter2018distance}, it was already known that computing the distance between two networks is NP-hard. For the first two of these, we additionally know that computation of distances is hard for each tier. Here we have shown that computing head move distance is also NP-hard, although we have not shown this for each tier separately. A first step in proving hardness in each tier might be to study head move distance computation in tier-1, which seems intriguingly simple. 

It could also be interesting to find an efficient algorithm for the task of finding a shortest head move sequence, or to characterize the exact distance between two phylogenetic networks in a more abstract way. No efficient (FPT) algorithm for this task is known, nor are there any exact characterizations of distances between networks given by rearrangement moves. A first attempt was recently made using a generalization of agreement forests, this approach currently only yields exact distances between trees and networks, and no exact distances between two networks \citep{klawitter2018agreement,klawitter2018distance}.

\noindent\small{{\bf Acknowledgements} Thanks to Leo van Iersel, Yukihiro Murakami, and Mark Jones for their valuable input and especially Leo van Iersel and Yukihiro Murakami for kindly reading my manuscript.}

\DeclareRobustCommand{\VAN}[3]{#3}
\bibliographystyle{chicago}
\bibliography{bibliography}

\begin{thebibliography}{}

\bibitem[\protect\citeauthoryear{Atas, Tuncbag, and Do{\u{g}}an}{Atas
  et~al.}{2018}]{atas2018phylogenetic}
Atas, H., N.~Tuncbag, and T.~Do{\u{g}}an (2018).
\newblock Phylogenetic and other conservation-based approaches to predict
  protein functional sites.
\newblock In {\em Computational Drug Discovery and Design}, pp.\  51--69.
  Springer.

\bibitem[\protect\citeauthoryear{Bordewich, Linz, and Semple}{Bordewich
  et~al.}{2017}]{bordewich2017lost}
Bordewich, M., S.~Linz, and C.~Semple (2017).
\newblock Lost in space? {Generalising} subtree prune and regraft to spaces of
  phylogenetic networks.
\newblock {\em Journal of theoretical biology\/}~{\em 423}, 1--12.

\bibitem[\protect\citeauthoryear{Bordewich and Semple}{Bordewich and
  Semple}{2005}]{bordewich2005computational}
Bordewich, M. and C.~Semple (2005).
\newblock On the computational complexity of the rooted subtree prune and
  regraft distance.
\newblock {\em Annals of combinatorics\/}~{\em 8\/}(4), 409--423.

\bibitem[\protect\citeauthoryear{Bordewich and Semple}{Bordewich and
  Semple}{2007}]{bordewich2007computing}
Bordewich, M. and C.~Semple (2007).
\newblock Computing the minimum number of hybridization events for a consistent
  evolutionary history.
\newblock {\em Discrete Applied Mathematics\/}~{\em 155\/}(8), 914--928.

\bibitem[\protect\citeauthoryear{Doolittle}{Doolittle}{1999}]{doolittle1999phylogenetic}
Doolittle, W.~F. (1999).
\newblock Phylogenetic classification and the universal tree.
\newblock {\em Science\/}~{\em 284\/}(5423), 2124--2128.

\bibitem[\protect\citeauthoryear{Felsenstein}{Felsenstein}{2004}]{felsenstein2004inferring}
Felsenstein, J. (2004).
\newblock {\em Inferring Phylogenies}, Volume~2.
\newblock Sinauer associates Sunderland, MA.

\bibitem[\protect\citeauthoryear{Fischer, van Iersel, Kelk, and
  Scornavacca}{Fischer et~al.}{2015}]{fischer2015computing}
Fischer, M., L.~van Iersel, S.~Kelk, and C.~Scornavacca (2015).
\newblock On computing the maximum parsimony score of a phylogenetic network.
\newblock {\em SIAM Journal on Discrete Mathematics\/}~{\em 29\/}(1), 559--585.

\bibitem[\protect\citeauthoryear{Francis, Huber, Moulton, and Wu}{Francis
  et~al.}{2018}]{francis2017bounds}
Francis, A., K.~T. Huber, V.~Moulton, and T.~Wu (2018).
\newblock Bounds for phylogenetic network space metrics.
\newblock {\em Journal of mathematical biology\/}~{\em 76\/}(5), 1229--1248.

\bibitem[\protect\citeauthoryear{Gambette, van Iersel, Jones, Lafond, Pardi,
  and Scornavacca}{Gambette et~al.}{2017}]{gambette2017rearrangement}
Gambette, P., L.~van Iersel, M.~Jones, M.~Lafond, F.~Pardi, and C.~Scornavacca
  (2017).
\newblock Rearrangement moves on rooted phylogenetic networks.
\newblock {\em PLoS computational biology\/}~{\em 13\/}(8), e1005611.

\bibitem[\protect\citeauthoryear{Gao, Bailes, Robertson, Chen, Rodenburg,
  Michael, Cummins, Arthur, Peeters, Shaw, et~al.}{Gao
  et~al.}{1999}]{gao1999origin}
Gao, F., E.~Bailes, D.~L. Robertson, Y.~Chen, C.~M. Rodenburg, S.~F. Michael,
  L.~B. Cummins, L.~O. Arthur, M.~Peeters, G.~M. Shaw, et~al. (1999).
\newblock Origin of {HIV-1} in the chimpanzee pan troglodytes troglodytes.
\newblock {\em Nature\/}~{\em 397\/}(6718), 436.

\bibitem[\protect\citeauthoryear{Guyeux, Al-Nuaimi, AlKindy, Couchot, and
  Salomon}{Guyeux et~al.}{2017}]{guyeux2017ability}
Guyeux, C., B.~Al-Nuaimi, B.~AlKindy, J.-F. Couchot, and M.~Salomon (2017).
\newblock On the ability to reconstruct ancestral genomes from mycobacterium
  genus.
\newblock In {\em International Conference on Bioinformatics and Biomedical
  Engineering}, pp.\  642--658. Springer.

\bibitem[\protect\citeauthoryear{Haggerty, Jachiet, Hanage, Fitzpatrick, Lopez,
  O'Connell, Pisani, Wilkinson, Bapteste, and McInerney}{Haggerty
  et~al.}{2013}]{haggerty2013pluralistic}
Haggerty, L.~S., P.-A. Jachiet, W.~P. Hanage, D.~A. Fitzpatrick, P.~Lopez,
  M.~J. O'Connell, D.~Pisani, M.~Wilkinson, E.~Bapteste, and J.~O. McInerney
  (2013).
\newblock A pluralistic account of homology: adapting the models to the data.
\newblock {\em Molecular biology and evolution\/}~{\em 31\/}(3), 501--516.

\bibitem[\protect\citeauthoryear{Huber, Moulton, and Wu}{Huber
  et~al.}{2016}]{huber2016transforming}
Huber, K.~T., V.~Moulton, and T.~Wu (2016).
\newblock Transforming phylogenetic networks: Moving beyond tree space.
\newblock {\em Journal of theoretical biology\/}~{\em 404}, 30--39.

\bibitem[\protect\citeauthoryear{{\VAN{Iersel}{Van}{van}}~Iersel, Janssen,
  Jones, Murakami, and Zeh}{{\VAN{Iersel}{Van}{van}}~Iersel
  et~al.}{2018}]{van2018polynomial}
{\VAN{Iersel}{Van}{van}}~Iersel, L., R.~Janssen, M.~Jones, Y.~Murakami, and
  N.~Zeh (2018).
\newblock Polynomial-time algorithms for phylogenetic inference problems.
\newblock In {\em International Conference on Algorithms for Computational
  Biology}, pp.\  LNBI 10849: 37--49. Springer.

\bibitem[\protect\citeauthoryear{{\VAN{Iersel}{Van}{van}}~Iersel, Jones, and
  Scornavacca}{{\VAN{Iersel}{Van}{van}}~Iersel et~al.}{2017}]{van2017improved}
{\VAN{Iersel}{Van}{van}}~Iersel, L., M.~Jones, and C.~Scornavacca (2017).
\newblock Improved maximum parsimony models for phylogenetic networks.
\newblock {\em Systematic biology\/}~{\em 67\/}(3), 518--542.

\bibitem[\protect\citeauthoryear{{\VAN{Iersel}{Van}{van}}~Iersel, Kelk, Lekic,
  Whidden, and Zeh}{{\VAN{Iersel}{Van}{van}}~Iersel
  et~al.}{2016}]{van2016hybridization}
{\VAN{Iersel}{Van}{van}}~Iersel, L., S.~Kelk, N.~Lekic, C.~Whidden, and N.~Zeh
  (2016).
\newblock Hybridization number on three rooted binary trees is {EPT}.
\newblock {\em SIAM Journal on Discrete Mathematics\/}~{\em 30\/}(3),
  1607--1631.

\bibitem[\protect\citeauthoryear{{\VAN{Iersel}{Van}{van}}~Iersel and
  Linz}{{\VAN{Iersel}{Van}{van}}~Iersel and Linz}{2013}]{van2012quadratic}
{\VAN{Iersel}{Van}{van}}~Iersel, L. and S.~Linz (2013).
\newblock A quadratic kernel for computing the hybridization number of multiple
  trees.
\newblock {\em Information Processing Letters\/}~{\em 113\/}(9), 318--323.

\bibitem[\protect\citeauthoryear{Janssen, Jones, Erd{\H{o}}s, van Iersel, and
  Scornavacca}{Janssen et~al.}{2018}]{janssen2017exploring}
Janssen, R., M.~Jones, P.~L. Erd{\H{o}}s, L.~van Iersel, and C.~Scornavacca
  (2018).
\newblock Exploring the tiers of rooted phylogenetic network space using tail
  moves.
\newblock {\em Bulletin of mathematical biology\/}~{\em 80\/}(8), 2177--2208.

\bibitem[\protect\citeauthoryear{Joy, Liang, McCloskey, Nguyen, and Poon}{Joy
  et~al.}{2016}]{joy2016ancestral}
Joy, J.~B., R.~H. Liang, R.~M. McCloskey, T.~Nguyen, and A.~F. Poon (2016).
\newblock Ancestral reconstruction.
\newblock {\em PLoS computational biology\/}~{\em 12\/}(7), e1004763.

\bibitem[\protect\citeauthoryear{Klawitter}{Klawitter}{2017}]{klawitter2017snpr}
Klawitter, J. (2017).
\newblock The {SNPR} neighbourhood of tree-child networks.
\newblock {\em ArXiv e-prints\/}~(1707.09579).

\bibitem[\protect\citeauthoryear{{Klawitter}}{{Klawitter}}{2018}]{klawitter2018agreement}
{Klawitter}, J. (2018, June).
\newblock {The agreement distance of rooted phylogenetic networks}.
\newblock {\em ArXiv e-prints\/}~(1806.05800).

\bibitem[\protect\citeauthoryear{{Klawitter} and {Linz}}{{Klawitter} and
  {Linz}}{2018}]{klawitter2018distance}
{Klawitter}, J. and S.~{Linz} (2018, May).
\newblock {On the Subnet Prune and Regraft Distance}.
\newblock {\em ArXiv e-prints\/}~(1805.07839).

\bibitem[\protect\citeauthoryear{Lakner, Van Der~Mark, Huelsenbeck, Larget, and
  Ronquist}{Lakner et~al.}{2008}]{lakner2008efficiency}
Lakner, C., P.~Van Der~Mark, J.~P. Huelsenbeck, B.~Larget, and F.~Ronquist
  (2008).
\newblock Efficiency of markov chain monte carlo tree proposals in bayesian
  phylogenetics.
\newblock {\em Systematic biology\/}~{\em 57\/}(1), 86--103.

\bibitem[\protect\citeauthoryear{Lessler, Chaisson, Kucirka, Bi, Grantz, Salje,
  Carcelen, Ott, Sheffield, Ferguson, et~al.}{Lessler
  et~al.}{2016}]{lessler2016assessing}
Lessler, J., L.~H. Chaisson, L.~M. Kucirka, Q.~Bi, K.~Grantz, H.~Salje, A.~C.
  Carcelen, C.~T. Ott, J.~S. Sheffield, N.~M. Ferguson, et~al. (2016).
\newblock Assessing the global threat from zika virus.
\newblock {\em Science\/}~{\em 353\/}(6300), aaf8160.

\bibitem[\protect\citeauthoryear{Nguyen, Schmidt, von Haeseler, and
  Minh}{Nguyen et~al.}{2014}]{nguyen2014iq}
Nguyen, L.-T., H.~A. Schmidt, A.~von Haeseler, and B.~Q. Minh (2014).
\newblock Iq-tree: a fast and effective stochastic algorithm for estimating
  maximum-likelihood phylogenies.
\newblock {\em Molecular biology and evolution\/}~{\em 32\/}(1), 268--274.

\bibitem[\protect\citeauthoryear{Semple and Steel}{Semple and
  Steel}{2003}]{semple2003phylogenetics}
Semple, C. and M.~A. Steel (2003).
\newblock {\em Phylogenetics}, Volume~24.
\newblock Oxford University Press on Demand.

\bibitem[\protect\citeauthoryear{Shindyalov, Kolchanov, and Sander}{Shindyalov
  et~al.}{1994}]{shindyalov1994can}
Shindyalov, I., N.~Kolchanov, and C.~Sander (1994).
\newblock Can three-dimensional contacts in protein structures be predicted by
  analysis of correlated mutations?
\newblock {\em Protein Engineering, Design and Selection\/}~{\em 7\/}(3),
  349--358.

\bibitem[\protect\citeauthoryear{Urheim, Ford, and John}{Urheim
  et~al.}{2016}]{urheim2016characterizing}
Urheim, E., E.~Ford, and K.~S. John (2016).
\newblock Characterizing local optima for maximum parsimony.
\newblock {\em Bulletin of mathematical biology\/}~{\em 78\/}(5), 1058--1075.

\bibitem[\protect\citeauthoryear{Yu, Harris, Blair, and He}{Yu
  et~al.}{2015}]{yu2015rasp}
Yu, Y., A.~J. Harris, C.~Blair, and X.~He (2015).
\newblock Rasp (reconstruct ancestral state in phylogenies): a tool for
  historical biogeography.
\newblock {\em Molecular phylogenetics and evolution\/}~{\em 87}, 46--49.

\end{thebibliography}


\end{document}